\documentclass[notitlepage,11pt]{article}
\usepackage{amsfonts}
\usepackage{amssymb}
\usepackage{amsmath}
\usepackage{harvard}
\usepackage{lscape}
\usepackage{graphicx}
\usepackage{latexsym,epsfig,amsmath}
\usepackage{epsfig,graphicx}
\usepackage{tabularx}
\usepackage{rotating}
\usepackage{hyperref}

\newtheorem{theorem}{Theorem}

\newtheorem{definition}{Definition}
\newcounter{examplec}
\newtheorem{example}[examplec]{Example}

\newcounter{lemmac}
\newtheorem{lemma}[lemmac]{Lemma}

\newcounter{propos}
\newtheorem{proposition}[propos]{Proposition}

\newcounter{assumec}
\newtheorem{assume}[assumec]{Assumption}
\renewcommand{\theassume}{\Alph{assume}}

\newenvironment{proof}[1][Proof]{\textbf{#1.} }{\ \rule{0.5em}{0.5em}}
\evensidemargin\oddsidemargin
\input{tcilatex}
\begin{document}

\title{Limit Theorems for Data with Network Structure\thanks{%
This paper was prepared for the Econometric Theory Lecture at the
Australasian Meetings of the Econometric Society in Auckland, New Zealand in
July 2018. A preliminary version was presented at NYU and Copenhagen. I
thank Peter C.B. Phillips, Quang Vuong and seminar participants at Cambridge
(UK), NYU, University College London and the Conference of the 30th year
Anniversary Reunion of the Review of Economic Studies Tour for helpful
comments. Financial support from the National Institute of Health through
SBIR grant 1 R43 AG056199-01 is gratefully acknowledged. }}
\author{Guido M. Kuersteiner\thanks{%
Department of Economics, University of Maryland, College Park, MD 20742,
e-mail: kuersteiner@econ.umd.edu}}
\maketitle

\begin{abstract}
This paper develops new limit theory for data that are generated by networks
or more generally display cross-sectional dependence structures that are
governed by observable and unobservable characteristics. Strategic network
formation models are an example. Wether two data points are highly
correlated or not depends on draws from underlying characteristics
distributions. The paper defines a measure of closeness that depends on
primitive conditions on the distribution of observable characteristics as
well as functional form of the underlying model. A summability condition
over the probability distribution of observable characteristics is shown to
be a critical ingredient in establishing limit results. The paper
establishes weak and strong laws of large numbers as well as a stable
central limit theorem for a class of statistics that include as special
cases network statistics such as average node degrees or average peer
characteristics. Some worked examples illustrating the theory are provided.
\end{abstract}

\QTP{Body Math}
\newpage

\section{Introduction}

There is growing interest in the economics literature in models that
represent strategic and non-strategic interactions between individuals.
Examples are the peer effects literature which has seen many applications in
the areas of microeconomics including in the fields of labor and
development. A related literature considers models of strategic network
formation. Strategic games are another area where models focus on the
interaction between individuals. A direct implication of all these models is
that the random sampling assumption underlying much of statistical theory
and related asymptotic approximations is not a good paradigm. This paper
aims to extend the available tool kit for the analysis of these
interdependent data structures by developing new measures of cross-sectional
dependence and by utilizing these measures to establish weak and strong laws
of large numbers as well as a stable central limit theorem.

There is a large and well established literature in probability theory that
analyses random graphs, a special class of network models dating back to Erd%
\H{o}s and R\'{e}nyi (1959). Random Geometric Graphs are extensively
analyzed in Penrose (2003) and constitute a class of models that is closer
to models that are relevant in economics and econometrics. The $\beta $%
-model is another network formation model that has received considerable
attentition in the statistics literature, see for example Holland and
Leinhardt (1981), Park and Newman (2004), Chatterjee, Diaconis and Sly
(2011) and that has recently been extended to applications in econometrics
by Graham (2017). Strategic network formation models have been porposed by
Jackson\ (2008) and been analyzed by Goldsmith-Pinkham and Imbens (2013) and
Menzel (2016) among others. Peer effects models focus on outcomes of
individuals linked by network structures. The links may be simple group
memberships or based on more sophisticated network formation models. Manski
(1993), Brock and Durlauf (2001), Calvo-Armengol, Patacchini and Zenou\
(2009), Graham (2008), Bramoull\'{e}, Djebbari and Fortin (2009),
Goldsmith-Pinkham and Imbens (2013), and de Paula (2016) consider
identification and estimation of these models. Game theoretic models include
Rust (1994), Aguirregabiria and Mira (2007) and Bajari, Benkard and Levin
(2007) to name only a few.

Limit theory related to these models has been developed both in the
probability and statistics literature as well as more recently in the
econometrics literature. Methods that deal with random fields,
generalizations of stochastic processes to multiple indices, include the
early contribution of Bolthausen (1982) who defines mixing coefficients
based on a non-random metric of distance between to points in the index set.
Conley (1999) appears to be the first application of Bolthauses's results in
the econometrics literature. Jenish and Prucha (2009) extend Bolthausen's
results by using sharper moment bounds based on work by Rio (1993). They
also prove uniform laws of large numbers for random fields. Jenish and
Prucha (2012) further builds on this line of work by considering near epoch
dependent random fields based on an underlying mixing random field. Penrose
and Yukish (2001) prove a CLT for functionals of point processes that are
translation invariant, satisfying a scaling property, and that are strongly
stabilizing. The proof is based on a coupling argument. Penrose and Yukish
(2003) and Leung (2016) establish laws of large numbers for functionals of
point processes. The functionals are translation and scale invariant and
stabilizing. These results are based on 'infill' asymptotics. The LLN\ is
obtained through a coupling argument and provides an approximate
representation for the limit functional. Graham (2017) considers models of
undirected dyadic network link formation allowing for homophily and agent
heterogeneity. A tetrad logit estimator conditions on a sufficient statistic
for degree heterogeneity. A CLT for the estimator is established under a
random sampling assumption and exploits conditional independence of the
network formation process, conditional on observed characteristics and fixed
effects. Kuersteiner and Prucha (2013) prove a general cross-sectional CLT
based on restrictions that imply a martingale difference sequence (mds)
property of sample averages. In follow up work Kuersteiner and Prucha (2015)
establish a CLT for linear-quadratic moment conditions in peer effects
models with endogenous network formation. The CLT is based on a spatial
martingale difference structure of the model errors. It depends on high
level conditions regarding the convergence of sample second moments. Lee and
Song (2017) consider random vectors defined on an undirected neighborhood
system. They derive a Berry-Essen bound and a stable functional central
limit theorem under conditional neighborhood independence. Menzel (2016)
develops a law of large numbers and central limit theorem for static
discrete action games with a large number of players under an
exchangeability assumption.

This paper extends and complements the existing literature in various
directions. The results in this paper are based on a new conditional
mixingale type assumption defined in terms of a random metric of distance.
The distance measure is model dependent and may include, in the case of
network models, the conditional probabilties of two nodes forming a link.
The relevant probability is conditional on node characteristics that are
drawn from some joint characteristics distribution. There are no assumptions
that characteristics are indepent. However, a requirement for the limiting
results in this paper is that nodes are sufficiently spread out as measured
by their characteristics so that dependence eventually dies off. This
restriction of sparsity that rules out a buildup of a mass of nodes with
very similar features is captured by a summability condition of the
probabilities that two nodes are close in an appropriate sense. The
summability condition is similar to the Borell-Cantelli lemma.

The two elements that control dependence in this paper are therefore the
mixingale condition which depends on features of the model that determine
how close two nodes or observations are conditional on their observed
characteristics. The second element is the characteristics distribution that
determines how frequently close enough data points are observed in a sample.
By combining these two elements it is possible to give a variance upper
bound for a maximal inequality based on Stout (1974) which then leads to a
weak and strong law of large numbers for a class of network statistics that
satisfy the mixingale assumptions. The proof of the central limit theorem is
based on a combination of techniques found in McLeish (1974) for pure
martingales and a blocking argument due to Eberlein (1984). The fact that
the proof of McLeish (1974) is focused on the martingale, or as adapted to
this paper, approximate martingale property of empirical sums is critical to
being able to avoid more complex mixing conditions. It is expected, and
documented with some examples, that the mixingale conditions proposed here
are easier to establish for specific models than related mixing conditions
would be. While verification of regularity conditions requires specific
models the general theory in this paper is completely non-parametric.
Nevertheless, results for sample averages of network statistics are relevant
for the asymptotic analysis of statistics in parametric settings. Some
examples from the peer effects literature are discussed.

This paper, by relying on the approximate behavior of the conditional mean,
is able to avoid some of the assumptions that are made in the literature,
including conditional independence, exchangeability and limited neighborhood
size. There is also no need to specify a fixed metric of distance relating
specific observations to fixed locations in an index space as is done in
Bolthausen (1982), Conley (1999) and Jenish and Prucha (2009, 2012). Rather,
in this paper, locations and associated distances are draws from some
underlying joint characteristics distribution that is not assumed to be
independent over different nodes or entities. The paper concludes by an
analysis of the regularity conditions in the specific case of a network
formation model that is similar to the model that Graham\ (2016) considered.

\section{Spatial Mixingales\label{Sec_Gen_Limit_Theory}}

The general model is based on network statistics $v$, defined broadly, that
are non-parametric functions of observed and unobserved characteristics. Let 
$\zeta =\left( \zeta _{1},...\right) $ be a collection of observed network
characteristics $\zeta _{i}$ for agent $i$ that affect $i$'s position within
the network, $z=\left( z_{1},...\right) $ a collection of observed
characteristics not necessarily related to network position, $\eta =\left(
\eta _{1},...\right) $ a collection of unobserved characteristics that
affect $i$'s network position and $\epsilon =\left( \epsilon _{ij}\right)
_{i,j=1}^{\infty }$ a collection of link specific idiosyncratic unobserved
shocks that affect the interaction between $i$ and $j$. For a network with $%
n $ agents, also referred to as nodes (see Chandrasekhar 2015) let $%
v_{i,n}\left( \zeta \right) $ be a network statistic for agent $i.$ For
simplicity assume that $v_{i,n}\left( \zeta \right) $ takes values in $%
\mathbb{R}$.

Let $\left( \Omega ,\mathcal{F},P\right) $ be a probability space. Let $%
\mathcal{Z}$ be a sub-sigma field of $\mathcal{F}$ such that $\zeta $ is
measurable with respect to $\mathcal{Z}$. Following Breiman (1992), Theorem
4.34 and Theorem A.46 for fixed $\omega \in \Omega ,$ let $Q_{\omega }\left(
B|\mathcal{Z}\right) $ be a regular conditional distribution given $\mathcal{%
Z}$ and define the conditional probability space $\left( \Omega ,\mathcal{F}%
,Q_{\omega }\right) .\footnote{%
A more detailed construction of the probability space is given in Section %
\ref{Section_ProbabilitySpace}.}$ Let $\chi _{i}=\left( z_{i}^{\prime
},\zeta _{i}^{\prime },\mu _{i}^{\prime },u_{i}^{\prime },\epsilon
_{ij}\right) \in \mathbb{R}^{d}$ be a collection of random variables defined
on $\left( \Omega ,\mathcal{F},P\right) $ and assume that $v_{i,n}\left(
\zeta \right) $ are measurable functions, possibly a finite section, of $%
\chi =\left( \chi _{1},...\right) .$

Network statistics are understood broadly for the purposes of this article.
They could be related to outcomes of strategic games where $v_{i,n}\left(
\zeta \right) $ might be the profit function of firm $i$ in a strategic game
with $n$ competitors. Alternatively, there might be an explicit network
represented by a graph with $n$ nodes and edges indicating a link between $i$
and $j.$ Such a graph can be represented by an adjacency matrix $D$ with
elements $d_{ij}$ where 
\begin{equation*}
d_{ij}=\left\{ 
\begin{array}{cc}
1 & \text{if }i\text{ and }j\text{ form a link} \\ 
0 & \text{if }i\text{ and }j\text{ do not form a link}%
\end{array}%
\right.
\end{equation*}%
and where $d_{ij}$ are functions of observable characteristics $\zeta ,$
unobservable characteristics $\mu $ that may also affect other outcomes and
idiosyncratic errors $\epsilon $ which are independent at the level of
individual links between $i$ and $j$ and are denoted as $\epsilon _{ij}.$

Examples for $v_{i,n}\left( \zeta \right) $ then are the degree 
\begin{equation*}
v_{i,n}\left( \zeta \right) =n_{i}=\sum_{j=1}^{n}d_{ij},
\end{equation*}
the clustering coefficient $v_{i,n}\left( \zeta \right) =\sum_{i\leq
j}d_{ij}d_{ik}d_{jk}$ or average characteristics of links that $i$ forms
with other members of the network, 
\begin{equation*}
v_{i,n}\left( \zeta \right) =\sum_{j=1}^{n}m_{ij}z_{j}
\end{equation*}
where $m_{ij}=n_{i}^{-1}d_{ij}.$

Networks may generate additional outcomes $y=\left( y_{1},...,y_{n}\right) $
that are implicitly, or if reduced forms exist, explicitly functions of $%
\zeta $ and $\eta $ as well as other exogenous observed and unobserved
factors $z,$ $\mu $ and $u.$ It is assumed that there exists a measurable
mapping $\Upsilon _{i,n}$ such that $y_{i}=\Upsilon _{i,n}\left( z,\zeta
,\eta ,u\right) .$ An example are linear peer effects models. Let $M\ $be
the spatial matrix with elements $m_{ij}=n_{i}^{-1}d_{ij}.$ Then,%
\begin{equation}
y=\lambda My+z\beta +u  \label{Model_Peer}
\end{equation}%
with reduced form $y=\left( I-\lambda M\right) ^{-1}\left( z\beta +u\right)
. $ Let $v=\left( v_{1,n}\left( \zeta \right) ,...,v_{n,n}\left( \zeta
\right) \right) .$ Relevant network statistics in this model are of the form 
$v=Mz$ or $v=Mu$ where $M$ and thus $v$ are functions of location
characteristics $\zeta .$ For example if $v=Mu,$ then the network statistic $%
v_{i,n}\left( \zeta \right) $ is given by $v_{i,n}\left( \zeta \right)
=\sum_{j=1}^{n}m_{ij}u_{j}=n_{i}^{-1}\sum_{j=1}^{n}d_{ij}u_{j},$ such that $%
v_{i,n}\left( \zeta \right) $ depends on $\zeta $ through $d_{ij}$ and $%
n_{i} $ as described above. In the conext of peer effects models,
establishing asymptotic properties of $Mz$ and $Mu$ is needed for example in
the analysis of estimators for the parameters $\lambda $ and $\beta $ in
maximum likelihood or moment based estimators.

The goal of this paper is to establish laws of large numbers and central
limit theorems for $\sum_{i=1}^{n}v_{i,n}\left( \zeta \right) $ under
general high level restrictions on the dependence of $v_{i,n}.$ This is done
without assuming specific parameteric models of how $v_{i,n}\left( \zeta
\right) $ is generated. Rather, dependence is described with the help of
mixing measures similar to the concept of mixingales in the time series
literature. The main technical difficulty is that proximity is determined by
a random variable $\zeta $ rather than given a priori.

Network statistics and outcomes are generated conditional on exogenous
network location indicators $\zeta .$ The overall dependence then rests both
on the distribution of $\zeta $ and the functional forms of $v$ and $y.$

To make progress on the latter the following device is introduced. It is
assumed that there exists a collection of functions $g_{ij}\left( \zeta
\right) $ with the property that $g_{ij}\left( \zeta \right) \in \left[ 0,1%
\right] ,$ and with the convention that $g_{ii}\left( \zeta \right) =1$ a.s.
A form of the triangular inequality 
\begin{equation}
g_{ij}\left( \zeta \right) ^{-1}\leq g_{ik}\left( \zeta \right)
^{-1}+g_{kj}\left( \zeta \right) ^{-1}\text{ for all }k
\label{Triangular_gij}
\end{equation}%
is assumed to hold. When $g_{ij}\left( \zeta \right) =0,$ the inequality is
interpreted as requiring either $g_{ik}\left( \zeta \right) =0$ or $%
g_{kj}\left( \zeta \right) =0$ for all $k.$ The interpretation of $g_{ij}$
is that of an inverse distance measure between agents $i$ and $j.$ When $%
v_{i}$ is profit or utility in an $n$-player game, $g_{ij}$ can be a measure
of the marginal effects of actions by $j$ on payoffs for $i.$ In the case of
network models a natural choice for $g_{ij}$ may be the conditional link
probability $E\left[ d_{ij}|\zeta \right] =p_{ij}\left( \zeta \right) $ such
that in this case $g_{ij}\left( \zeta \right) =p_{ij}\left( \zeta \right) .$
Whether the triangular inequality in (\ref{Triangular_gij}) holds for $E%
\left[ d_{ij}|\zeta \right] $ depends on the specific functional form of $%
d_{ij}$ as well as the conditional probability measure $E\left[ .|\zeta %
\right] .$ An expample where (\ref{Triangular_gij}) holds is presented
below. Related concepts of spatial distance functions were proposed by
Bolthausen (1982) and later introduced in the econometrics literature by
Conley (1999) and Jenish and Prucha (2009,2012). The difference between
these approaches and the treatment here is that there is no fixed ordering
of the data. The notion of distance between $i$ and $j$ is a random variable
that depends on the realization of the process $\zeta $ that determines
network location. An a priori ordering of the sample is therfore not
possible, unlike in Bolthausen (1982) and papers that are based on his
theory.

Often $p_{ij}\left( \zeta \right) $ depends on $\zeta $ only through the
elements $\zeta _{i}$ and $\zeta _{j}$ which are specific to agents $i$ and $%
j.$ This is the case in network formation models such as Goldsmith-Pinkham
and Imbens (2013) or Graham (2016, 2017). However, generally, such a
restriction may not hold in strategic games and it is not imposed in this
paper. When $g_{ij}\left( \zeta \right) $ depends on all $\zeta $ then it
may be an approximation or bound to a parameter in a game with finite number
of players. For example, the marginal effect of $j$'s actions on $i$'s
profits may change as $n$ changes. Since the function $g_{ij}\left( \zeta
\right) $ is not allowed to depend on $n$ for technical reasons that will
become clear later, it could be chosen for example as the supremum of the
marginal effect over all games of sizes $n\in \left\{ 2,...\right\} .$

Assume that the functions $g_{ij}\left( \zeta \right) $ are measurable with
respect to $\mathcal{Z}$. In network models it seems natural to define
distance in terms of the function $g_{ij}$ which loosely speaking measures
the intensity of the interaction between $i$ and $j.$ This implies that $%
g_{ij}$ decreases as the distance between $i$ and $j$ increases. An example
given above is the probability of $i$ and $j$ forming a link in a network.
This probably declines if the the distance, measured in units that are
meaningful in the context of a specific model, between $i$ and $j$
increases. From the perspective of formulating an asymptotic theory such a
decreasing function $g_{ij}$ is somewhat inconvenient.

Therefore define a map $\Lambda $ that transforms $g_{ij}$ into a measure
that is more akin of a metric. In some cases a transformation to a metric in
the conventional Euclidian sense is possible, although not required.

\begin{definition}
\label{Def_Lambda}For $k\in \mathbb{R}_{+}\cup \left\{ 0\right\} $ let $%
\Lambda $ be a non-random strictly monotonically decreasing function $%
\Lambda :\mathbb{R}_{+}\cup \left\{ 0\right\} \mathbb{\rightarrow }\left[ 0,1%
\right] $ with $\Lambda \left( k\right) >\Lambda \left( k^{\prime }\right) $
for all $k<k^{\prime },$ $\lim_{k\rightarrow \infty }\Lambda \left( k\right)
=0$ and $\Lambda \left( 0\right) =1.$
\end{definition}

The following example serves as an illustration of the role that $\Lambda $
plays. The setting of the example is a simple network formation model where
links are determined by the distance between characteristics $\zeta _{i}$
and $\zeta _{j}$ and an idiosyncratic disturbance $\epsilon _{ij}.$ In a
model of this type, the distance function $g_{ij}$ can be chosen as the
conditional probability $E\left[ d_{ij}|\zeta \right] $ which is only a
function of $\zeta $ as long as $\epsilon _{ij}$ is iid and independent of $%
\zeta .$ A simplified version of a model by Graham (2016) serves as an
example.

\begin{example}
\label{Example_Lambda}Let $\left\Vert .\right\Vert $ be the Euclidian norm$.$
Consider a network formation model for a directed network where 
\begin{equation}
d_{ij}=1\left\{ \alpha _{0}+\alpha _{\zeta }\left\Vert \zeta _{i}-\zeta
_{j}\right\Vert +\epsilon _{ij}>0\right\}  \label{Example_dij}
\end{equation}%
with $d_{ii}=0$ and $\epsilon _{ij}$ is iid logistic, and in particular $%
\epsilon _{ij}$ is independent of $\epsilon _{ji}$ implying that in general $%
d_{ij}\neq d_{ji},$ and $\alpha _{\zeta }<0.$ Then, $g_{ij}\left( \zeta
\right) =P\left( d_{ij}|\zeta \right) =H\left( \alpha _{0}+\alpha _{\zeta
}\left\Vert \zeta _{i}-\zeta _{j}\right\Vert \right) $ where $H\left(
.\right) =\exp \left( .\right) /\left( 1+\exp \left( .\right) \right) $ is
the logistic CDF. In this case choose $\Lambda \left( k\right) =cH\left(
\alpha _{0}+\alpha _{\zeta }k\right) $ with $c=\left( 1+\exp \left( \alpha
_{0}\right) \right) \exp \left( -\alpha _{0}\right) $ such that the inverse
of $\Lambda ,$ 
\begin{equation*}
\Lambda ^{-1}\left( g\right) =\alpha _{\zeta }^{-1}\left( \log \left(
g/\left( 1-g\right) \right) -\alpha _{0}\right)
\end{equation*}
when applied to $g_{ij}\left( \zeta \right) $ is equal to the norm $%
\left\Vert \zeta _{i}-\zeta _{j}\right\Vert .$ It follows that the set $%
\left\{ \zeta |g_{ij}\left( \zeta \right) \leq \Lambda \left( k\right)
\right\} $ is the set of values $\zeta $ where $\left\Vert \zeta _{i}-\zeta
_{j}\right\Vert \geq k.$
\end{example}

The example constitutes a special case where $g_{ij}\left( \zeta \right) $
depends on $i$ and $j$ only through the value of $\zeta .$ More generally, $%
g_{ij}$ could display heterogeneity beyond variation in $\zeta .$ In those
cases, $\Lambda $ may not be the exact inverse transformation. All that is
required of $\Lambda $ is that it satisfies Definition \ref{Def_Lambda}.

With the function $\Lambda $ defined in this way introduce the random
variables 
\begin{equation*}
w_{j,i,n}^{k}\left( \zeta \right) =v_{j,n}\left( \zeta \right) 1\left\{
g_{ij}\left( \zeta \right) \leq \Lambda \left( k\right) \right\}
\end{equation*}%
where the variable $w_{j,i,n}^{k}\left( \zeta \right) $ is the network
statistic $v_{j,n}\left( \zeta \right) $ of agent $j$ truncated by the event 
$1\left\{ g_{ij}\left( \zeta \right) \leq \Lambda \left( k\right) \right\} .$
For the model and definitions in Example \ref{Example_Lambda} it follows
that the truncating event corresponds to $\left\Vert \zeta _{i}-\zeta
_{j}\right\Vert \geq k$, in other words realizations of the location
distribution where $i$ and $j$ are separated by at least an amount $k.$

The next step in the argument now consists in defining a collection of
filtrations that contain information about network statistics of individuals
that are sufficiently distant from the current location, agent $i.$ If
dependence in the network is decaying with distrance then conditioning on
these network statistics should matter less and less as the distance is
increased. Controlling for the rate at which the dependence disappears
provides a way to describe dependence in the process that generates the
network statistics. Let 
\begin{equation}
\mathcal{B}_{i,n}^{k}=\sigma \left( w_{1,i,n}^{k}\left( \zeta \right)
,...,w_{i-1,i,n}^{k}\left( \zeta \right) ,w_{i+1,i,n}^{k}\left( \zeta
\right) ,...,w_{n,i,n}^{k}\left( \zeta \right) \right)
\label{Definition_B_i,n_k}
\end{equation}%
for all $n,k$ and $i\in \left\{ 1,...,n\right\} .$ Let $\left\Vert
.\right\Vert _{p}$ be the $L_{p}$ norm $\left\Vert x\right\Vert _{p}=\left(
\int \left\vert x\right\vert ^{p}dP\left( x\right) \right) ^{1/p}$ and $%
\left\Vert .\right\Vert _{p,\zeta }=\left( \int \left\vert x\right\vert
^{p}dP\left( x|\zeta \right) \right) ^{1/p}$ the $L_{p}$ norm on the
probability space $\left( \Omega ,\mathcal{F},Q_{\omega }\right) .$ Assume
that $E\left[ v_{i,n}\left( \zeta \right) \right] =\mu _{i,n}$ exists for
all $i$ and $n$ and that $\lim_{n\rightarrow \infty }E\left[ v_{i,n}\left(
\zeta \right) \right] =\lim_{n\rightarrow \infty }\mu _{i,n}=\mu _{i}$
exists for all $i.$ Now define the spatial mixing coefficients following
related definitions for time series processes, for example by McLeish
(1975), for all $n\geq 1,k\geq 0,$ and constants $c_{i}>0$ with $%
\sup_{i}c_{i}<\infty $ and $\sup_{i}\limfunc{Var}\left( v_{i,n}\left( \zeta
\right) |\zeta \right) \leq Kc_{i}$ for some constant $K,$%
\begin{equation}
\left\Vert \mu _{i,n}-E\left[ v_{i,n}\left( \zeta \right) |\mathcal{B}%
_{i,n}^{k}\right] \right\Vert _{2,\zeta }\leq c_{i}\psi _{i,k}\left( \zeta
\right) .\text{ }  \label{Mix1}
\end{equation}%
The fields $\mathcal{B}_{i,n}^{k}$ condition on information that is at least
a spatial distance of $\Lambda ^{-1}\left( k\right) $ from agent $i$ with
statistic $v_{i,n}\left( \zeta \right) .$ Since locations $\zeta $ are
random, the selection of agents $j$ into the conditioning set $\mathcal{B}%
_{i,n}^{k}$ is also random. The process $v_{i,n}\left( \zeta \right) $ is
called a spatial mixingale if $E\left[ \psi _{i,k}\left( \zeta \right) %
\right] \rightarrow 0$ as $k\rightarrow \infty .$

The motivation for the measure in (\ref{Mix1}) is that $\psi _{i,k}\left(
\zeta \right) =0$ if $v_{i,n}\left( \zeta \right) $ is independent of $%
\mathcal{B}_{i,n}^{k}.$ The criterion in (\ref{Mix1}) depends both on the
functional form of $v_{i,n}\left( \zeta \right) $ as well as the
distribution of $\zeta .$ To illustrate the connection consider the
following example

\begin{example}
\label{Example_Mixing}Let $d_{ij}$ be generated as in (\ref{Example_dij})
with $\epsilon _{ij}$ iid logistic and set $\alpha _{0}=0$ and $\alpha
_{\zeta }=-1$. Assume that 
\begin{equation}
\sup_{i}\sum_{j=1}^{\infty }P\left( \left\vert \zeta _{i}-\zeta
_{j}\right\vert \leq k\right) \leq K<\infty
\label{Example2_P_SumRestriction}
\end{equation}
for any $0\leq k<\infty $ and some $K>0.$ Consider the network statistic $%
v_{i,n}\left( \zeta \right) =\sum_{j=1}^{n}d_{ij}.$ Let $\mathcal{B}%
_{i,n}^{k}$ be defined as in (\ref{Definition_B_i,n_k}). Then, 
\begin{equation}
E\left[ \left\Vert \mu _{i,n}-E\left[ v_{i,n}\left( \zeta \right) |\mathcal{B%
}_{i,n}^{k}\right] \right\Vert _{2,\zeta }\right] \rightarrow 0
\label{Example_Convergence}
\end{equation}%
as $k\rightarrow \infty .$
\end{example}

The result in (\ref{Example_Convergence}) is established in Section \ref%
{ExampleProofs}. In most cases direct evaluation of $E\left[ .|\mathcal{B}%
_{i,n}^{k}\right] $ is too difficult. In those cases, approximations are an
alternative way of obtaining results. The proof of (\ref{Example_Convergence}%
) illustrates how this can be done. As the example illustrates, the
conditional nature of the mixing coefficients $\psi _{i,k}$ implies that
some restrictions on the distribution of $\zeta $ need to be imposed to make
the definitions useful. The nature of these restrictions depend on the type
of limiting result that is desired. To obtain a weak law of large numbers
the following restriction, which is a generalization of the assumption made
in (\ref{Example2_P_SumRestriction}) in Example \ref{Example_Mixing}, is
imposed on the joint distribution of $\zeta $ denoted by $P_{\zeta }$. Note
that $P_{\zeta }$ is the marginal distribution of $\zeta $ obtained from the
measure $P$ by integrating over the remaining components in $\chi .$

\begin{assume}
\label{Assume_Probability_Sum}Assume that $g_{ij}\left( \zeta \right)
=g_{ji}\left( \zeta \right) $ and $g_{ij}\left( \zeta \right) ^{-1}\leq
g_{ik}\left( \zeta \right) ^{-1}+g_{kj}\left( \zeta \right) ^{-1}.$ Assume
that there exists a nonstochastic function $\Lambda :\mathbb{R\rightarrow }%
\left[ 0,1\right] $ with $\Lambda \left( k\right) >\Lambda \left( k^{\prime
}\right) $ for all $k<k^{\prime }$ where $k,k^{\prime }\in \mathbb{R}$. Let $%
k_{m}$ be an increasing sequence of numbers $k_{m}\in \mathbb{R}$ for $m\in 
\mathbb{N}$ and define the disjoint events $A_{k_{m}}\left( i,j\right) $ as 
\begin{equation*}
A_{k_{m}}\left( i,j\right) =\left\{ \omega |\Lambda \left( k_{m}\right)
<g_{ij}\left( \zeta \right) \leq \Lambda \left( k_{m-1}\right) \right\}
\end{equation*}%
and 
\begin{equation*}
\Pr \left( A_{k_{m}}\left( i,j\right) \right) =\int_{A_{k_{m}}\left(
i,j\right) }dP_{\zeta }\left( \zeta \right) .
\end{equation*}%
such that for some $K<\infty ,$ and all $n$ 
\begin{equation}
\sum_{i=1}^{n}\frac{\log ^{2}\left( i+1\right) }{i^{2}}\sum_{j=i}^{n}%
\sum_{m=1}^{\infty }E\left[ \psi _{i,k_{m}}\left( \zeta \right)
|A_{k_{m}}\left( i,j\right) \right] \Pr \left( A_{k_{m}}\left( i,j\right)
\right) \leq K<\infty .  \label{psi_Summation}
\end{equation}
\end{assume}

The condition given in (\ref{psi_Summation}) is motivated by two
observations. The first is that the sets $A_{k_{m}}\left( i,j\right) $ for $%
m=1,...,\infty $ and $k_{0}=1$ constitute a partition of the sample sapce
for $\zeta $ such that $\Omega =\cup _{m=1}^{\infty }A_{k_{m}}\left(
i,j\right) $ for any $i$ and $j$ fixed and with $A_{k}\left( i,j\right) \cap
A_{k^{\prime }}\left( i,j\right) =\varnothing $ for any $k\neq k^{\prime }$.
It follows that the expectation $E\left[ \zeta \right] $ can be represented
as $E\left[ \zeta \right] =\sum_{m=1}^{\infty }E\left[ \zeta
|A_{k_{m}}\left( i,j\right) \right] \Pr \left( A_{k_{m}}\left( i,j\right)
\right) $ where $i$ and $j$ are arbitrary fixed integers. The second
component of (\ref{psi_Summation}) consists in an approximation of the
conditional covariance between $v_{i,n}\left( \zeta \right) $ and $%
v_{j,n}\left( \zeta \right) $ by the mixing coefficient $\psi
_{i,k_{m}}\left( \zeta \right) .$ By combining these two elements an upper
bound for the covariance between $v_{i,n}\left( \zeta \right) $ and $%
v_{j,n}\left( \zeta \right) $ is obtained in Lemma \ref{WeakLLN} below. The
covariance upperbound directly leads to a weak law of large numbers. As the
proof of the result in Example \ref{Example_Mixing} shows, functional form
restrictions on $v_{i,n}\left( \zeta \right) $ can be useful in bounding the
behavior of $\psi _{i,k}\left( \zeta \right) $ as $k$ tends to $\infty .$ In
addition, as the Example illustrates and as is evident from Assumption \ref%
{Assume_Probability_Sum}, it is the interplay between assumptions about the
functional form of $v_{i,n}\left( \zeta \right) $ and assumptions about the
distribution of $\zeta $ that in combination allow to control the dependence
in $v_{i,n}\left( \zeta \right) .$

The following example illustrates how (\ref{psi_Summation}) can be verified
in a simple case where network formation is driven by link specific
observables $\zeta _{ij}$ and limited to a neighborhood where $\left\vert
\zeta _{ij}\right\vert \leq 1.$ In other words, only individuals with link
characteristics $\zeta _{ij}$ that are within a certain range can connect.
The model is discussed in more detail in Section \ref{Section_Network_Model}
and is formalized as follows.

\begin{example}
\label{Example_NeighborhoodNetwork}Consider a network formation model for a
directed network where links are formed according to 
\begin{equation}
d_{ij}=1\left\{ -\left\vert \zeta _{ij}\right\vert +\epsilon _{ij}>0\right\}
1\left\{ \left\vert \zeta _{ij}\right\vert <\kappa _{u}\right\}
\label{Definition_dij_Neighborhood}
\end{equation}%
with $\kappa _{u}=1$ and $d_{ii}=0$ and where $\zeta _{ij}=\zeta _{ji}$ are
independently distributed random variables with uniform distribution on the
interval $[\left\vert i-j\right\vert -1,\left\vert i-j\right\vert ).$ Also, $%
\epsilon _{ij}$ is iid logistic, and in particular $\epsilon _{ij}$ is
independent of $\epsilon _{ji}$ implying that in general $d_{ij}\neq d_{ji}.$
Let $g_{ij}\left( \zeta \right) =H\left( -\left\vert \zeta
_{i}{}_{j}\right\vert \right) $ and $\Lambda \left( k\right) =H\left(
-k\right) .$ Also assume that 
\begin{equation}
\sup_{i}\left\vert \mu _{i,n}\right\vert +\left( E\left[ v_{i,n}^{2}\left(
\zeta \right) \right] \right) ^{1/2}\leq K<\infty  \label{EX_NN_MomentBound}
\end{equation}%
Then, it follows that $\psi _{i,k}\left( \zeta \right) =0$ for $k>1$ and
Condition (\ref{psi_Summation}) holds.
\end{example}

The proof of the first part of Example \ref{Example_NeighborhoodNetwork} is
given in Section \ref{ExampleProofs}. The second part of the assertion in
Example \ref{Example_NeighborhoodNetwork} can be understood as follows.
First note that for any $m^{\prime }<\infty $ the distribution of $\zeta
_{ij}$ satisfies 
\begin{equation}
\sup_{i}\sum_{j=i}^{n}\sum_{m=1}^{m^{\prime }}\Pr \left( A_{k_{m}}\left(
i,j\right) \right) =\sup_{i}\sum_{j=i}^{n}\Pr \left( \Lambda \left(
k_{m^{\prime }}\right) <g_{ij}\left( \zeta \right) \right) \leq K<\infty
\label{R_Prob_Sum_1}
\end{equation}%
where $1\left\{ \Lambda \left( k_{m^{\prime }}\right) <g_{ij}\left( \zeta
\right) \right\} =1\left\{ \left\vert \zeta _{ij}\right\vert \leq k\right\}
. $ To see this note that for any $i$ fixed, there are only at most $2\left(
k+1\right) +1$ values for $j$ for which $1\left\{ \left\vert \zeta
_{ij}\right\vert \leq k\right\} \neq 0$ with positive probability which
means that the sum $\sum_{j=i}^{n}\Pr \left( \left\vert \zeta
_{ij}\right\vert \leq k\right) \leq 2k+1.$

Choosing the sequence $k_{m}=m$ for simplicity it holds that%
\begin{equation*}
\psi _{i,k_{m}}\left( \zeta \right) =0
\end{equation*}%
for all $m>1.$ In addition, for $m=1$ one can choose $E\left[ \psi
_{i,1}\left( \zeta \right) \right] =\left\vert \mu _{i,n}\right\vert +\left(
E\left[ v_{i,n}^{2}\left( \zeta \right) \right] \right) ^{1/2}$ because,
using the H\"{o}lder inequality, 
\begin{equation*}
\left\Vert \mu _{i,n}-E\left[ v_{i,n}\left( \zeta \right) |\mathcal{B}%
_{i,n}^{1}\right] \right\Vert _{2,\zeta }\leq \left\vert \mu
_{i,n}\right\vert +\left( E\left[ v_{i,n}^{2}\left( \zeta \right) |\mathcal{B%
}_{i,n}^{1}\right] \right) ^{1/2}
\end{equation*}%
such that the bound follows from Jensen's inequality. These arguments show
that under the additional assumption that the moment bound in (\ref%
{EX_NN_MomentBound}) holds it follows that\textbf{\ } 
\begin{eqnarray}
&&\sum_{i=1}^{n}\frac{\log ^{2}\left( i+1\right) }{i^{2}}\sum_{j=i}^{n}%
\sum_{m=1}^{\infty }E\left[ \psi _{i,k_{m}}\left( \zeta \right)
|A_{k_{m}}\left( i,j\right) \right] \Pr \left( A_{k_{m}}\left( i,j\right)
\right)  \label{Remark_Probability_Summability} \\
&\leq &\sup_{i}\left( \sum_{j=1}^{\infty }\Pr \left( 1\left\vert \zeta
_{ij}\right\vert \leq 1\right) \right) \sum_{i=1}^{n}\frac{\log ^{2}\left(
i+1\right) }{i^{2}}\left( \left\vert \mu _{i,n}\right\vert +\left( E\left[
v_{i,n}^{2}\left( \zeta \right) \right] \right) ^{1/2}\right)  \notag \\
&\leq &K^{2}\sum_{i=1}^{n}\frac{\log ^{2}\left( i+1\right) }{i^{2}}<\infty .
\notag
\end{eqnarray}%
An interpretation of the summability condition in (\ref%
{Remark_Probability_Summability}) that 
\begin{equation}
\sum_{j=1}^{\infty }\Pr \left( \left\vert \zeta _{ij}\right\vert \leq
1\right) \leq K  \label{Characteristic_Dispersion}
\end{equation}%
can be obtained from the Borel-Cantelli Lemma. Individual $i$ has, with
probability one, at most finitely many neighbors located in an area
contained within a radius $1.$ In order to satisfy this condition
individuals need to be spread out sufficiently in characteristics space.
Assumption \ref{Assume_Probability_Sum} then provides a precise defintion of
sparsity.

\section{Laws of Large Numbers\label{Section_LLN}}

This section develops a number of laws of large numbers. The first result
establishes an upperbound for the covariance between $v_{i,n}\left( \zeta
\right) $ and $v_{j,n}\left( \zeta \right) .$ The form of the upper bound,
combined with Assumption \ref{Assume_Probability_Sum}, directly leads to a
weak law of large numbers (WLLN). While it may be natural to try to extend
the proofs of strong laws for mixingale time series to the current contect
an inspection of the proofs for example in McLeish (1975) indicate that
applying martingale methods directly to this context seems difficult. The
triangular array nature of $v_{i,n}\left( \zeta \right) $ as well as the
fact that the dependence structure in the data may change as $n$ increases
pose challenges that make it hard to develop an analog to Doob's inequality
(see Hall and Heyde, 1980, p. 20).

An alternative approach pursued here and also mentioned in Hall and Heyde,
(1980, p.22) is to use methods based on moment restrictions proposed by
Stout (1974). The framework in Stout requires some adjustments, most notably
an extension to triangular arrays of random variables which is first
provided. As it turns out, the moment inequality in Lemma \ref{WeakLLN} is
the key component needed to apply the insights from Stout (1974).

Before proceeding, unifrom bounds on the moments of $v_{i,n}\left( \zeta
\right) $ are imposed in the following assumption. These bounds are
necessary because of the role covariances play in the results that follow.

\begin{assume}
\label{Assume_Moments_Mixing}Assume that for some $\delta >0,$ $\sup_{i}E%
\left[ \left\vert v_{i,n}\left( \zeta \right) \right\vert ^{2+\delta }|\zeta %
\right] \leq K<\infty $ a.s. for all $n$ and $\sup_{i}\limfunc{Var}\left(
v_{i,n}\left( \zeta \right) |\zeta \right) \leq Kc_{i}$ for some constant $%
K. $
\end{assume}

The following weak law of large numbers can now be established.

\begin{lemma}
\label{WeakLLN}Let Assumptions \ref{Assume_Probability_Sum} and \ref%
{Assume_Moments_Mixing} hold. Assume that $\sup_{i}c_{i}\leq K$. Then%
\begin{equation}
\limfunc{Cov}\left( v_{i,n}\left( \zeta \right) ,v_{j,n}\left( \zeta \right)
\right) \leq 2Kc_{i}c_{j}\sum_{m=0}^{\infty }E\left[ \psi _{i,k_{m}}\left(
\zeta \right) |A_{k_{m}}\left( i,j\right) \right] P\left( A_{k_{m}}\left(
i,j\right) \right) .  \label{Cov_Bound}
\end{equation}%
For $S_{n}=\sum_{i=1}^{n}\left( v_{i,n}\left( \zeta \right) -\mu
_{i,n}\right) $ it follows that 
\begin{equation}
\limfunc{Var}\left( n^{-1/2}S_{n}\right) \leq
Kn^{-1}\sum_{i,j=1}^{n}c_{j}c_{i}\sum_{m=0}^{\infty }E\left[ \psi
_{i,k_{m}}\left( \zeta \right) |A_{k_{m}}\left( i,j\right) \right] P\left(
A_{k_{m}}\left( i,j\right) \right)  \label{VarS_Bound}
\end{equation}%
and%
\begin{equation*}
n^{-1}S_{n}\rightarrow _{p}0.
\end{equation*}
\end{lemma}

The bounds in Lemma \ref{WeakLLN} can be used to establish a maximal
inequality, almost sure convergence results and a strong law of large
numbers. These results are derived by extending a maximual inequality due to
Stout (1974, Section 2.4) to triangular arrays. Let $\left\{ v_{i,l}\left(
\zeta \right) \right\} _{i=1}^{l}$ for $l=1,...$ be a triangular array and
use the short hand notation $v_{i,l}=v_{i,l}\left( \zeta \right) $ with the
convention that $v_{i,l}=0$ for $i>l.$ Define 
\begin{equation*}
M_{a,n}=\max_{a<k\leq n,l\geq 1}\left\vert \sum_{i=a+1}^{a+k}v_{i,l}-\mu
_{i,l}\right\vert
\end{equation*}%
where $\mu _{i,l}=E\left[ v_{i,l}\right] .$ Let $F_{a,n}=P_{n}\left(
v_{a+1,1},v_{a+1,2,}v_{a+2,2},v_{a+1,3}....,v_{a+n,n},...,v_{a+n,l},...%
\right) $ be the joint probability distribution of the random array $v_{i,l}$
for $i\leq a+n$ and $l>1.$ Impose the following additional restrictions.

\begin{assume}
\label{Assume_Triangular_Array}For $\tilde{v}_{i,n}=v_{i,n}\left( \zeta
\right) -\mu _{i,n}$ assume that 
\begin{equation}
\sup_{\left\{ m|m\geq n\right\} }\left\vert \tilde{v}_{i,m}-\tilde{v}%
_{i,n}\right\vert \leq u_{i,n}\left( n\log ^{2}\left( n+1\right) \right)
^{-1}  \label{Triangular_1}
\end{equation}%
with 
\begin{equation}
\underset{n\rightarrow \infty }{\lim \sup }\sum_{i=1}^{\infty }\left( \frac{E%
\left[ u_{i,n}^{2}\right] }{i\log ^{2}\left( i+1\right) }\right)
^{1/2}<\infty .  \label{Triangular_2}
\end{equation}
\end{assume}

Assumption \ref{Assume_Triangular_Array} covers two possible sampling
schemes. In one scheme, the network is generated on an infinite dimensional
sampling space as constructed above. Network statistics $v$ in that
framework do not change as the sample size $n$ increases. This implies that $%
v_{i,n}=v_{i,m}=v_{i}$ and $\mu _{i,n}=\mu _{i,m}=\mu _{i}$ for all $n,m$
and the conditions in (\ref{Triangular_1}) and (\ref{Triangular_2})
automatically hold for $u_{i}=0$ a.s. The second scenario covered by
Assumption \ref{Assume_Triangular_Array} is a situation where the network
structure changes as $n$ increases. This scenario corresponds to a situation
where new agents are randomly added to the network as $n$ grows, and thus
potentially are affecting the equilibrium network structure. The assumption
then restricts the effect additional agents have on the existing network
structure. As $n$ tends to infinity the effect needs to be negligible in a
way made precise in (\ref{Triangular_1}) and (\ref{Triangular_2}).

\begin{lemma}[Stout, 1974, Theorem 2.4.1]
\label{Lemma_Stout_2.4.1}Suppose that $g$ is a functional defined on the
joint distribution functions such that 
\begin{equation}
g\left( F_{a,k}\right) +g\left( F_{a+k,m}\right) \leq g\left(
F_{a,k+m}\right)  \label{S2.4.1}
\end{equation}%
for all $1\leq k<k+m$ and $a\geq 0,$%
\begin{equation}
E\left[ \left( \tsum\nolimits_{i=a+1}^{a+n}\tilde{v}_{i,l}\right) ^{2}\right]
\leq g\left( F_{a,n}\right)  \label{S2.4.2}
\end{equation}%
for all $l\geq 1,$ $n\geq 1$and $a\geq 0.$ Then%
\begin{equation}
E\left[ M_{a,n}^{2}\right] \leq \left( \log \left( 2n\right) /\log 2\right)
^{2}g\left( F_{a,n}\right)  \label{S2.4.3}
\end{equation}%
for all $n\geq 1$ and $a\geq 0.$
\end{lemma}

Condition (\ref{S2.4.2}) plays a crucial role in establishing Lemma \ref%
{Lemma_Stout_2.4.1}. It provides a moment bound that is uniform over all
elements in the triangular array. The assumption is justified here in light
of Lemma \ref{WeakLLN} and Assumptions \ref{Assume_Probability_Sum} and \ref%
{Assume_Moments_Mixing}.

The next task consists in extending Stout (1974, Theorem 2.4.2) to the case
of triangular arrays satisfying the uniform boundedness conditions imposed
above. The following Lemma provides the necessary result.

\begin{lemma}[Stout, 1974, Theorem 2.4.2]
\label{Lemma_Stout_2.4.2}Let Assumption \ref{Assume_Triangular_Array} hold.
Suppose that $g$ is a functional defined on the joint distribution functions
satisfying the restrictions in (\ref{S2.4.1}) and (\ref{S2.4.2}). Further
assume that there exists a function $h$ such that 
\begin{equation*}
h\left( F_{a,k}\right) +h\left( F_{a+k,m}\right) \leq h\left(
F_{a,k+m}\right)
\end{equation*}%
for all $1\leq k<k+m$ and $a\geq 0,$ $h\left( F_{a,n}\right) \leq K<\infty $
for all $n\geq 1$ and $a\geq 0,$ and%
\begin{equation*}
g\left( F_{a,n}\right) \leq Kh\left( F_{a,n}\right) /\log ^{2}\left(
a+1\right)
\end{equation*}%
for all $n\geq 1$ and $a>0.$ Let $S_{n,m}=\sum_{i=1}^{m}\tilde{v}_{i,n}$ for
some sequence $m\leq n.$ Then, $S_{n,n}$ converges almost surely.
\end{lemma}

The maximal inequality and almost sure convergence result are now direct
consequences of the modified limit laws in Lemmas \ref{Lemma_Stout_2.4.1}
and \ref{Lemma_Stout_2.4.2} which are extending Stout to triangular arrays.

\begin{theorem}
\label{Theorem_Maximal}Let Assumptions\ref{Assume_Probability_Sum}, \ref%
{Assume_Moments_Mixing} and \ref{Assume_Triangular_Array} hold. Let $%
M_{a,n}=\max_{a<k\leq n,l\geq 1}\left\vert \sum_{i=a+1}^{a+k}v_{i,l}\left(
\zeta \right) -\mu _{i,l}\right\vert .$ Then, for constants $c_{i}\geq 0,$ 
\begin{equation*}
E\left[ M_{a,n}^{2}\right] \leq \left( \log \left( 2n\right) /\log 2\right)
^{2}\sum_{i,j=a+1}^{a+n}c_{i}c_{j}\sum_{m=0}^{\infty }E\left[ \psi
_{i,k_{m}}\left( \zeta \right) |A_{k_{m}}\left( i,j\right) \right] P\left(
A_{k_{m}}\left( i,j\right) \right)
\end{equation*}%
for all $n\geq 1$ and $a\geq 0.$
\end{theorem}

The almost sure convergence result for the empirical sum $S_{n}$ below
requires implicit constraints on the upper bounds for the variance of $%
v_{i,n}.$ For convergence these variances need to decay to zero at certain
rates as implied by the condition in (\ref{Th_AS_Cond1}) below. For a strong
law of large numbers which is based on the almost sure convergence result,
these bounds on the variances are replaced with appropriate norming of $%
S_{n} $ as well constraints on the distribution of characteristics in
Assumption \ref{Assume_Probability_Sum}. The almost sure convergence result
is stated first.

\begin{theorem}
\label{Theorem_AlmostSure}Let Assumptions \ref{Assume_Moments_Mixing} and %
\ref{Assume_Triangular_Array} hold. Let $S_{n}=\sum_{i=1}^{n}\left(
v_{i,n}-\mu _{i,n}\right) .$ If there are constants $c_{i}$ is such that for
all $n\geq 1$ and $a\geq 0$%
\begin{equation}
\sum_{i=a+1}^{a+n}c_{i}\log ^{2}\left( i\right)
\sum_{j=i}^{a+n}c_{j}\sum_{m=0}^{\infty }E\left[ \psi _{i,k_{m}}\left( \zeta
\right) |A_{k_{m}}\left( i,j\right) \right] P\left( A_{k_{m}}\left(
i,j\right) \right) <\infty  \label{Th_AS_Cond1}
\end{equation}%
then $S_{n}\ $converges almost surely.
\end{theorem}

When Assumption \ref{Assume_Probability_Sum} holds then $c_{i}=i^{-1}$ is
sufficient since for 
\begin{equation*}
P_{ij}=\sum_{m=0}^{\infty }E\left[ \psi _{i,k_{m}}\left( \zeta \right)
|A_{k_{m}}\left( i,j\right) \right] P\left( A_{k_{m}}\left( i,j\right)
\right)
\end{equation*}%
it follows that 
\begin{equation*}
\sum_{i=a+1}^{a+n}i^{-1}\log ^{2}\left( i\right)
\sum_{j=i}^{a+n}j^{-1}P_{ij}\leq \sum_{i=a+1}^{a+n}i^{-2}\log ^{2}\left(
i\right) \sum_{j=i}^{a+n}P_{ij}<\infty
\end{equation*}%
satisfies the condition above.

Theorems \ref{Theorem_Maximal} and \ref{Theorem_AlmostSure} form the basis
for the following strong law of large numbers which critically hinges on the
bounds established in Lemma \ref{WeakLLN}. In particular, the bound in (\ref%
{VarS_Bound}) only depends on the sample size through the summation upper
bound. The result below summarizes the laws of large numbers covered in this
section.

\begin{theorem}
\label{Theorem_Stout_3.7.1}Let Assumptions \ref{Assume_Probability_Sum}, \ref%
{Assume_Moments_Mixing} hold. Assume that $\sup_{i}c_{i}<\infty .$ Let $%
S_{n}=\sum_{i=1}^{n}\left( v_{i,n}-\mu _{i,n}\right) .$ Then, $%
S_{n}/n\rightarrow _{p}0$. If in addition also Assumption \ref%
{Assume_Triangular_Array} holds then $S_{n}/n\rightarrow 0$ almost surely.
\end{theorem}

The strong law is an extension of Theorem 3.7.1 in Stout (1974) in two
directions. One is that triangular arrays are covered by Theorem \ref%
{Theorem_Stout_3.7.1} while Stout does not consider triangular arrays. This
is achieved by imposing the additional stability condition in Assumption \ref%
{Assume_Triangular_Array}. The second direction in which the result is
extended is by giving explicit upper bounds in the conext of network models
for the maximal inequality that drives the strong law (cf. Stout, 1974,
Theorem 2.4.1). The upper bound is directly linked to the mixing
coefficients defined in (\ref{Mix1}) and summability restrictions on the
joint distribution of $\zeta $ in Assumption \ref{Assume_Probability_Sum}.
This leads to a set of more primitive conditions that can be checked for
specific models.

\section{Central Limit Theory\label{Section_CLT}}

The proof of the central limit theorem builds on the notion of spatial
mixing developed in (\ref{Mix1}). It uses ideas from two strands of the
probability literature. One is a blocking argument that was proposed by
Eberlein (1984) in a general setting and applied to time series processes
under mixing conditions. The idea consists in dividing the sample into
blocks that increase in size with total sample size but at a slower rate.
The blocks are separated by buffer zones of data that is being discarded for
the purpose of the proof. The buffer zones also grow in size, but at a
slower rate than the blocks that are being kept for the proof of the CLT.
Under regularity conditions, the blocks and buffer zones can be chosen in
such a way that the discarded data asymptotically does not affect the
limiting distribution and that the blocks of data that are being kept can be
treated as independent.

The proof of Eberlein rests on the concept of absolute regularity which
implies that blocks are ultimately independent in the total variation norm.
Bolthausen (1982) establishes a CLT for spatially mixing processes that
would lend themselves to similar arguments as in Eberlein (1984). However,
the concept of mixing may be difficult to verify in practice. Thus, in
addition to a blocking scheme the proofs in this paper use a second set of
tools developed in the probability literature and used to prove CLT's for
dependent processes. A product expansion of the characteristic function
implicit in the work of Salem and Zygmund (1947) is used in McLeish (1974,
Theorem 1) to establish sufficient conditions for a CLT. McLeish (1974) uses
the approach to establish a CLT for martingale difference arrays. His work
was subseqently extended to a stable CLT by showing weak $L_{1}$ convergence%
\footnote{%
See Aldous and Eagleson (1978) for a definition of weak $L_{1}$ convergence.}
of the characteristic function by Hall and Heyde (1980, Theorem 3.2).

In the case of martinagle difference arrays conditioning arguments can be
used to elimiate terms from the product expansion of the characteristic
function. The moment conditions implied by (\ref{Mix1}) then are an
extension of the martingale difference concept where the conditional mean
zero property only holds for sufficiently distant realizations of the
process. With this modification the exact cancellations in the
characteristic function expansion turn into approximate cancellations that
can be neglected asymptotically under the right conditions. It is
interesting to note that McLeish (1975) who was the first to prove a
mixingale central limit theorem chose an entirely different proof strategy
based on martingale approximations and requiring the use of maximal
inequalities. As for the strong laws of large numbers, this proof strategy
does not appear well suited for the current application. It appears that
using the McLeish (1974) proof strategy in the context of a mixingale
conditions is a new result.

The strategy of proving the CLT rests on partitioning the sample into sets
of observations which are contributing to the limiting distribution and sets
ob observations that serve as buffer zones and that are asymptotically
negligible. For this purpose, fix $N\ll n$ and choose a set of centers $%
q_{1},...,q_{N}$ where each $q_{i}\in \left\{ 1,...,n\right\} .$ Conditional
on the observed $\zeta $ and for a $k$ fixed and each $q_{i}$ choose a set $%
J_{k}\left( q_{i}\right) $ of indicies such that $J_{k}\left( q_{i}\right)
\subset \left\{ 1,...,n\right\} .$ The set $J_{k}\left( q_{i}\right) $ is
constructed by selecting elements from $\left\{ 1,...,n\right\} $ without
replacement such that for each $j\in J_{k}\left( q_{i}\right) $ it follows
that the distance between $q_{i}$ and $j$ is at most $\Lambda \left(
k\right) \leq g_{q_{i}j}\left( \zeta \right) ,$%
\begin{equation}
J_{k}\left( q_{i}\right) =\left\{ j\in \left\{ 1,...,n\right\} |\Lambda
\left( k\right) \leq g_{q_{i}j}\left( \zeta \right) \right\}  \label{Def_J}
\end{equation}%
Similarly, for some $h>k$ define a buffer zone of observations denoted by $%
T_{k,h}\left( q_{i}\right) $ around $q_{i}$ with the property that all $\tau
\in T_{k,h}\left( q_{i}\right) $ satisfy the restriction that $\Lambda
\left( h\right) \leq g_{q_{i}\tau }\left( \zeta \right) <\Lambda \left(
k\right) ,$%
\begin{equation}
T_{k,h}\left( q_{i}\right) =\left\{ \tau \in \left\{ 1,...,n\right\}
|\Lambda \left( h\right) \leq g_{q_{i}\tau }\left( \zeta \right) <\Lambda
\left( k\right) \right\} .  \label{Def_T}
\end{equation}
Ultimately, $N$ is increasing with $n,$ although at a slower rate, in such a
way that both $J_{k}\left( q_{i}\right) $ and $T_{k,h}\left( q_{i}\right) $
asymptotically contain the appropriate number of elements. An explicit
recursive algorithm of how to construct these sets is given below. The
algorithm requires the parameters $k$ and $h$ to increase with sample size
and lets $k$ and $h$ depend on the point of approximation $q_{i}.$ This is
made explicit below by using the notation $k_{n}^{i}$ and $h_{n}^{i}.$ The
sequences $k_{n}^{i}$ and $h_{n}^{i}$ are chose to guarantee that the
cardinality of $J_{k}\left( q_{i}\right) $ denoted by $\left\vert
J_{k}\left( q_{i}\right) \right\vert $ satisfyies $\left\vert
J_{k_{n}^{i}}\left( q_{i}\right) \right\vert =c_{J}n^{3/4}$ in large samples
and that the cardinality $\left\vert T_{k,h}\left( q_{i}\right) \right\vert $
of $T_{k,h}\left( q_{i}\right) $ satisfies $\left\vert
T_{k_{n}^{i},h_{n}^{i}}\left( q_{i}\right) \right\vert =c_{T}n^{1/4-\epsilon
}$ for some $\epsilon >0$ in large samples. The fact that $k$ and $h$ depend
on $i$ allows for heterogeneity, in particular local variation in the amount
of spatial clustering. Constructing these sets is important in practice for
two reasons. One is to check whether the regularity conditions of the CLT
can be satisfied for a particular model and the second, maybe even more
important reason is to construct valid standard errors.

Using the definition of the sets $J_{k}\left( q_{i}\right) $ and $%
T_{k,h}\left( q_{i}\right) $ is used to form the random variables 
\begin{equation}
X_{i,n}=\sum_{j\in J_{k}\left( q_{i}\right) }\left( v_{j,n}\left( \zeta
\right) -\mu _{j,n}\right) \text{ for }i=1,...,N  \label{Def_Xi}
\end{equation}%
and 
\begin{equation}
U_{i,n}=\sum_{j\in T_{k,h}\left( q_{i}\right) }\left( v_{j,n}\left( \zeta
\right) -\mu _{j,n}\right) \text{ for }i=1,...,N.  \label{Def_Ui}
\end{equation}%
It follows that for $S_{n}=\sum_{i=1}^{n}\left( v_{i,n}\left( \zeta \right)
-\mu _{i,n}\right) $ one obtains $S_{n}=\sum_{i=1}^{N}\left(
X_{i,n}+U_{i,n}\right) .$ The proof of the central limit theorem then
consists in establishing that $n^{-1/2}\sum_{i=1}^{N}X_{i,n}\rightarrow
_{L_{1}}N\left( 0,\eta ^{2}\right) $ for some possibly random variable $\eta 
$ and that $n^{-1/2}\sum_{i=1}^{N}U_{i,n}=o_{p}\left( 1\right) .$

Whether these two results can be established depends on two features of the
data-generating process of $v_{i,n}\left( \zeta \right) .$ One is the rate
at which mixing coefficients $\psi _{i,h}\left( \zeta \right) $ tend to zero
as $h\rightarrow \infty .$ The second is whether $h$ and $k$ can be chosen
as functions of the sample size $n$ in such a way that $\left\vert
J_{k}\left( q_{i}\right) \right\vert $ is large enough and $\left\vert
T_{k,h}\left( q_{i}\right) \right\vert $ is small enough such that the
approximation argument in (\ref{Def_Xi}) and (\ref{Def_Ui}) can be applied.
Being able to construct the two types of sets $J_{k}\left( q_{i}\right) $
and $T_{k,h}\left( q_{i}\right) $ with the required amount of data in turn
depends on the interaction between properties of the model captured by $\psi
_{i,h}\left( \zeta \right) $ and properties of the distribution $P_{\zeta }.$
Enough sparicity is required so that neighborhoods $J_{k}\left( q_{i}\right) 
$ of $q_{i}$ are not overcrowded as $n$ grows.

To formulate the central limit theorem a set of filtrations needs to be
defined. Let $\mathcal{C}$ be a $\sigma $-field that is common to all
agents. Common factors are assumed to be measurable with respect to $%
\mathcal{C}.$ Now define the filtrations 
\begin{eqnarray}
\mathcal{F}_{n}^{0} &=&\left\{ \Omega ,\emptyset \right\} \vee \mathcal{C},%
\text{ }  \label{Def_Fi} \\
\mathcal{F}_{n}^{i} &=&\sigma \left( v_{j,n}\left( \zeta \right) |j\in
J_{k}\left( q_{i}\right) \right) \vee \mathcal{F}_{n}^{i-1}  \notag
\end{eqnarray}%
for $i=1,...,N$ where $\mathcal{A}\vee \mathcal{B}$ stands for the smallest $%
\sigma $-field that contains both $\sigma $-fields $\mathcal{A}$ and $%
\mathcal{B}.$ The construction implies that $\mathcal{F}_{n}^{i}\subseteq 
\mathcal{F}_{n}^{i+1}$ and that $X_{i,n}$ is measurable with respect to $%
\mathcal{F}_{n}^{i}.$ By construction, the distance between any element of $%
X_{i,n}$ and any element in $\mathcal{F}_{n}^{i-1}$ measured in terms of $%
g\left( .\right) $ is at least $g_{ij}\left( \zeta \right) \leq \left(
\Lambda \left( h\right) ^{-1}-\Lambda \left( k\right) ^{-1}\right) ^{-1}.$
Since $\Lambda $ is monotonically decreasing in its argument it has an
inverse $\Lambda ^{-1}.$ Then, for $h^{\prime }$ such that $\Lambda
^{-1}\left( 1/\Lambda \left( h\right) -1/\Lambda \left( k\right) \right)
^{-1}=h^{\prime }$ it follows that $\mathcal{B}_{j,n}^{h^{\prime }}\supseteq 
\mathcal{F}_{n}^{i-1}$ for all $j\in J_{k}\left( q_{i}\right) .$ This
implies that $X_{i,n}$ is a mixingale sequence relative to $\mathcal{F}%
_{n}^{i}$ since by Lemma \ref{Lemma_Mix_Ineq} in Section \ref%
{Section_CLT_Proofs} it follows that 
\begin{equation}
E\left[ \left\Vert E\left[ X_{i,n}|\mathcal{F}_{n}^{i-1}\right] \right\Vert
_{2,\zeta }^{2}\right] \leq \sup_{j}\left\vert c_{j}\right\vert E\left[ \psi
_{h^{\prime }}\left( \zeta \right) ^{2}\left\vert J_{k}\left( q_{i}\right)
\right\vert \right] .  \label{Mixing_Bound}
\end{equation}%
For example, if $\left\vert J_{k}\left( q_{i}\right) \right\vert \leq
c_{J}n^{3/4}$ and $h$ is chosen such that $E\left[ \psi _{h^{\prime }}\left(
\zeta \right) ^{2}\right] =n^{-1-\delta }$ then the RHS of (\ref%
{Mixing_Bound}) is $O\left( n^{-1/4-\delta }\right) .$ It is worth noting
that $\psi _{h^{\prime }}\left( \zeta \right) $ on the RHS of (\ref%
{Mixing_Bound}) does not depend on $i.$ The mixing coefficients $\psi
_{h^{\prime }}\left( \zeta \right) $ only decrease to zero because of
increasing buffer zones. There is no usable spatial orientation in the
sequence $X_{i,n}$ other than the fact that these components are separated
by buffer zones of increasing size. By construction $X_{i,n}$ is measurable
with respect to $\mathcal{F}_{n}^{i+k}$ for $k\geq 0$ such that%
\begin{equation*}
\left\Vert X_{i,n}-E\left[ X_{i,n}|\mathcal{F}_{n}^{i+k}\right] \right\Vert
_{2,\zeta }=0.
\end{equation*}

The first step of establishing a CLT for $S_{n}$ consists in proving a
central limit theorem for $S_{n,x}=n^{-1/2}\sum_{i=1}^{N}X_{i,n}.$ The
argument is based on the proof for martingale difference sequence CLTs by
McLeish (1974) and an extension of McLeish's result by Hall and Heyde (1980,
Theorem 3.2) to stable convergence. McLeish (1975b) proves a functional
central limit theorem for mixingales using a technique based on differential
equations for characteristic functions developed by Billingley (1968). It is
not clear that this approach is applixable in this context because the
maximal inequality needed for the result is not invariant to re-ordering of
the sample. The latter is critical to the blocking scheme where $J_{k}\left(
q_{i}\right) $ and $T_{k,h}\left( q_{i}\right) $ generally are functions of
the sample size. Because of the same sample size dependent blocking scheme,
Condition 3.21 of Hall and Heyde requiring a certain nesting property for
the filtrations also does not hold in the current environment. By focusing
on a baseline filtration $\mathcal{C}$, Kuersteiner and Prucha (2013) prove
a version of the Hall and Heyde result that does not require their nesting
condition. Stability with regard to a baseline filtration $\mathcal{C}$ as
in Kuersteiner and Prucha is therefore established here as well. The problem
studied in this paper is purely cross-sectional and based on mixingale
rather than mds assumptions and thus differs significantly from the panel
setting with mds sequences considered in Kuersteiner and Prucha (2013). The
following proposition delivers a CLT for $S_{n,x}=n^{-1/2}%
\sum_{i=1}^{N}X_{i,n}.$

\begin{proposition}
\label{Theorem_CLT_Blocks}Let $S_{n,x}=n^{-1/2}\sum_{i=1}^{N}X_{i,n}$ with $%
\left\{ X_{i},\mathcal{F}_{n}^{i}\right\} _{i=1}^{N}$ as defined above.
Assume that \newline
(i) $\max_{i}\left\vert n^{-1/2}X_{i,n}\right\vert \rightarrow _{p}0$\newline
(ii) $n^{-1}\sum_{i=1}^{N}X_{i,n}^{2}\rightarrow _{p}\eta ^{2}$ where $\eta $
is $\mathcal{C}$-measurable.\newline
(iii) $E\left( n^{-1}\max_{i}\left\vert X_{i,n}^{2}\right\vert \right) $ is
bounded in $n,$\newline
(iv) $\sup_{i}\psi _{i,h}\left( \zeta \right) \leq \psi _{h}\left( \zeta
\right) $ and $\psi _{h}\left( \zeta \right) \geq \psi _{h^{\prime }}\left(
\zeta \right) $ a.s. for all $h\leq h^{\prime }.$ \newline
(v) There are sequences $k_{n},h_{n}$ and $h_{n}^{\prime }$ such that for $%
h_{n}^{\prime }=\Lambda ^{-1}\left( 1/\Lambda \left( h_{n}\right) -1/\Lambda
\left( k_{n}\right) \right) ^{-1}$ and $E\left[ \psi _{h_{n}^{\prime
}}\left( \zeta \right) ^{2}\right] =O\left( n^{-\left( 1+\delta \right)
}\right) $\newline
Then, $S_{n,x}\rightarrow ^{d}Z$ ($\mathcal{C}$-stably) where $E\left[ \exp
\left( itZ\right) \right] =E\left[ \exp \left( -1/2\eta ^{2}t^{2}\right) %
\right] .$
\end{proposition}

The next step in the argument consists in combining the CLT for $S_{n,x}$
with an argument showing that $n^{-1/2}\sum_{i=1}^{N}U_{i,n}$ is
asymptotically negligible. The result rests on high level assumptions about
the asymptotic sizes of the sets $J_{k_{n}^{i}}\left( q_{i}\right) $ and $%
T_{k_{n}^{i},h_{n}^{i}}\left( q_{i}\right) .$ The size of these sets is
allowed to vary over different approximation points $q_{i}$ but ultimately
is required to settle at equivalent asymptotic sizes of $c_{J}n^{3/4}$ for $%
\left\vert J_{k_{n}^{i}}\left( q_{i}\right) \right\vert $ and $c_{T}n^{1/4}$
for $\left\vert T_{k_{n}^{i},h_{n}^{i}}\left( q_{i}\right) \right\vert $
where $c_{J}$ and $c_{T}$ are constants. An algorithm for choosing $%
J_{k_{n}^{i}}\left( q_{i}\right) $ and $T_{k_{n}^{i},h_{n}^{i}}\left(
q_{i}\right) $ is then required to provide a result that can be applied in
practice. This task is taken up after stating the following proposition.

\begin{proposition}
\label{Theorem_Spatial_Mixing}Assume the following. Let $S_{n,x}=n^{-1/2}%
\sum_{i=1}^{N}X_{i,n}$ with $\left\{ X_{i,n},\mathcal{F}_{n}^{i}\right\}
_{i=1}^{N}$ as defined above. Further assume that \newline
i) $\sup_{j}\left\vert v_{j,n}\left( \zeta \right) \right\vert \leq z\left(
\zeta \right) $ for all $n$ and $z\left( \zeta \right) $ is a random
variable with $E\left[ z\left( \zeta \right) ^{2+\delta }\right] <\infty $
for some $\delta >0$; \newline
ii) $n^{-1}\sum_{i=1}^{N}X_{i,n}^{2}\rightarrow _{p}\eta ^{2}$ where $\eta $
is $\mathcal{C}$-measurable; \newline
iii) For $\psi _{i,k}\left( \zeta \right) $ and $c_{i}$ defined in (\ref%
{Mix1}) assume that $\sup_{i}\psi _{i,k}\left( \zeta \right) \leq \psi
_{k}\left( \zeta \right) $ a.s. and $\sup_{i}c_{i}<\infty ;$\newline
iv) For arbitrary constants $0<c_{T}<\infty $ and $0<c_{J}<\infty $ let $%
N=n/\left( c_{T}\left\lfloor n^{1/4}\right\rfloor +c_{J}\left\lfloor
n^{3/4}\right\rfloor \right) $ and $J_{k}\left( q_{i}\right) $ defined in (%
\ref{Def_J}) there exists sequences $k_{n}^{i}$ such that 
\begin{equation*}
\sup_{i}\left\vert \frac{\left\vert J_{k_{n}^{i}}\left( q_{i}\right)
\right\vert }{n}N-1\right\vert \rightarrow 0\text{ a.s.}
\end{equation*}%
v) For $T_{k_{n}^{i},h_{n}^{i}}\left( q_{i}\right) $ defined in (\ref{Def_T}%
) and $\epsilon >0$ it follows that 
\begin{equation*}
\sup_{i}\left\vert \frac{\left\vert T_{k_{n}^{i},h_{n}^{i}}\left(
q_{i}\right) \right\vert }{n^{1/4-\epsilon }}-1\right\vert \rightarrow 0%
\text{ a.s.}
\end{equation*}%
\newline
vi) Let $\Lambda \ $be as given in Definition \ref{Def_Lambda}. For
sequences $k_{n}^{i},h_{n}^{i}$ satisfying (iv) and (v) let $k_{n},h_{n}$
and $h_{n}^{\prime }$ with $h_{n}>k_{n}$\textbf{\ }be such that $1/\Lambda
\left( h_{n}\right) -1/\Lambda \left( k_{n}\right) \leq \inf \left(
1/\Lambda \left( h_{n}^{i}\right) -1/\Lambda \left( k_{n}^{i}\right) \right) 
$ and $h_{n}^{\prime }=\Lambda ^{-1}\left( \left( 1/\Lambda \left(
h_{n}\right) -1/\Lambda \left( k_{n}\right) \right) ^{-1}\right) .$ Then%
\textbf{\ }it follows that $E\left[ \psi _{h_{n}^{\prime }}\left( \zeta
\right) \right] =O\left( n^{-1+\delta }\right) .$\newline
If (i)-(vi) hold then, 
\begin{equation*}
n^{-1/2}S_{n}\rightarrow _{d}N\left( 0,\eta ^{2}\right) \text{ }\mathcal{C}%
\text{-stably}.
\end{equation*}
\end{proposition}

In practice the usefulness of Proposition \ref{Theorem_Spatial_Mixing}
depends on the ability to estimate $\eta $ consistently so that confidence
intervals and test statistics can be formed. In addition, to verify the
regularity conditions for specific models one needs to check if blocks of
data $J_{k_{n}^{i}}\left( q_{i}\right) $ and $T_{k_{n}^{i},h_{n}^{i}}\left(
q_{i}\right) $ can indeed be constructed in a way that the regularity
conditions in Assumptions (iv)-(vi) of Proposition \ref%
{Theorem_Spatial_Mixing} hold. An example of how this is done for a
particular model is discussed in Section \ref{Section_Network_Model}.
However, in most empirical settings explicit verification of regularity
conditions is not practical. An alternative approach consists in
constructing sets $J_{k_{n}^{i}}\left( q_{i}\right) $ and $%
T_{k_{n}^{i},h_{n}^{i}}\left( q_{i}\right) $ in a way that satisfies the
asymptotic size constraints of Assumptions (iv) and (v) by construction. The
question whether Condition (vi) above holds for these sets then can be
answered or be left open depending on the circumstances and focus of the
analysis.

An explicit algorithm to construct the sets $J_{k_{n}^{i}}\left(
q_{i}\right) $ and $T_{k_{n}^{i},h_{n}^{i}}\left( q_{i}\right) $ is
presented now. Part of the notation used is inspired by the treatment of
nearest neighbors in Abadie and Imbens (2006, p. 239). The first step of the
argument assumes that $k$ and $h$ are fixed and describes how to choose sets 
$J_{k}\left( q_{i}\right) $ and $T_{k,h}\left( q_{i}\right) $ and
approximation centers $q_{i}\in \left\{ q_{1},...,q_{N}\right\} $ with each $%
q_{i}\in \left\{ 1,...,n\right\} .$ Subsequently, $k$ and $h$ are adjusted
so that the sets have the required number of elements.

For each agent $i$ in the sample create an ordered index of agents that are
close in terms of the $g_{ij}\left( \zeta \right) $ norm. For each $i\in
\left\{ 1,...,n\right\} $ and $m\in \left\{ 1,...,n\right\} $ define 
\begin{equation}
j_{m}\left( i\right) =\sum_{j=1}^{n}j1\left\{ \left(
\tsum\nolimits_{l=1}^{n}1\left\{ g_{il}\left( \zeta \right) \geq
g_{ij}\left( \zeta \right) \right\} \right) =m\right\} .  \label{Algo_1}
\end{equation}%
The index $j_{m}\left( i\right) $ locates the $m$-th closest neighbor of $i$
in terms of the metric $g_{ij}.$ Note that each $j_{m}\left( i\right) $ is a 
$\mathcal{Z}$-measurable function and thus a random variable that depends on 
$\zeta $. There are a total of $n^{2}$ indices $j_{m}\left( i\right) $
constructed in this way: for each $i\in \left\{ 1,...,n\right\} $ there is a
set of indices $\left\{ j_{1}\left( i\right) ,...,j_{n}\left( i\right)
\right\} .$ Using this construction create sets of ordered indices for each
agent, $J\left( i\right) =\left\{ j_{1}\left( i\right) ,...,j_{n}\left(
i\right) \right\} .$ Define the notation $v_{j_{m}\left( i\right) ,n}\left(
\zeta \right) =\sum_{l=1}^{n}v_{l,n}\left( \zeta \right) 1\left\{
l=j_{m}\left( i\right) \right\} $ to denote the sample observation related
to the $m$-th closest neighbor of $i.$ Now define the following sets
recursively. Fix $k,h\in \mathbb{Z}$ and $k<h$ throughout the recursion.
Start the recursion with $q_{1}=1$ and define $l_{1}=\sum_{j=1}^{n}1\left\{
g_{q_{1}j}\left( \zeta \right) \geq \Lambda \left( k\right) \right\} $ as
the number of elements within $\Lambda \left( k\right) $ distance of $q_{1}$
and $r_{1}=\sum_{j=1}^{n}1\left\{ \Lambda \left( k\right) >g_{q_{1}j}\left(
\zeta \right) \geq \Lambda \left( h\right) \right\} $ as the number of
elements in the buffer zone of $q_{1}$. Both $l_{1}$ and $r_{1}$ and $%
j_{l_{1}}\left( i\right) $ as well as $j_{l_{1}+r_{1}}\left( i\right) $ are $%
\mathcal{Z}$-measurable. Then set 
\begin{eqnarray}
J_{k}\left( q_{1}\right) &=&\left\{ j_{1}\left( q_{1}\right)
,...,j_{l_{1}}\left( q_{1}\right) \right\}  \label{Algo_2} \\
T_{k,h}\left( q_{1}\right) &=&\left\{ j_{l_{1}+1}\left( q_{1}\right)
,...,j_{l_{1}+r_{1}}\left( q_{1}\right) \right\}  \label{Algo_3} \\
I_{1} &=&J_{k}\left( q_{1}\right) \cup T_{k,h}\left( q_{1}\right) .
\label{Algo_4}
\end{eqnarray}%
The set $J_{k}\left( q_{1}\right) $ is the set of all indices of agents
within a distance $\Lambda \left( k\right) $ of agent $q_{1}.$ The set $%
T_{k,h}\left( q_{1}\right) $ is the set of indices of all agents at least a
distance $\Lambda \left( k\right) $ but not more than $\Lambda \left(
h\right) $ appart from $q_{1}.$ The set $I_{1}$ contains all the indices
that were assigned to either $J_{k}$ or $T_{k,h}.$ Letting $J=\left\{
1,...,n\right\} $ it follows that $J\cap I_{1}^{c}=J\backslash I_{1}$
denotes all the indices not yet assigned. Now assume that the sets $%
J_{k}\left( q_{1}\right) ,...,J_{k}\left( q_{N-1}\right) ,$ $T_{k,h}\left(
q_{1}\right) ,...,T_{k}\left( q_{N-1}\right) $ and $I_{1},...,I_{N-1}$ were
created recursively, and in particular that $I_{N-1}$ denotes the set of
already assigned observations. If $I_{N-1}=J$ then terminate the recursion.
Othrewise define 
\begin{equation}
q_{N}=\arg \min_{q\in J\backslash I_{N-1},i\in I_{N-1}}g_{qi}\left( \zeta
\right) \text{ s.t. }g_{qi}\left( \zeta \right) <\Lambda \left( k\right)
\label{Algo_5}
\end{equation}%
if such a $q_{N}$ exists. If no $q_{N}$ exists that satisfies the constraint
then terminate the recursion and assign all indicies in $J\backslash I_{N-1}$
to $T_{k,h}\left( q_{N}\right) $ where $q_{N}$ is some arbitrary element of $%
J\backslash I_{N-1}$. If the recursion continues then $q_{N}$ denotes the
closest agent who is at least a distance $\Lambda \left( k\right) $ from all
already assigned agents apart. With such a $q_{N}$ then define the number of
points in the $J_{k}\left( q_{N}\right) $ and $T_{k,h}\left( q_{N}\right) $
sets respectively as $l_{N}$ and $r_{N}$ in analogy with $l_{1}$ and $r_{1}$
as 
\begin{eqnarray}
l_{N} &=&\sum_{j\in J\left( q_{N}\right) \backslash I_{N-1}}1\left\{
g_{q_{N}j}\left( \zeta \right) \geq \Lambda \left( k\right) \right\}
\label{Algo_6} \\
r_{N} &=&\sum_{j\in J\left( q_{N}\right) \backslash I_{N-1}}^{n}1\left\{
\Lambda \left( k\right) >g_{q_{N}j}\left( \zeta \right) \geq \Lambda \left(
h\right) \right\}  \label{Algo_7}
\end{eqnarray}%
and the sets%
\begin{eqnarray}
J_{k}\left( q_{N}\right) &=&\left\{ j_{1}\left( q_{N}\right)
,...,j_{l_{N}}\left( q_{N}\right) \right\}  \label{Algo_8} \\
T_{k,h}\left( q_{N}\right) &=&\left\{ j_{l_{N}+1}\left( q_{N}\right)
,...,j_{l_{N}+r_{N}}\left( q_{N}\right) \right\}  \label{Algo_9} \\
I_{N} &=&I_{N-1}\cup \left( J_{k}\left( q_{N}\right) \cup T_{k,h}\left(
q_{N}\right) \right) .  \label{Algo_10}
\end{eqnarray}%
The recursion continues as long as sets can be formed. When the recursion
terminates set $J_{k}\left( q_{N}\right) =\emptyset $, $T_{k,h}\left(
q_{N}\right) =J\backslash I_{N-1}$ with $q_{N}$ chosen arbitrarily from $%
J\backslash I_{N-1}$ and $I_{N}=J.$

The final step of the argument consists in constructing sequences $k_{n}^{i}$
and $h_{n}^{i}$ such that the sets $J_{k_{n}^{i}}\left( q_{i}\right) $ and $%
T_{k_{n}^{i},h_{n}^{i}}\left( q_{i}\right) $ have the required number of
elements. For this purpose select what amounts to two bandwidth type
parameters $L_{n}=c_{J}n^{3/4}$ and $R_{n}=c_{T}n^{1/4-\epsilon }$ for some $%
\epsilon >0.$ Let $l_{i}^{k}=\sum_{j=1}^{n}1\left\{ g_{ij}\left( \zeta
\right) \geq \Lambda \left( k\right) \right\} $ and $r_{i}^{h,k}=%
\sum_{j=1}^{n}1\left\{ \Lambda \left( k\right) >g_{ij}\left( \zeta \right)
\geq \Lambda \left( h\right) \right\} $ be defined as before. For each $i$
and $n$ choose $k_{n}\left( i\right) $ such that $l_{i}^{k_{n}\left(
i\right) }$ is the largest value that satisfies $\left\lfloor
L_{n}\right\rfloor -1\leq l_{i}^{k_{n}\left( i\right) }\leq \left\lfloor
L_{n}\right\rfloor $ where $\left\lfloor L_{n}\right\rfloor $ denotes the
largest integer that is smaller than $L_{n}.$ Now, using the $k_{n}\left(
i\right) $ just defined, find $h_{n}\left( i\right) $ such that $%
r_{i}^{h_{n}\left( i\right) ,k_{n}\left( i\right) }$ is the largest value
that satisfies $\left\lfloor R_{n}\right\rfloor -1\leq r_{i}^{h_{n}\left(
i\right) ,k_{n}\left( i\right) }\leq \left\lfloor R_{n}\right\rfloor .$ This
procedure produces a collection of cut-off points $\left\{ k_{n}\left(
1\right) ,...,k_{n}\left( n\right) \right\} $ and $\left\{ h_{n}\left(
1\right) ,...,h_{n}\left( n\right) \right\} .$ In other words, each
observation $i$ in the sample is assigned a pair $\left( k_{n}\left(
i\right) ,h_{n}\left( i\right) \right) $. Next, set $N=n/\left( \left\lfloor
L_{n}\right\rfloor +\left\lfloor R_{n}\right\rfloor \right) .$

The sets $J$ and $T$ are now chosen according to the algorithm laid out in (%
\ref{Algo_1})-(\ref{Algo_10}). In particular choose the center points $q_{i}$
in (\ref{Algo_5}) for $i=1,...,N$ and given sets $%
I_{i-1},J_{k_{n}^{1}},...,J_{k_{n}^{i-1}},$ and $%
T_{k_{n}^{1},h_{n}^{1}},...,T_{k_{n}^{i-1},h_{n}^{i-1}}$ accroding to 
\begin{equation}
q_{i}=\arg \min_{q\in J\backslash I_{i-1},i\in I_{i-1}}g_{qi}\left( \zeta
\right) \text{ s.t. }g_{qi}\left( \zeta \right) <\Lambda \left( k_{n}\left(
q\right) \right)  \label{Algo_11}
\end{equation}%
where in particular the previously determined cut-off points $k_{n}\left(
q\right) $ specific to a candidate point $q$ are used. For $q_{i}$
determined in this way, define 
\begin{equation}
k_{n}^{i}=k_{n}\left( q_{i}\right) ,\text{ }h_{n}^{i}=h_{n}\left(
q_{i}\right) .  \label{Algo_12}
\end{equation}%
Then, form $J_{k_{n}^{i}}\left( q_{i}\right) $ according to (\ref{Algo_8})
with cut-off index $l_{q_{i}}^{k_{n}^{i}}.$ This guarantees that the set $%
J_{k_{n}^{i}}\left( q_{i}\right) $ has the required number of elements $%
\left\lfloor L_{n}\right\rfloor .$ Similarly, form $T_{k_{n}^{i},h_{n}^{i}}%
\left( q_{i}\right) $ according to (\ref{Algo_9}) with the cut-offs $%
l_{q_{i}}^{k_{n}^{i}}$ and $r_{q_{i}}^{h_{n}^{i},k_{n}^{i}}$ which again
guarantees that the set $T_{k_{n}^{i},h_{n}^{i}}\left( q_{i}\right) $ has
the desired number of elements.

The main result of this section can now be formulated. Relative to
Proposition \ref{Theorem_Spatial_Mixing} high level conditions on the sizes
of the sets $J_{k_{n}^{i}}\left( q_{i}\right) $ and $T_{k_{n}^{i},h_{n}^{i}}%
\left( q_{i}\right) $ are replaced with a regularity condition on the
distribution of $\zeta $ ruling out ties in the algorithm in ((\ref{Algo_1}%
)-(\ref{Algo_12}) that orders the data according to $g_{ij}\left( \zeta
\right) .$ Without such ties, there always exist sequences $k_{n}^{i}$ and $%
h_{n}^{i}$ such that the sets $J_{k_{n}^{i}}\left( q_{i}\right) $ and $%
T_{k_{n}^{i},h_{n}^{i}}\left( q_{i}\right) $ have $c_{J}n^{3/4}$ and $%
c_{T}n^{1/4-\epsilon }$ elements respectively in large samples. A maintained
assumption of the result below is that mixing coefficients that measure
dependence across the various partitions of the data decay at the required
rate. This is Condition (vi) in Proposition \ref{Theorem_Spatial_Mixing}
which remains a key assumption.

\begin{theorem}
\label{Theorem_Main}Let $k_{n}^{i},$ $h_{n}^{i},$ $J_{k_{n}^{i}}\left(
q_{i}\right) $ and $T_{k_{n}^{i},h_{n}^{i}}\left( q_{i}\right) $ be given by
(\ref{Algo_1})-(\ref{Algo_12}). Let $X_{in}$ be as defined in (\ref{Def_Xi}%
). Assume that $\Pr \left( g_{ij}\left( \zeta \right) =g_{ik}\left( \zeta
\right) \right) =0$ for all $i$ and all $j\neq k.$ Assume that Conditions
(i),(ii),(iii) and (vi) of Proposition \ref{Theorem_Spatial_Mixing} hold. 
\newline
Then, 
\begin{equation*}
n^{-1/2}S_{n}\rightarrow _{d}N\left( 0,\eta ^{2}\right) \text{ }\mathcal{C}%
\text{-stably}.
\end{equation*}
\end{theorem}

An estimator for the standard deviation $\eta $ can now be formed under the
additional assumption that $E\left[ v_{j,n}\left( \zeta \right) |j\in
J_{k_{n}^{i}}\left( q_{i}\right) \right] =\mu _{n}^{i}$ does not depend on $%
j.$ Then, define 
\begin{equation*}
\tilde{X}_{i,n}=\sum_{j\in J_{k_{n}^{i}}\left( q_{i}\right) }\left(
v_{j,n}\left( \zeta \right) -\hat{\mu}_{n}^{i}\right) \text{ for }i=1,...,N
\end{equation*}%
where $\hat{\mu}_{n}^{i}=\left\vert J_{k_{n}^{i}}\left( q_{i}\right)
\right\vert ^{-1}\sum_{j\in J_{k_{n}^{i}}\left( q_{i}\right) }v_{j,n}\left(
\zeta \right) $ is the local sample averge of $v_{j,n}\left( \zeta \right) $
over the set $J_{k_{n}^{i}}\left( q_{i}\right) .$ The estimator $\hat{\eta}%
^{2}$ of $\eta ^{2}$ can now be formed, in accordance with Condition (ii)\
in Theorem \ref{Theorem_CLT_Blocks}, as 
\begin{equation*}
\hat{\eta}^{2}=n^{-1}\sum_{i=1}^{N}\tilde{X}_{i,n}^{2}.
\end{equation*}%
As long as $\sup_{i}\left\vert \mu _{n}^{i}-\hat{\mu}_{n}^{i}\right\vert
=o_{p}\left( 1\right) $ this estimator is consistent for $\eta ^{2}.$ Using
well known results in Andrews (2005) it follows that the standardized
statistic $n^{-1/2}\hat{\eta}^{-1}S_{n}\rightarrow _{d}N\left( 0,1\right) $
even if $\eta $ is random in the limit, and consequently $n^{-1/2}S_{n}$ has
a mixed Gaussian rather than standard Gaussian limiting distribution. The
standard Gaussian limiting distribution of the standardized statistic is the
reason why conventional confidence intervals and Wald tests pivotal even in
the stable limit scenario, i.e. when $\eta $ is random.

\section{Network Model\label{Section_Network_Model}}

This section illustrates how the general theory developed in this paper can
be applied to specific models. The example builds on network formation
models analyzed by Goldsmith-Pinkham and Imbens (2013), Chandasekhar (2015),
de Paula (2016), Graham (2016), Leung (2016), Ridder and Sheng (2016) and
Sheng (2016).

Here consider the directed network model of Example \ref{Example_Mixing} in
more detail. Let $U_{i}\left( j\right) $ be the utility of individual $i$
forming a possible link with individual $j.$ The adjacency matrix $D$ is
formed by a strategic network formation model whereby $d_{ij}=1\left\{
U_{i}\left( j\right) >0\right\} $ is the $i,j$-th element of $D,$ with $%
d_{ii}=0$ and $d_{ij}\neq d_{ji}$ in general.\textbf{\ }Note that this
formulation differs from Goldsmith-Pinkham and Imbens (2013), for example,
who consider undirected networks with $d_{ij}=d_{ji}$. In the context of
friendship networks the distinction could be interpreted as the difference
between desired and actual friendships. In the directed case where $d_{ij}=1$
and $d_{ji}=0$ are possible, $d_{ij}=1$ indicates $i$ desires firendship
with $j$ while $d_{ji}=0$ indicates that this desire is not shared by $j.$
Actual friendship then can be modeled as $d_{ij}d_{ji},$ whereas the
undirected network model directly represents actual friendship.

The utility function $U_{i}\left( j\right) $ depends on observable
characteristics $\zeta $ and unobservable link specific factors $\epsilon .$
For simplicity, the observable link specific characteristics of agent $i,$ $%
\zeta _{ij}=\zeta _{ji},$ are assumed to take values in $\mathbb{R}$. The
variable $\zeta _{ij}$ could be constructed as follows. Assume that each
agent $i$ draws a vector of indepdent link specific observed characteristics 
$\zeta _{i}=\left( \zeta _{i}\left( 1\right) ,...,\zeta _{i}\left( j\right)
,...\zeta _{i}\left( n\right) \right) $ and $\zeta _{ij}\equiv \zeta
_{i}\left( j\right) -\zeta _{j}\left( i\right) $ is the link specific
difference in these characteristics. A possible motivation for this
formulation is the realization that individuals are highly complex and not
easily characterized by a single attribute. For each possible link, somewhat
different features are therefore relevant. Thus, while all components of $%
\zeta _{i}$ might share common features, represented for example by a common
mean, for each possible link somewhat different features matter, represented
as deviations form that common mean. The utility function is modelled as 
\begin{equation*}
U_{i}\left( j\right) =\left( \alpha _{0}+\alpha _{\zeta }\left\vert \zeta
_{ij}\right\vert +\epsilon _{ij}\right) f_{u}\left( \left\vert \zeta
_{ij}\right\vert \right)
\end{equation*}%
where $\epsilon _{ij}$ is iid, and independent of $\zeta $. The function $%
f_{u}\left( x\right) $ is defined as $f_{u}\left( x\right) =1\left\{
\left\vert x\right\vert \leq \kappa _{u}\right\} $ where $\kappa _{u}$ is a
fixed and finite cut-off and $1\left\{ .\right\} $ stands for the indicator
function. The case where $f_{u}\left( x\right) =1$ for all $x$ is covered by
setting $\kappa _{u}=\infty .$ Allowing for values of $\kappa _{u}<\infty $
represents the case where no links are possible if characteristics are too
different from zero, an arbitrary point of normalization. A prime example of
such characteristics are physical location. The parameters are restricted to 
$\alpha _{\zeta }\leq 0$ which my represent homophily, the property that
similarities between $i$ and $j$ are desirable, see for example
Chandrasekhar (2015) or Graham (2016).

The degree of dependence between elements in $D$ depends both on the
functional form of $U_{i}\left( j\right) $ as well as on the distribution of 
$\zeta $ and $\epsilon .$ Convergence conditions in Assumptions \ref%
{Assume_Probability_Sum} and \ref{Assume_Moments_Mixing} can only hold if
the network is sufficiently sparse. The functional form of $U_{i}\left(
j\right) $ is one component that generates sparsity, the other being the
distribution of the random network locations $\zeta _{ij}.$ Related
conditions can be found in Meester and Roy (1996), Mele (2015), Leung (2016,
2018), Menzel (2016). A key difference between these references and the
sparsity assumption introduced here is that the physical size of the network
as measured by the location variables $\zeta $ increases without bound,
keeping the utility parameters constant, while these authors use
localization parameters that limit link formation locally as the sample size
grows.

Now consider a specific parameterization of the network formation model
where $\kappa _{u}<\infty $ and the support of $\zeta _{i}$ is bounded. For
ease of exposition normalize $\kappa _{u}$ to be a finite integer. This
scenario is similar to $m$-dependent processes in time series analysis where
only a finite number of elements in the random sequence are dependent.
Consider the degree of node $i.$ To further simplify the example abstract
from dependence of the degree from sample size in the following way. For
each $i=\left\{ 1,...n\right\} $ define infinite sequences of random
variables $\left\{ \zeta _{ij}\right\} _{j=-\infty }^{\infty }$ and $\left\{
\epsilon _{ij}\right\} _{j=-\infty }^{\infty }.$\footnote{%
The probability space constructed in Section \ref{Section_ProbabilitySpace}
can accommodate these sequences.} For $i=\left\{ 1,...,n\right\} $ and $%
j=\left\{ -\kappa _{u},...,n+\kappa _{u}\right\} $ let $d_{ij}$ be given as
in (\ref{Definition_dij_Neighborhood}). It follows that $d_{ij}=0$ for $%
\left\vert i-j\right\vert >\kappa _{u}.$ Then define the degree of $i$ as $%
v_{i}=\sum_{j=-\kappa _{u}-1}^{n+\kappa _{u}+1}d_{ij}$ for all $i\in \left\{
1,...,n\right\} .$ The setup corresponds to a situation where the network
has formed independently of the observed sample. Network statistics $v_{i}$
are then recorded directly in the data rather than obtained from
calculations done based on observed $d_{ij}.$ Alternatively, one can also
consider the sample based network statistic $v_{i,n}=\sum_{j=1}^{n}d_{ij}.$
In this case, the sampling scheme involves observing $d_{ij}$ in the data
and computing $v_{i,n}$ based on these observed data.

\begin{proposition}
\label{Proposition_Example_Network}Let $\mu _{i,n}=E\left[ v_{i,n}\right] $
and $\mu _{i}=E\left[ v_{i}\right] .$ Let $\mathcal{B}_{i,n}^{k}$ be defined
as in (\ref{Definition_B_i,n_k}) and define $\mathcal{B}_{i}^{k}$ as 
\begin{equation*}
\mathcal{B}_{i}^{k}=\sigma \left( w_{1,i}^{k}\left( \zeta \right)
,...,w_{i-1,i}^{k}\left( \zeta \right) ,w_{i+1,i}^{k}\left( \zeta \right)
,...,w_{n,i}^{k}\left( \zeta \right) \right)
\end{equation*}%
where $w_{j,i}^{k}\left( \zeta \right) =v_{j}\left( \zeta \right) 1\left\{
g_{ij}\left( \zeta \right) \leq \Lambda \left( k\right) \right\} .$ Under
the conditions of Example \ref{Example_NeighborhoodNetwork}, except that $%
\kappa _{u}>1$ is allowed, the following holds:\newline
(i) $\left\Vert \mu _{i,n}-E\left[ v_{i,n}\left( \zeta \right) |\mathcal{B}%
_{i,n}^{k}\right] \right\Vert _{2,\zeta }=0$ for $k>\kappa _{u},$ $%
\left\Vert \mu _{i}-E\left[ v_{i}\left( \zeta \right) |\mathcal{B}_{i}^{k}%
\right] \right\Vert _{2,\zeta }=0$ for $k>\kappa _{u}.$\newline
(ii) $n^{-1}\sum_{i=1}^{n}\left( v_{i,n}-\mu _{i,n}\right) \rightarrow
_{a.s.}0$ and $n^{-1}\sum_{i=1}^{n}v_{i}\rightarrow _{a.s.}\mu .$\newline
(iii) $n^{-1/2}\sum_{i=1}^{n}\left( v_{i,n}-\mu _{i,n}\right) \rightarrow
_{d}N\left( 0,\sigma ^{2}\right) $ and $n^{-1/2}\sum_{i=1}^{n}\left(
v_{i}-\mu \right) \rightarrow _{d}N\left( 0,\sigma ^{2}\right) .$
\end{proposition}

The proof of the proposition is contained in Section \ref{Proofs_Network}.
Two elements of the example significantly simplify the argument. One is that
heterogeneity of the characteristics distribution, in particular a location
shift parameterized through the mean, is directly tied to the observation
index $i.$ As a result, the algorithm for finding the blocks $J_{k}\left(
q\right) $ and $T_{k,h}\left( q\right) $ greatly simplifies. The second
element is the built in limited dependence of the network statitics which is
achieved by requiring independence in location characteristics and
unobservables as well as a functional form restriction that limits the size
of neighborhoods. With these functional form restrictions the mixingale
coefficients can be computed easily for $k$ large enough and the mixingale
condition holds trivially.

\section{Conclusion}

The paper develops a general asymptotic theory for network data as well as
data that is dependent in a way that does not easily allow to reduce
statistics of interest to an independent sampling framework. The setup is
completely non-parametetric, although results from this work clearly are
relevant for the analysis of parameter estimators in parametric network
models. The conditions needed for the laws of large numbers and the central
limit theorem are high level in the sense that they restrict conditional
moments of observables as well as the joint distribution of location
characteristics.

The work done is clearly limited in scope and much more needs to be
accomplished to turn this approach into a fully applicable method. The list
of topics for future work includes checking regularity conditions for
specific models. The hope is that approximation techniques discussed in
Example \ref{Example_Lambda} can be generalized to more complex models.
Similarly, a more detailed investigation of distributions $P_{\zeta }$ that
lead to tractable models is of interest. Another topic for future work is to
expand on the algorithm that was proposed in order to obtain valid standard
errors. Currently, it is assumed that the functions $g_{ij}\left( \zeta
\right) $ are known. In practice, two avenues seem reasonable. One is to
work with ad-hoc functions such as $\max \left( 1,1/\left\Vert \zeta
_{i}-\zeta _{j}\right\Vert \right) $ or to use more model based approaches.
In the latter, it is expected that $g_{ij}$ would need to be estimated
parametrically or non-parametrically. The feasibility of the proposed
algorithm under those circumstances then would need to be established and
it's properties investigated.

\newpage

\appendix

\renewcommand{\thetheorem}{\Alph{section}.\arabic{theorem}} %
\renewcommand{\theassume}{\Alph{section}.\arabic{assumec}} %
\renewcommand{\theequation}{\Alph{section}.\arabic{equation}}

\setcounter{assumec}{0} \setcounter{theorem}{0} \setcounter{lemmac}{0} %
\setcounter{equation}{0} %\setcounter{page}{1}

\renewcommand{\theassumec}{\Alph{section}.\arabic{assumec}} %
\renewcommand{\thelemmac}{\Alph{section}.\arabic{lemmac}}\renewcommand{%
\thetheorem}{\Alph{section}.\arabic{theorem}} \renewcommand{\theequation}{%
\Alph{section}.\arabic{equation}}

\newpage

\section{Appendix\label{APPSUP}}

\subsection{Probability Space\label{Section_ProbabilitySpace}}

Let $\mathcal{B}\left( \mathbb{R}^{d}\right) $ the Borel algebra of subset
of $\mathbb{R}^{d}.$ Consider the sequence of probability spaces 
\begin{equation*}
\left( \mathbb{R}^{d},\mathcal{B}\left( \mathbb{R}^{d}\right) \right)
=\left( \Omega _{1},\mathcal{F}_{1}\right) ,\left( \mathbb{R}^{d}\times 
\mathbb{R}^{d},\mathcal{B}\left( \mathbb{R}^{d}\right) \otimes \mathcal{B}%
\left( \mathbb{R}^{d}\right) \right) =\left( \Omega _{1}\times \Omega _{2},%
\mathcal{F}_{1}\otimes \mathcal{F}_{2}\right) ...
\end{equation*}%
with probability measures $P_{1},$ $P_{2},...$ Because $\mathbb{R}^{d}$ is a
complete separable metric space it follows from Kolmogorov's exentsion
thoerem, Shiryaev (1996) Theorem II.3.3 and Remark (p.165), that there
exists a probability measure $P$ on $\left( \Omega _{1}\times \Omega _{2}...,%
\mathcal{F}_{1}\otimes \mathcal{F}_{2}...\right) =\left( \Omega ,\mathcal{F}%
\right) $ such that $P$ agrees with all $P_{i}.$ Let $\chi =\left( \chi
_{1},\chi _{2},...\right) $ where $\chi _{i}$ are random variables on $%
\left( \Omega _{i},\mathcal{F}_{i}\right) .$ By the extension theorem the
process $\chi $ exists on $\left( \Omega ,\mathcal{F},P\right) .$ Assume
that $v_{i,n}\left( \zeta \right) $ are measurable functions that depend
only on $\left( \chi _{1},\chi _{2},...,\chi _{n}\right) .$

Let $\mathcal{Z}$ be a sub-sigma field of $\mathcal{F}$. By Breiman (1992),
Theorem 4.34 and Theorem A.46 there exists a regular conditional
distribution of $\chi $ given $\mathcal{Z}$ where for fixed $\omega \in
\Omega ,$ $Q_{\omega }\left( B|\mathcal{Z}\right) $ is called a regular
conditional distribution for $\chi $ given $\mathcal{Z}$ if for $B\in 
\mathcal{F}$ fixed, $Q_{\omega }\left( B|\mathcal{Z}\right) $ is a version
of $P\left( \chi \in B|\mathcal{Z}\right) $ and $Q_{\omega }\left( B|%
\mathcal{Z}\right) $ is a probability on $\mathcal{F}$. Note that this
construction guarantees the existence of conditional expectations. Following
Eagleson (1975) for fixed $\omega \in \Omega $ define the measure space $%
\left( \Omega ,\mathcal{F},Q_{\omega }\right) $ with expectation relative to 
$Q_{\omega }$ denoted by $E_{\omega }.$ In what follows $\mathcal{Z}$ is the
sigma field generated by $\zeta .$

\subsection{Proofs for Section \protect\ref{Section_LLN}}

The proofs of results reported in the main section follow.

\begin{proof}[Proof of Lemma \protect\ref{WeakLLN}]
The proof is based on showing that the variance of the process $n^{-1}S_{n}$
where $S_{n}=\sum_{i=1}^{n}\left( v_{i,n}\left( \zeta \right) -\mu
_{i,n}\right) $ tends to zero. First consider the covariance between $%
v_{i,n}\left( \zeta \right) $ and $v_{j,n}\left( \zeta \right) $ where%
\begin{eqnarray*}
&&\limfunc{Cov}\left( v_{i,n}\left( \zeta \right) ,v_{j,n}\left( \zeta
\right) \right) \\
&=&E\left[ \left( v_{i,n}\left( \zeta \right) -\mu _{i,n}\right) \left(
v_{j,n}\left( \zeta \right) -\mu _{j,n}\right) \right] \\
&=&\sum_{m=0}^{\infty }E\left[ E\left[ \left( v_{i,n}\left( \zeta \right)
-\mu _{i,n}\right) \left( v_{j,n}\left( \zeta \right) -\mu _{j,n}\right)
|\zeta \right] |A_{k_{m}}\left( i,j\right) \right] P\left( A_{k_{m}}\left(
i,j\right) \right) \\
&=&\sum_{m=0}^{\infty }E\left[ E\left[ \left( v_{i,n}\left( \zeta \right) -E%
\left[ v_{i,n}\left( \zeta \right) |\mathcal{B}_{i,n}^{k_{m}}\right] \right)
\left( v_{j,n}\left( \zeta \right) -\mu _{j,n}\right) |\zeta \right]
|A_{k_{m}}\left( i,j\right) \right] P\left( A_{k_{m}}\left( i,j\right)
\right) \\
&&+\sum_{m=0}^{\infty }E\left[ E\left[ \left( E\left[ v_{i,n}\left( \zeta
\right) |\mathcal{B}_{i,n}^{k_{m}}\right] -\mu _{i,n}\right) \left(
v_{j,n}\left( \zeta \right) -\mu _{j,n}\right) |\zeta \right]
|A_{k_{m}}\left( i,j\right) \right] P\left( A_{k_{m}}\left( i,j\right)
\right) .
\end{eqnarray*}%
Now use the conditional Cauchy-Schwarz inequality and the fact that $E\left[
\left\vert v_{j,n}\left( \zeta \right) -\mu _{i,n}\right\vert ^{2}|\zeta %
\right] ^{1/2}\leq c_{i}$ by assumption, 
\begin{eqnarray}
&&E\left[ \left( E\left[ v_{i,n}\left( \zeta \right) |\mathcal{B}%
_{i,n}^{k_{m}}\right] -\mu _{i,n}\right) \left( v_{j,n}\left( \zeta \right)
-\mu _{j,n}\right) |\zeta \right]  \label{WeakLLN_D1} \\
&\leq &E\left[ \left\vert E\left[ v_{i,n}\left( \zeta \right) |\mathcal{B}%
_{i,n}^{k_{m}}\right] -\mu _{i,n}\right\vert ^{2}|\zeta \right] ^{1/2}E\left[
\left\vert v_{j,n}\left( \zeta \right) -\mu _{j,n}\right\vert ^{2}|\zeta %
\right] ^{1/2}  \notag \\
&\leq &c_{i}\psi _{i,k_{m}}\left( \zeta \right) E\left[ \left\vert
v_{j,n}\left( \zeta \right) -\mu _{j,n}\right\vert ^{2}|\zeta \right] ^{1/2}
\notag \\
&\leq &Kc_{i}c_{j}\psi _{i,k_{m}}\left( \zeta \right)  \notag
\end{eqnarray}%
and%
\begin{eqnarray}
&&E\left[ E\left[ \left( v_{i,n}\left( \zeta \right) -E\left[ v_{i,n}\left(
\zeta \right) |\mathcal{B}_{i,n}^{k_{m}}\right] \right) \left( v_{j,n}\left(
\zeta \right) -\mu _{i,n}\right) |\zeta \right] |A_{k_{m}}\left( i,j\right) %
\right]  \label{WeakLLN_D2} \\
&=&E\left[ \left( v_{i,n}\left( \zeta \right) -E\left[ v_{i,n}\left( \zeta
\right) |\mathcal{B}_{i,n}^{k_{m}}\right] \right) \left( v_{j,n}\left( \zeta
\right) -\mu _{i,n}\right) |A_{k_{m}}\left( i,j\right) \right] =0  \notag
\end{eqnarray}%
because conditional on $A_{k_{m}}\left( i,j\right) ,$ $\left( v_{j,n}\left(
\zeta \right) -\mu _{i,n}\right) $ is measurable with respect to $\mathcal{B}%
_{i,n}^{k_{m}}.$ It then follows that 
\begin{eqnarray*}
&&E\left[ \left( v_{i,n}\left( \zeta \right) -E\left[ v_{i,n}\left( \zeta
\right) |\mathcal{B}_{i,n}^{k_{m}}\right] \right) \left( v_{j,n}\left( \zeta
\right) -\mu _{i,n}\right) |A_{k_{m}}\left( i,j\right) \right] \\
&=&E\left[ v_{i,n}\left( \zeta \right) \left( v_{j,n}\left( \zeta \right)
-\mu _{i,n}\right) |A_{k_{m}}\left( i,j\right) \right] -E\left[ E\left[
v_{i,n}\left( \zeta \right) \left( v_{j,n}\left( \zeta \right) -\mu
_{i,n}\right) |\mathcal{B}_{i,n}^{k_{m}}\right] |A_{k_{m}}\left( i,j\right) %
\right] .
\end{eqnarray*}%
Since $\mathcal{B}_{i,n}^{k_{m}}\supseteq A_{k_{m}}\left( i,j\right) $ it
follows from Breiman (1992, Proposition 4.20) that 
\begin{eqnarray*}
&&E\left[ E\left[ v_{i,n}\left( \zeta \right) \left( v_{j,n}\left( \zeta
\right) -\mu _{i,n}\right) |\mathcal{B}_{i,n}^{k_{m}}\right]
|A_{k_{m}}\left( i,j\right) \right] \\
&=&E\left[ v_{i,n}\left( \zeta \right) \left( v_{j,n}\left( \zeta \right)
-\mu _{i,n}\right) |A_{k_{m}}\left( i,j\right) \right]
\end{eqnarray*}%
which establishes (\ref{WeakLLN_D2}). It then follows from (\ref{WeakLLN_D1}%
) and (\ref{WeakLLN_D2}) that 
\begin{equation}
\limfunc{Cov}\left( v_{i,n}\left( \zeta \right) ,v_{j,n}\left( \zeta \right)
\right) \leq 2Kc_{i}c_{j}\sum_{m=0}^{\infty }E\left[ \psi _{i,k_{m}}\left(
\zeta \right) |A_{k_{m}}\left( i,j\right) \right] P\left( A_{k_{m}}\left(
i,j\right) \right)  \label{Cov_Ineq}
\end{equation}%
Using the inequality in (\ref{Cov_Ineq}) leads to 
\begin{eqnarray}
\limfunc{Var}\left( n^{-1}S_{n}\right) &\leq &n^{-2}\sum_{i,j=1}^{n}E\left[
\left( v_{i,n}\left( \zeta \right) -\mu _{i,n}\right) \left( v_{j,n}\left(
\zeta \right) -\mu _{i,n}\right) \right]  \label{D_VarS_Bound} \\
&\leq &2n^{-2}K\sum_{i,j=1}^{n}c_{i}c_{j}\sum_{m=0}^{\infty }E\left[ \psi
_{i,k_{m}}\left( \zeta \right) |A_{k_{m}}\left( i,j\right) \right] P\left(
A_{k_{m}}\left( i,j\right) \right)  \notag \\
&\leq &n^{-1}\left( \sup c_{i}\right)
^{2}2K\sup_{i}\sum_{j=1}^{n}\sum_{m=0}^{\infty }E\left[ \psi
_{i,k_{m}}\left( \zeta \right) |A_{k_{m}}\left( i,j\right) \right] P\left(
A_{k_{m}}\left( i,j\right) \right)  \notag \\
&=&O\left( n^{-1}\right)  \notag
\end{eqnarray}%
such that the order in the last line of (\ref{D_VarS_Bound}) follows.
\end{proof}

\begin{proof}[Proof of Lemma \protect\ref{Lemma_Stout_2.4.1}]
The proof closely follows Stout (1974, p.24) with the necessary adjustments
to allow for triangular arrays. Fix $a\geq 0$ and proceed by induction. For $%
n=1$ the claim of Lemma \ref{Lemma_Stout_2.4.1} holds trivially because the
right hand side of (\ref{S2.4.2}) does not depend on $l.$ Then assume that (%
\ref{S2.4.3}) holds for all $n\leq N$ and all $a\geq 0$ where we take $N$ as
even. First consider the case where $n\leq N/2$ such that for all $l\geq 1$ 
\begin{equation*}
\left( \tsum\nolimits_{i=a+1}^{a+n}\tilde{v}_{i,l}\right) ^{2}\leq
M_{a,N/2}^{2}.
\end{equation*}%
If $N/2<n\leq N$ then for $l\geq 1$ 
\begin{equation*}
\left( \tsum\nolimits_{i=a+1}^{a+n}\tilde{v}_{i,l}\right) ^{2}\leq
M_{a,N/2}^{2}+2\left\vert \tsum\nolimits_{i=a+1}^{a+N/2}\tilde{v}%
_{i,l}\right\vert M_{a+N/2,N/2}+M_{a+N/2,N/2}^{2}
\end{equation*}%
and thus, as in Stout,%
\begin{equation*}
M_{a,N}^{2}\leq M_{a,N/2}^{2}+2\left\vert \tsum\nolimits_{i=a+1}^{a+N/2}%
\tilde{v}_{i,l}\right\vert M_{a+N/2,N/2}+M_{a+N/2,N/2}^{2}.
\end{equation*}%
Proceeding as in Stout by taking expectations on both sides and using the
induction hypothesis leads to 
\begin{equation*}
E\left[ M_{a,N}^{2}\right] \leq \left( \frac{\log N}{\log 2}+\left( \frac{%
\log N}{\log 2}\right) ^{2}\right) \left( g\left( F_{a,N/2}\right) +g\left(
F_{a+N/2,N/2}\right) \right) .
\end{equation*}%
Using (\ref{S2.4.1}) and the fact that $\left( \log N\right) \left( \log
2\right) +\left( \log N\right) ^{2}\leq \log \left( 2N\right) ^{2}$ leads to 
\begin{equation*}
E\left[ M_{a,N}^{2}\right] \leq \left( \log \left( 2N\right) /\log 2\right)
^{2}g\left( F_{a,N}\right) .
\end{equation*}%
The case where $N$ is odd follows in the same way as in the proof of Stout.
\end{proof}

\begin{proof}[Proof of Lemma \protect\ref{Lemma_Stout_2.4.2}]
The proof again closely follows Stout (1974, p.26). Fix $m=a+n$ for any $%
a\geq 0.$ Then, 
\begin{equation*}
E\left[ \left( \tsum\nolimits_{i=a+1}^{a+n}\tilde{v}_{i,a+n}\right) ^{2}%
\right] \leq g\left( F_{a,n}\right) \leq \frac{Kh\left( F_{a,n}\right) }{%
\log ^{2}\left( a+1\right) }\leq \frac{K^{2}}{\log ^{2}\left( a+1\right) }%
\rightarrow 0
\end{equation*}%
as $a\rightarrow \infty .$ This implies that $S_{a+n,a+n}$ is a Cauchy
sequence in $L_{2}.$ By the completness of $L_{2}$ there exists a random
variable $S$ with $E\left[ S^{2}\right] <\infty $ such that $E\left[ \left(
S_{n,n}-S\right) ^{2}\right] \rightarrow 0$ as $n\rightarrow \infty $ where $%
a=0$ was chosen without loss of generality.

First establish that there exist a subsequence that converges almost surely.
In particular consider $S_{2^{k},2^{k}}$ which converges almost surely if $%
\sum_{k=1}^{\infty }E\left[ \left( S-S_{2^{k},2^{k}}\right) ^{2}\right]
<\infty .$ But 
\begin{eqnarray*}
E\left[ \left( S-S_{2^{k},2^{k}}\right) ^{2}\right] &=&\lim_{n}E\left[
\left( S_{n,n}-S_{n,2^{k}}+S_{n,2^{k}}-S_{2^{k},2^{k}}\right) ^{2}\right] \\
&\leq &\lim_{n}E\left[ \left( S_{n,n}-S_{n,2^{k}}\right) ^{2}\right]
+\lim_{n}E\left[ \left( S_{n,2^{k}}-S_{2^{k},2^{k}}\right) ^{2}\right] \\
&&+2\lim_{n}\left( E\left[ \left( S_{n,n}-S_{n,2^{k}}\right) ^{2}\right] E%
\left[ \left( S_{n,2^{k}}-S_{2^{k},2^{k}}\right) ^{2}\right] \right) ^{1/2}.
\end{eqnarray*}%
Now, noting that $S_{n,n}-S_{n,2^{k}}=\tsum\nolimits_{i=2^{k}+1}^{2^{k}+n}%
\tilde{v}_{i,n},$ it follows that 
\begin{equation*}
\lim_{n}E\left[ \left( S_{n,n}-S_{n,2^{k}}\right) ^{2}\right] \leq \underset{%
n\rightarrow \infty }{\lim \sup }g\left( F_{2^{k},n-2^{k}}\right) \leq K%
\underset{n\rightarrow \infty }{\lim \sup }\frac{h\left(
F_{2^{k},n-2^{k}}\right) }{\log ^{2}\left( 2^{k}+1\right) }\leq \frac{K^{2}}{%
\log ^{2}\left( 2^{k}+1\right) }
\end{equation*}%
and by assumption 
\begin{eqnarray*}
\underset{n\rightarrow \infty }{\lim \sup }E\left[ \left(
S_{n,2^{k}}-S_{2^{k},2^{k}}\right) ^{2}\right] &\leq &\underset{n\rightarrow
\infty }{\lim \sup }E\left[ \left( \tsum\nolimits_{i=1}^{2^{k}}\left\vert 
\tilde{v}_{i,n}-\tilde{v}_{i,2^{k}}\right\vert \right) ^{2}\right] \\
&\leq &\lim_{n\rightarrow \infty }E\left[ \left(
\tsum\nolimits_{i=1}^{2^{k}}\sup_{m>n}\left\vert \tilde{v}_{i,m}-\tilde{v}%
_{i,2^{k}}\right\vert \right) ^{2}\right] \\
&\leq &\lim_{n\rightarrow \infty }E\left[ \left( \left( 2^{k}\log ^{2}\left(
2^{k}+1\right) \right) ^{-1}\tsum\nolimits_{i=1}^{2^{k}}u_{i}\right) ^{2}%
\right] \\
&\leq &\left( 2^{k}\log ^{2}\left( 2^{k}+1\right) \right) ^{-1}\left(
\tsum\nolimits_{i=1}^{2^{k}}\left( i\log ^{2}\left( i+1\right) \right)
^{-1/2}E\left[ u_{i}^{2}\right] ^{1/2}\right) ^{2} \\
&\leq &\frac{K^{2}}{2^{k}\log ^{2}\left( 2^{k}+1\right) }.
\end{eqnarray*}%
where for the second last inequality and by Stout (1974, p.201), for all $%
k\geq 1,$ 
\begin{eqnarray*}
E\left[ \left( \left( 2^{k}\log ^{2}\left( 2^{k}+1\right) \right)
^{-1}\tsum\nolimits_{i=1}^{2^{k}}u_{i}\right) ^{2}\right] &\leq &\left(
2^{k}\log ^{2}\left( 2^{k}+1\right) \right) ^{-2}E\left[ \left(
\tsum\nolimits_{i=1}^{2^{k}}u_{i}\right) ^{2}\right] \\
&\leq &\left( 2^{k}\log ^{2}\left( 2^{k}+1\right) \right) ^{-2}\left( E\left[
\tsum\nolimits_{i=1}^{2^{k}}E\left[ u_{i}^{2}\right] ^{1/2}\right] \right)
^{2} \\
&\leq &\left( 2^{k}\log ^{2}\left( 2^{k}+1\right) \right) ^{-1}\left( E\left[
\tsum\nolimits_{i=1}^{2^{k}}\left( \frac{E\left[ u_{i}^{2}\right] }{i\log
^{2}\left( i+1\right) }\right) ^{1/2}\right] \right) ^{2}.
\end{eqnarray*}%
It then follows that 
\begin{equation*}
E\left[ \left( S-S_{2^{k},2^{k}}\right) ^{2}\right] \leq 3\frac{K^{2}}{\log
^{2}\left( 2^{k}+1\right) }
\end{equation*}%
which implies that 
\begin{equation*}
\sum_{k=1}^{\infty }E\left[ \left( S-S_{2^{k},2^{k}}\right) ^{2}\right] \leq
3K^{2}\sum_{k=1}^{\infty }\log ^{-2}\left( 2^{k}+1\right) <\infty
\end{equation*}%
such that $S_{2^{k},2^{k}}$ converges to $S$ almost surely. Finally, show
that 
\begin{equation*}
\max_{2^{k-1}\leq n\leq 2^{k}}\left\vert
S_{n,n}-S_{2^{k-1},2^{k-1}}\right\vert \rightarrow 0\text{ a.s. as }%
k\rightarrow \infty .
\end{equation*}%
This follows as before if 
\begin{equation}
\sum_{k=1}^{\infty }E\left[ \max_{2^{k-1}\leq n\leq 2^{k}}\left(
S_{n,n}-S_{2^{k-1},2^{k-1}}\right) ^{2}\right] <\infty .  \label{Dev_sum}
\end{equation}%
Consider 
\begin{eqnarray*}
\max_{2^{k-1}\leq n\leq 2^{k}}\left( S_{n,n}-S_{2^{k-1},2^{k-1}}\right) ^{2}
&=&\max_{2^{k-1}\leq n\leq 2^{k}}\left(
S_{n,n}-S_{n,2^{k-1}}+S_{n,2^{k-1}}-S_{2^{k-1},2^{k-1}}\right) ^{2} \\
&\leq &\max_{2^{k-1}\leq n\leq 2^{k}}\left( S_{n,n}-S_{n,2^{k-1}}\right)
^{2}+\max_{2^{k-1}\leq n\leq 2^{k}}\left(
S_{n,2^{k-1}}-S_{2^{k-1},2^{k-1}}\right) ^{2} \\
&&+2\left( \max_{2^{k-1}\leq n\leq 2^{k}}\left( S_{n,n}-S_{n,2^{k-1}}\right)
\max_{2^{k-1}\leq n\leq 2^{k}}\left(
S_{n,2^{k-1}}-S_{2^{k-1},2^{k-1}}\right) \right)
\end{eqnarray*}%
where by Lemma \ref{Lemma_Stout_2.4.1} 
\begin{eqnarray}
E\left[ \max_{2^{k-1}\leq n\leq 2^{k}}\left( S_{n,n}-S_{n,2^{k-1}}\right)
^{2}\right] &\leq &\left( \frac{\log \left( 2^{k}\right) }{\log \left(
2\right) }\right) ^{2}g\left( F_{2^{k-1},2^{k-1}}\right)  \label{L2.4.2_D1}
\\
&\leq &K\left( \frac{\log \left( 2^{k}\right) }{\log \left( 2\right) }%
\right) ^{2}\frac{h\left( F_{2^{k-1},2^{k-1}}\right) }{\log ^{2}\left(
2^{k-1}+1\right) }.  \notag
\end{eqnarray}%
For 
\begin{eqnarray}
E\left[ \max_{2^{k-1}\leq n\leq 2^{k}}\left(
S_{n,2^{k-1}}-S_{2^{k-1},2^{k-1}}\right) ^{2}\right] &=&E\left[
\max_{2^{k-1}\leq n\leq 2^{k}}\left( \tsum\nolimits_{i=1}^{2^{k-1}}\left( 
\tilde{v}_{i,n}-\tilde{v}_{i,2^{k-1}}\right) \right) ^{2}\right]
\label{L2.4.2_D2} \\
&\leq &E\left[ \left( \tsum\nolimits_{i=1}^{2^{k-1}}\max_{2^{k-1}\leq n\leq
2^{k}}\left\vert \tilde{v}_{i,n}-\tilde{v}_{i,2^{k-1}}\right\vert \right)
^{2}\right]  \notag \\
&\leq &\frac{1}{2^{\left( k-1\right) }\log ^{2}\left( 2^{k-1}+1\right) }E%
\left[ \left( \tsum\nolimits_{i=1}^{2^{k-1}}\frac{u_{i,n}}{i\log ^{2}\left(
i+1\right) }\right) ^{2}\right]  \notag \\
&\leq &\frac{1}{2^{\left( k-1\right) }\log ^{2}\left( 2^{k-1}+1\right) }%
\left( \tsum\nolimits_{i=1}^{2^{k-1}}\left( \frac{E\left[ u_{i,n}^{2}\right] 
}{i\log \left( i+1\right) }\right) ^{1/2}\right) ^{2}  \notag \\
&=&O\left( \frac{1}{2^{\left( k-1\right) }\log ^{2}\left( 2^{k-1}+1\right) }%
\right)  \notag
\end{eqnarray}%
where, as before, the last inequality uses Stout (1974, p. 201). Finally,
using the Cauchy-Schwarz inequality gives 
\begin{eqnarray}
&&E\left[ \left( \max_{2^{k-1}\leq n\leq 2^{k}}\left(
S_{n,n}-S_{n,2^{k-1}}\right) \max_{2^{k-1}\leq n\leq 2^{k}}\left(
S_{n,2^{k-1}}-S_{2^{k-1},2^{k-1}}\right) \right) \right]  \label{L2.4.2_D3}
\\
&\leq &E\left[ \max_{2^{k-1}\leq n\leq 2^{k}}\left(
S_{n,n}-S_{n,2^{k-1}}\right) ^{2}\right] ^{1/2}E\left[ \max_{2^{k-1}\leq
n\leq 2^{k}}\left( S_{n,2^{k-1}}-S_{2^{k-1},2^{k-1}}\right) ^{2}\right]
^{1/2}  \notag \\
&=&O\left( \left( \frac{\log \left( 2^{k}\right) }{\log \left( 2\right) }%
\right) \frac{\left( h\left( F_{2^{k-1},2^{k-1}}\right) \right) ^{1/2}}{\log
\left( 2^{k-1}+1\right) }\frac{1}{2^{\left( k-1\right) /2}\log \left(
2^{k-1}+1\right) }\right)  \notag \\
&=&O\left( \frac{1}{\log ^{2}\left( 2^{k-1}+1\right) }\right) .  \notag
\end{eqnarray}%
It then follows from (\ref{L2.4.2_D1}), (\ref{L2.4.2_D2}) and (\ref%
{L2.4.2_D3}) that 
\begin{eqnarray*}
E\left[ \max_{2^{k-1}\leq n\leq 2^{k}}\left(
S_{n,n}-S_{2^{k-1},2^{k-1}}\right) ^{2}\right] &\leq &K\left( \frac{\log
\left( 2^{k}\right) }{\log \left( 2\right) }\right) ^{2}\frac{h\left(
F_{2^{k-1},2^{k-1}}\right) }{\log ^{2}\left( 2^{k-1}+1\right) } \\
&&+\frac{K}{2^{\left( k-1\right) }\log ^{2}\left( 2^{k-1}+1\right) }+\frac{K%
}{\log ^{2}\left( 2^{k-1}+1\right) }.
\end{eqnarray*}%
Since $\log \left( 2^{k-1}+1\right) ^{-2}$ is summable over $k$ the last two
terms are summable. Similarly, 
\begin{equation*}
\sum_{k=1}^{\infty }\left( \frac{\log \left( 2^{k}\right) }{\log \left(
2\right) }\right) ^{2}\frac{h\left( F_{2^{k-1},2^{k-1}}\right) }{\log
^{2}\left( 2^{k-1}+1\right) }<\infty
\end{equation*}%
by Stout (1974, p.27). This establishes (\ref{Dev_sum}) and completes the
proof.
\end{proof}

\begin{proof}[Proof of Theorem \protect\ref{Theorem_Maximal}]
Using Lemma \ref{Lemma_Stout_2.4.1} it remains to be shown that there exists
a function $g\left( .\right) $ such that $g\left( F_{a,k}\right) +g\left(
F_{a+k,m}\right) \leq g\left( F_{a,k+m}\right) $ for all $1\leq k<k+m$ and $%
a\geq 0$ and $E\left[ \left( \tsum\nolimits_{i=a+1}^{a+n}\tilde{v}%
_{i,l}\right) ^{2}\right] \leq g\left( F_{a,n}\right) $ for all $l\geq 1,$ $%
n\geq 1$and $a\geq 0.$ Using the bound in Lemma \ref{WeakLLN} it follows by
the same arguments as in the proof of Lemma \ref{WeakLLN} that 
\begin{equation}
E\left[ \left( \tsum\nolimits_{i=a+1}^{a+n}\tilde{v}_{i,l}\right) ^{2}\right]
\leq K\sum_{i,j=a+1}^{a+n}c_{i}c_{j}\sum_{m=0}^{\infty }E\left[ \psi
_{i,k_{m}}\left( \zeta \right) |A_{k_{m}}\left( i,j\right) \right] P\left(
A_{k_{m}}\left( i,j\right) \right) .  \label{Maximal_D1}
\end{equation}%
It is worth pointing out that the critical element in the bound in (\ref%
{Maximal_D1}) is the fact that the right hand side does not depned on $l.$
Since $c_{i}\geq 0,$ $E\left[ \psi _{i,k_{m}}\left( \zeta \right)
|A_{k_{m}}\left( i,j\right) \right] \geq 0$ and $P\left( A_{k_{m}}\left(
i,j\right) \right) \geq 0$ it follows that 
\begin{equation*}
g\left( F_{a,n}\right) =K\sum_{i,j=a+1}^{a+n}c_{i}c_{j}\sum_{m=0}^{\infty }E 
\left[ \psi _{i,k_{m}}\left( \zeta \right) |A_{k_{m}}\left( i,j\right) %
\right] P\left( A_{k_{m}}\left( i,j\right) \right)
\end{equation*}%
satisfies the required properties. The result then follows directly from
Lemma \ref{Lemma_Stout_2.4.1}.
\end{proof}

\begin{proof}[Proof of Theorem \protect\ref{Theorem_AlmostSure}]
Using Lemmas \ref{Lemma_Stout_2.4.1} and \ref{Lemma_Stout_2.4.2} it remains
to be shown that there exists a function $g\left( .\right) $ such that $%
g\left( F_{a,k}\right) +g\left( F_{a+k,l}\right) \leq g\left(
F_{a,k+l}\right) $ for all $1\leq k<k+l$ and $a\geq 0$ and $E\left[ \left(
\tsum\nolimits_{i=a+1}^{a+n}\tilde{v}_{i,l}\right) ^{2}\right] \leq g\left(
F_{a,n}\right) $ for all $l\geq 1,$ $n\geq 1$and $a\geq 0$ and a function $%
h\left( .\right) $ such that $h\left( F_{a,k}\right) +h\left(
F_{a+k,l}\right) \leq h\left( F_{a,k+l}\right) $ for all $1\leq k<k+l$ and $%
a\geq 0,$ $h\left( F_{a,n}\right) \leq K<\infty $ for all $n\geq 1$ and $%
a\geq 0,$ and $g\left( F_{a,n}\right) \leq Kh\left( F_{a,n}\right) /\log
^{2}\left( a+1\right) $ for all $n\geq 1$ and $a>0.$

Since $P\left( A_{k,n}\left( i,j\right) \right) =P\left( A_{k,n}\left(
j,i\right) \right) $ one can write 
\begin{eqnarray*}
&&\sum_{i,j=a+1}^{a+n}c_{i}c_{j}\sum_{m=0}^{\infty }E\left[ \psi
_{i,k_{m}}\left( \zeta \right) |A_{k_{m}}\left( i,j\right) \right] P\left(
A_{k_{m}}\left( i,j\right) \right) \\
&\leq &2K\sum_{i=a+1}^{a+n}c_{i}\sum_{j=i}^{a+n}c_{j}\sum_{m=0}^{\infty }E 
\left[ \psi _{i,k_{m}}\left( \zeta \right) |A_{k_{m}}\left( i,j\right) %
\right] P\left( A_{k_{m}}\left( i,j\right) \right) .
\end{eqnarray*}%
Let 
\begin{equation*}
\tilde{g}\left( i,a,n\right) =\sum_{j=i}^{a+n}c_{j}\sum_{m=0}^{\infty }E%
\left[ \psi _{i,k_{m}}\left( \zeta \right) |A_{k_{m}}\left( i,j\right) %
\right] P\left( A_{k_{m}}\left( i,j\right) \right) .
\end{equation*}%
It follows from $c_{j}\geq 0,$ $E\left[ \psi _{i,k_{m}}\left( \zeta \right)
|A_{k_{m}}\left( i,j\right) \right] \geq 0$ and $P\left( A_{k_{m}}\left(
i,j\right) \right) \geq 0$ that $\tilde{g}\left( i,a,n\right) $ is
increasing in $a$ and $n$ and $g\left( i,a+k,l\right) =g\left(
i,a,l+k\right) .$ Now choose $g\left( F_{a,n}\right)
=2\sum_{i=a+1}^{a+n}c_{i}\tilde{g}\left( i,a,n\right) $ such that 
\begin{eqnarray}
g\left( F_{a,k}\right) +g\left( F_{a+k,l}\right) &=&2K\left(
\sum_{i=a+1}^{a+k}c_{i}\tilde{g}\left( i,a,k\right)
+\sum_{i=a+k+1}^{a+l}c_{i}\tilde{g}\left( i,a,l\right) \right)
\label{Th_AS_D1} \\
&\leq &2K\left( \sum_{i=a+1}^{a+k}c_{i}\tilde{g}\left( i,a,k+l\right)
+\sum_{i=a+k+1}^{a+l}c_{i}\tilde{g}\left( i,a,k+l\right) \right)  \notag \\
&\leq &2K\sum_{i=a+1}^{a+k+l}c_{i}\tilde{g}\left( i,a,k+l\right) =g\left(
F_{a,k+l}\right) .  \notag
\end{eqnarray}%
By the proof of Theorem \ref{Theorem_Maximal} it follows that $E\left[
\left( \tsum\nolimits_{i=a+1}^{a+n}\tilde{v}_{i,h}\right) ^{2}\right] \leq
g\left( F_{a,n}\right) $ for all $h\geq 1,$ $n\geq 1$and $a\geq 0.$ Now
choose $h\left( .\right) $ as 
\begin{equation*}
h\left( F_{a,n}\right) =2K\sum_{i=a+1}^{a+n}c_{i}\log ^{2}\left( i\right) 
\tilde{g}\left( i,a,n\right) .
\end{equation*}%
Under the conditions of the Theorem in (\ref{Th_AS_Cond1}) it follows that $%
h\left( F_{a,n}\right) \leq K<\infty $ for some $K.$ By the same logic as in
(\ref{Th_AS_D1}) it follows that $h\left( F_{a,k}\right) +h\left(
F_{a+k,l}\right) \leq h\left( F_{a,k+l}\right) $ for all $1\leq k<k+l$ and $%
a\geq 0.$ Finally, since $\log ^{2}\left( i\right) /\log ^{2}\left(
a+1\right) \geq 1$ for $i\geq a+1$ it follows that $g\left( F_{a,n}\right)
\leq Kh\left( F_{a,n}\right) /\log ^{2}\left( a+1\right) .$ Then, the result
follows from Lemma \ref{Lemma_Stout_2.4.2}.
\end{proof}

\begin{proof}[Proof of Theorem \protect\ref{Theorem_Stout_3.7.1}]
The proof is based on establishing the conditions of Theorem 3.7.1. in Stout
(1974). Note that Lemma \ref{WeakLLN} and in particular (\ref{Cov_Bound})
imply 
\begin{equation*}
\limfunc{Cov}\left( \frac{v_{i,n}\left( \zeta \right) }{i},\frac{%
v_{j,n}\left( \zeta \right) }{j}\right) \leq 2K\frac{c_{i}c_{j}}{ij}%
\sum_{m=0}^{\infty }E\left[ \psi _{i,k_{m}}\left( \zeta \right)
|A_{k_{m}}\left( i,j\right) \right] P\left( A_{k_{m}}\left( i,j\right)
\right) .
\end{equation*}%
Then, by the same arguments as in the proof of Lemma \ref{WeakLLN} it
follows that for any $a\geq 0,$ 
\begin{eqnarray*}
&&E\left[ \left( \tsum\nolimits_{i=a+1}^{a+n}\frac{\left( v_{i,n}-\mu
_{i,n}\right) }{i}\right) ^{2}\right] \\
&\leq &2K\sum_{i,j=a+1}^{a+n}\frac{c_{i}c_{j}}{ij}\sum_{m=0}^{\infty }E\left[
\psi _{i,k_{m}}\left( \zeta \right) |A_{k_{m}}\left( i,j\right) \right]
P\left( A_{k_{m}}\left( i,j\right) \right) \\
&\leq &4K^{3}\sum_{i=a+1}^{a+n}\frac{1}{i^{2}}\sum_{m=0}^{\infty
}\sum_{j=i}^{a+n}E\left[ \psi _{i,k_{m}}\left( \zeta \right)
|A_{k_{m}}\left( i,j\right) \right] P\left( A_{k_{m}}\left( i,j\right)
\right) .
\end{eqnarray*}%
Choose the function $g\left( F_{a,n}\right)
=\sum_{i=a+1}^{a+n}i^{-2}\sum_{m=0}^{\infty }\sum_{j=i}^{a+n}E\left[ \psi
_{i,k_{m}}\left( \zeta \right) |A_{k_{m}}\left( i,j\right) \right] P\left(
A_{k_{m}}\left( i,j\right) \right) .$ Because $P\left( A_{k_{m}}\left(
i,j\right) \right) \geq 0$ and $E\left[ \psi _{i,k_{m}}\left( \zeta \right)
|A_{k_{m}}\left( i,j\right) \right] \geq 0$ it follows that $g\left(
.\right) $ satisfies (\ref{S2.4.1}). As in Corollary 2.4.1 of Stout (1974),
choose $h\left( F_{a,n}\right) $ as%
\begin{equation*}
h\left( F_{a,n}\right) =\sum_{i=a+1}^{a+n}\frac{\log ^{2}\left( i\right) }{%
i^{2}}\sum_{j=i}^{a+n}\sum_{m=0}^{\infty }E\left[ \psi _{i,k_{m}}\left(
\zeta \right) |A_{k_{m}}\left( i,j\right) \right] P\left( A_{k_{m}}\left(
i,j\right) \right) .
\end{equation*}%
By Assumption \ref{Assume_Probability_Sum} it follows that $h\left(
F_{a,n}\right) \leq K$ for all $n\geq 1$ and $a\geq 0.$ By the same argument
as for $g\left( .\right) $ it also follows that 
\begin{equation*}
h\left( F_{a,k}\right) +h\left( F_{a+k,m}\right) \leq h\left(
F_{a,k+m}\right)
\end{equation*}%
for all $1\leq k<k+m$ and $a\geq 0.$ Finally, it is obvious that $g\left(
F_{a,n}\right) \leq Kh\left( F_{a,n}\right) /\log ^{2}\left( a+1\right) $
for all $n\geq 1,$ $a\geq 0.$ Then, it follows from Lemmas \ref%
{Lemma_Stout_2.4.1} and \ref{Lemma_Stout_2.4.2} and Theorem 3.7.1 in Stout
(1974) that $n^{-1}S_{n}\rightarrow 0$ almost surely.
\end{proof}

\subsection{Results and Proofs for Section \protect\ref{Section_CLT}\label%
{Section_CLT_Proofs}}

First, a lemma required in the proof of the CLT is presented.

\begin{lemma}
\label{Lemma_Suff_1}Let $S_{n,x}=n^{-1/2}\sum_{i=1}^{N}X_{i,n}$ with $%
\left\{ X_{i,n},\mathcal{F}_{n}^{i}\right\} _{i=1}^{N}$ as defined in (\ref%
{Def_Xi}) and (\ref{Def_Fi}). Further assume that $\sup_{i}E\left[
\left\vert v_{i,n}\left( \zeta \right) \right\vert ^{2+\delta }|\zeta \right]
<\infty $ a.s. Then, 
\begin{equation*}
\max_{i}\left\vert n^{-1/2}X_{i,n}\right\vert \rightarrow _{p}0
\end{equation*}%
and 
\begin{equation*}
E\left[ n^{-1}\max_{i}\left\vert X_{i,n}^{2}\right\vert \right] \text{ is
bounded in }n.
\end{equation*}
\end{lemma}

\begin{proof}[Proof of Lemma \protect\ref{Lemma_Suff_1}]
Clearly, $\max_{i}\left\vert X_{i,n}^{2}\right\vert \leq
\sum_{i=1}^{N}X_{i,n}^{2}.$ Then, $E\left[ n^{-1}\max_{i}\left\vert
X_{i,n}^{2}\right\vert \right] $ is bounded if $U_{n}^{2}=n^{-1}%
\sum_{i=1}^{N}X_{i,n}^{2}$ is uniformly integrable. Consider 
\begin{eqnarray}
E\left[ U_{n}^{2}1\left\{ U_{n}>\varepsilon \right\} \right] &=&E\left[
\tsum\nolimits_{i=1}^{N}\left( n^{-1/2}X_{i,n}\right) ^{2}1\left\{
\tsum\nolimits_{i=1}^{N}X_{i,n}^{2}>\varepsilon n\right\} \right]
\label{D_L_Suff_1-1} \\
&\leq &\frac{1}{\varepsilon ^{\delta /2}}E\left[ \left\vert
\tsum\nolimits_{i=1}^{N}\left( n^{-1/2}X_{i,n}\right) ^{2}\right\vert
^{1+\delta /2}\right]  \notag \\
&\leq &\frac{1}{\varepsilon ^{\delta /2}}E\left[ \tsum\nolimits_{i=1}^{N}E%
\left[ \left\vert n^{-1/2}X_{i,n}\right\vert ^{2+\delta }|\zeta \right]
^{1/\left( 1+\delta /2\right) }\right] ^{1+\delta /2}  \notag
\end{eqnarray}%
where the last inequality follows from the law of iterated expectations and
Stout (1974, p. 201). By the same inequality one obtains from H\"{o}lder's
inequality that $\left\vert v_{j,n}\left( \zeta \right) -\mu
_{j,n}\right\vert ^{2+\delta }\leq 2^{1+\delta }\left( \left\vert
v_{j,n}\left( \zeta \right) \right\vert ^{2+\delta }+\left\vert \mu
_{i,n}\right\vert ^{2+\delta }\right) $ such that 
\begin{eqnarray}
E\left[ \left\vert n^{-1/2}X_{i,n}\right\vert ^{2+\delta }|\zeta \right]
&\leq &\left( n^{-\left( 1+\delta /2\right) }\sum_{j\in J_{k}\left(
q_{i}\right) }E\left[ \left\vert v_{j,n}\left( \zeta \right) -\mu
_{j,n}\right\vert ^{2+\delta }|\zeta \right] ^{1/\left( 2+\delta \right)
}\right) ^{2+\delta }  \label{D_L_Suff_1-2} \\
&\leq &\left( n^{-\left( 1+\delta /2\right) }\sum_{j\in J_{k}\left(
q_{i}\right) }\left( 2^{1+\delta }\left( E\left[ \left\vert v_{j,n}\left(
\zeta \right) \right\vert ^{2+\delta }|\zeta \right] +\left\vert \mu
_{i,n}\right\vert ^{2+\delta }\right) \right) ^{1/\left( 2+\delta \right)
}\right) ^{2+\delta }  \notag \\
&\leq &Kn^{-\left( 1+\delta /2\right) \left( 2+\delta \right) }\left(
\left\vert J_{k}\left( q_{i}\right) \right\vert \right) ^{2+\delta }.  \notag
\end{eqnarray}%
Substituting back in (\ref{D_L_Suff_1-1}) gives%
\begin{eqnarray*}
E\left[ U_{n}^{2}1\left\{ U_{n}>\varepsilon \right\} \right] &\leq &\frac{1}{%
\varepsilon ^{\delta /2}}Kn^{-\left( 2+\delta \right) }\left( E\left[
\tsum\nolimits_{i=1}^{N}\left( \left\vert J_{k}\left( q_{i}\right)
\right\vert \right) ^{\frac{2+\delta }{1+\delta /2}}\right] \right)
^{1+\delta /2} \\
&\leq &\frac{1}{\varepsilon ^{\delta /2}}Kn^{-\left( 2+\delta \right)
}\left( E\left[ \left( \tsum\nolimits_{i=1}^{N}\left\vert J_{k}\left(
q_{i}\right) \right\vert \right) ^{\frac{2+\delta }{1+\delta /2}}\right]
\right) ^{1+\delta /2}
\end{eqnarray*}%
Since $\tsum\nolimits_{i=1}^{N}\left\vert J_{k}\left( q_{i}\right)
\right\vert \leq n$ for all $N\geq 1$ it follows that $\tsum%
\nolimits_{i=1}^{N}\left( \left\vert J_{k}\left( q_{i}\right) \right\vert
\right) ^{\frac{2+\delta }{1+\delta /2}}\leq \left(
\tsum\nolimits_{i=1}^{N}\left\vert J_{k}\left( q_{i}\right) \right\vert
\right) ^{\frac{2+\delta }{1+\delta /2}}\leq n^{\frac{2+\delta }{1+\delta /2}%
}.$ This implies that 
\begin{equation*}
E\left[ U_{n}^{2}1\left\{ U_{n}>\varepsilon \right\} \right] \leq \frac{1}{%
\varepsilon ^{\delta /2}}K\rightarrow 0\text{ as }\varepsilon \rightarrow
\infty
\end{equation*}%
which establishes that $U_{n}^{2}$ is uniformly integrable. This proofs the
first claim of the Lemma.

By Hall and Heyde (1980, p.53) $\max_{i}\left\vert
n^{-1/2}X_{i,n}\right\vert \rightarrow _{p}0$ follows from $\varepsilon >0$
and%
\begin{eqnarray*}
P\left( \max_{i}\left\vert n^{-1/2}X_{i,n}\right\vert >\varepsilon
^{1/2}\right) &=&P\left( \tsum\nolimits_{i=1}^{N}\left(
n^{-1/2}X_{i,n}\right) ^{2}1\left\{ \left\vert X_{i,n}\right\vert
>\varepsilon ^{1/2}n^{1/2}\right\} >\varepsilon \right) \\
&\leq &\frac{1}{\varepsilon }E\left[ \tsum\nolimits_{i=1}^{N}\left(
n^{-1/2}X_{i,n}\right) ^{2}1\left\{ \left\vert X_{i,n}\right\vert
>\varepsilon ^{1/2}n^{1/2}\right\} \right] \\
&\leq &\frac{1}{\varepsilon ^{1+\delta /2}}E\left[ \tsum\nolimits_{i=1}^{N}E%
\left[ \left\vert n^{-1/2}X_{i,n}\right\vert ^{2+\delta }|\zeta \right] %
\right] \\
&\leq &Kn^{-\left( 1+\delta /2\right) \left( 2+\delta \right) }E\left[
\tsum\nolimits_{i=1}^{N}\left( \left\vert J_{k}\left( q_{i}\right)
\right\vert \right) ^{2+\delta }\right] \\
&\leq &Kn^{-\delta \left( 2+\delta \right) /2}\rightarrow 0
\end{eqnarray*}%
which establishes the second claim.
\end{proof}

\begin{lemma}
\label{Lemma_Mix_Ineq}Let $\left\{ X_{i,n},\mathcal{F}_{n}^{i}\right\}
_{i=1}^{N}$ as defined in (\ref{Def_Xi}) and (\ref{Def_Fi}). Further assume
that $\sup_{i}E\left[ \left\vert v_{i,n}\left( \zeta \right) \right\vert
^{2+\delta }|\zeta \right] <\infty $ a.s. and that $\sup_{i}\psi
_{i,h}\left( \zeta \right) \leq \psi _{h}\left( \zeta \right) .$ Then, 
\begin{equation*}
E\left[ \left\Vert E\left[ X_{i,n}|\mathcal{F}_{n}^{i-1}\right] \right\Vert
_{2,\zeta }\right] \leq \sup_{j}\left\vert c_{j}\right\vert E\left[ \psi
_{h^{\prime }}\left( \zeta \right) ^{2}\left\vert J_{k_{n}^{i}}\left(
q_{i}\right) \right\vert \right] .
\end{equation*}
\end{lemma}

\begin{proof}
First note that because $J_{k}\left( q_{i}\right) $ is measurable with
respect to $\mathcal{Z}$ and using the triangular and Jensen's inequalities
it follows that 
\begin{eqnarray*}
E\left[ \left\Vert E\left[ X_{i,n}|\mathcal{F}_{n}^{i-1}\right] \right\Vert
_{2,\zeta }^{2}\right] &=&E\left[ \left\Vert \sum_{j\in J_{k_{n}^{i}}\left(
q_{i}\right) }\left( \mu _{j,n}-E\left[ v_{j,n}\left( \zeta \right) |%
\mathcal{F}_{n}^{i-1}\right] \right) \right\Vert _{2,\zeta }^{2}\right] \\
&\leq &E\left[ \sum_{j\in J_{k_{n}^{i}}\left( q_{i}\right) }\left\Vert E%
\left[ \left\vert \mu _{j,n}-E\left[ v_{j,n}\left( \zeta \right) |\mathcal{B}%
_{j,n}^{h^{\prime }}\right] \right\vert ^{2}|\mathcal{F}_{n}^{i-1}\right]
^{1/2}\right\Vert _{2,\zeta }^{2}\right] .
\end{eqnarray*}%
Now use the definition of $\left\Vert .\right\Vert _{2,\zeta }^{2}$ and
iterated expectations to conclude that 
\begin{eqnarray*}
&&E\left[ \sum_{j\in J_{k_{n}^{i}}\left( q_{i}\right) }\left\Vert E\left[
\left\vert \mu _{j,n}-E\left[ v_{j,n}\left( \zeta \right) |\mathcal{B}%
_{j,n}^{h^{\prime }}\right] \right\vert ^{2}|\mathcal{F}_{n}^{i-1}\right]
^{1/2}\right\Vert _{2,\zeta }^{2}\right] \\
&=&E\left[ \sum_{j\in J_{k_{n}^{i}}\left( q_{i}\right) }E\left[ \left\vert
\mu _{j,n}-E\left[ v_{j,n}\left( \zeta \right) |\mathcal{B}_{j,n}^{h^{\prime
}}\right] \right\vert ^{2}|\mathcal{F}_{n}^{i-1}\right] \right]
\end{eqnarray*}%
Repeated use of iterated expectations gives 
\begin{eqnarray*}
E\left[ \sum_{j\in J_{k_{n}^{i}}\left( q_{i}\right) }E\left[ \left\vert \mu
_{j,n}-E\left[ v_{j,n}\left( \zeta \right) |\mathcal{B}_{j,n}^{h^{\prime }}%
\right] \right\vert ^{2}|\mathcal{F}_{n}^{i-1}\right] \right] &=&E\left[
\sum_{j\in J_{k_{n}^{i}}\left( q_{i}\right) }\left\vert \mu _{j,n}-E\left[
v_{j,n}\left( \zeta \right) |\mathcal{B}_{j,n}^{h^{\prime }}\right]
\right\vert ^{2}\right] \\
&=&E\left[ \sum_{j\in J_{k_{n}^{i}}\left( q_{i}\right) }\left\Vert \mu
_{j,n}-E\left[ v_{j,n}\left( \zeta \right) |\mathcal{B}_{j,n}^{h^{\prime }}%
\right] \right\Vert _{2,\zeta }^{2}\right] \\
&\leq &E\left[ \sum_{j\in J_{k_{n}^{i}}\left( q_{i}\right) }c_{j}^{2}\psi
_{j,h^{\prime }}\left( \zeta \right) ^{2}\right] \\
&\leq &\sup c_{i}^{2}E\left[ \psi _{h^{\prime }}\left( \zeta \right)
^{2}\left\vert J_{k_{n}^{i}}\left( q_{i}\right) \right\vert \right] .
\end{eqnarray*}%
where the first inequality uses (\ref{Mix1}) and the second uses the fact
that $c_{j}$ is abounded constant as well as
\end{proof}

\begin{proof}[Proof of Proposition \protect\ref{Theorem_CLT_Blocks}]
The proof follows Hall and Heyde (1980, Theorem 3.2) as well as the
modifications to their proof in Kuersteiner and Prucha (2013). First recall
that by the conditions of the theorem 
\begin{equation}
\max_{i}\left\vert n^{-1/2}X_{i,n}\right\vert \rightarrow _{p}0,
\label{3.18}
\end{equation}%
\begin{equation}
n^{-1}\sum_{i=1}^{N}X_{i,n}^{2}\rightarrow _{p}\eta ^{2},  \label{3.19}
\end{equation}%
and 
\begin{equation}
E\left( n^{-1}\max_{i}\left\vert X_{i,n}^{2}\right\vert \right) \text{ is
bounded in }n.  \label{3.20}
\end{equation}%
Suppose that $\eta ^{2}$ is a.s. bounded such that for some $C>1$, 
\begin{equation}
P\left( \eta ^{2}<C\right) =1.  \label{3.22}
\end{equation}%
Define $X_{i,n}^{\dag }=X_{i,n}\mathbf{1}\left\{
n^{-1}\sum_{j=1}^{i-1}X_{j,n}^{2}\leq 2C\right\} $ with $X_{1,n}^{\dag
}=X_{1,n}$, $S_{i,n}=n^{-1/2}\sum_{j=1}^{i}X_{j,n}$ and $S_{i,n}^{\dag
}=n^{-1/2}\sum_{j=1}^{i}X_{j,n}^{\dag }$ for $1\leq i\leq N$.

Clearly for any $j\leq i$ the random variable $X_{j,n}$ is measurable w.r.t.
to $\mathcal{F}_{n}^{i}$, since $\mathcal{F}_{n}^{j}\subseteq \mathcal{F}%
_{n}^{i}$. Since the random variables $X_{n1},\ldots ,X_{ni}$ are measurable
w.r.t. $\mathcal{F}_{n}^{i}$, $S_{i,n}^{\dag }$ is measurable w.r.t. $%
\mathcal{F}_{n}^{i}$. Also, since $\left\vert S_{i,n}^{\dag }\right\vert
\leq $ $\left\vert S_{i,n}\right\vert $ it follows that $E\left[
S_{i,n}^{\dag 2}\right] \leq $ $E\left[ S_{i,n}^{2}\right] <\infty $.
Furthermore for $1\leq j\leq i,$ $E\left[ S_{j,n}^{\dag }|\mathcal{F}_{n}^{i}%
\right] =S_{i,n}^{\dag }.$ Use the notation $\mathcal{I}_{i,C}=\mathbf{1}%
\left\{ n^{-1}\tsum\nolimits_{j=1}^{i-1}X_{j,n}^{2}\leq 2C\right\} .$ It
follows immediately that for $k\geq 0$ 
\begin{equation*}
\left\Vert X_{i,n}^{\dag }-E\left[ X_{i,n}^{\dag }|\mathcal{F}_{n}^{i+k}%
\right] \right\Vert _{2,\zeta }=0.
\end{equation*}%
By construction, the distance between any element of $X_{i,n}$ and any
element in $\mathcal{F}_{n}^{i-1}$ measured in terms of $g\left( .\right) $
is at least $g_{ij}\left( \zeta \right) \leq \left( \Lambda \left(
h_{n}^{i-1}\right) ^{-1}-\Lambda \left( k_{n}^{i-1}\right) ^{-1}\right)
^{-1}.$ Since $\Lambda $ is monotonically decreasing in its argument it has
an inverse $\Lambda ^{-1}.$ Then, for $h^{\prime }$ such that $\Lambda
^{-1}\left( 1/\Lambda \left( h_{n}^{i-1}\right) -1/\Lambda \left(
k_{n}^{i-1}\right) \right) ^{-1}=h_{i}^{\prime }$ it follows that $\mathcal{B%
}_{j,n}^{h_{i}^{\prime }}\supseteq \mathcal{F}_{n}^{i-1}$ for all $j\in
J_{k}\left( q_{i}\right) .$ Then,%
\begin{eqnarray}
E\left[ \left\Vert E\left[ X_{i,n}^{\dag }|\mathcal{F}_{n}^{i-1}\right]
\right\Vert _{2,\zeta }^{2}\right]  &\leq &E\left[ \left\Vert \sum_{j\in
J_{k}\left( q_{i}\right) }\left( E\left[ \left( \mu _{j,n}-v_{j,n}\left(
\zeta \right) \right) \mathcal{I}_{i,C}|\mathcal{F}_{n}^{i-1}\right] \right)
\right\Vert _{2,\zeta }^{2}\right]   \label{X_dag_Mix} \\
&\leq &E\left[ \sum_{j\in J_{k}\left( q_{i}\right) }\left\Vert \left( E\left[
\left\vert \mu _{j,n}-E\left[ v_{j,n}\left( \zeta \right) |\mathcal{B}%
_{j,n}^{h^{\prime }}\right] \right\vert ^{2}|\mathcal{F}_{n}^{i-1}\right] 
\mathcal{I}_{i,C}\right) ^{1/2}\right\Vert _{2,\zeta }^{2}\right]   \notag \\
&\leq &\sup_{i}c_{i}E\left[ \psi _{h_{i}^{\prime }}\left( \zeta \right)
^{2}\left\vert J_{k}\left( q_{i}\right) \right\vert \right]   \notag
\end{eqnarray}%
where the second inequality uses the fact that $\mathbf{1}\left\{
n^{-1}\tsum\nolimits_{j=1}^{i-1}X_{j,n}^{2}\leq 2C\right\} $ is measurable
with respect to $\mathcal{F}_{n}^{i-1}$ and the last inequality uses the
fact that $\mathbf{1}\left\{ .\right\} \leq 1$ and uses the same argument as
in the proof of Lemma \ref{Lemma_Mix_Ineq}.

Next let $U_{N}^{2}=\sum_{i=1}^{N}X_{i,n}^{2}$, then clearly $%
P(U_{N}^{2}>2C)\rightarrow 0$ in light of (\ref{3.19}). Consequently%
\begin{equation}
P(X_{i,n}^{\dag }\neq X_{i,n}\text{ for some }i\leq N)\leq
P(U_{N}^{2}>2C)\rightarrow 0,  \label{3.23}
\end{equation}%
which in turn implies $P(S_{N,N}^{\dag }\neq S_{N,N})\rightarrow 0$, and
furthermore%
\begin{equation*}
E\left[ \left\vert \varsigma \exp (itS_{N,N}^{\dag })-\varsigma \exp
(itS_{N,N})\right\vert \right] \rightarrow 0
\end{equation*}%
for any $P$-essentially bounded and $\mathcal{C}$-measurable random variable 
$\varsigma $. Consequently, $S_{N,N}\overset{d}{\mathbf{\rightarrow }}Z$ ($%
\mathcal{C}$-stably) iff $S_{N,N}^{\dag }\overset{d}{\mathbf{\rightarrow }}Z$
($\mathcal{C}$-stably). Observe furthermore that in view of (\ref{3.23}) the
sequence $\{X_{i,n}^{\dag }\}$ satisfy that $\max_{i}\left\vert
n^{-1/2}X_{i,n}^{\dag }\right\vert \overset{p}{\rightarrow }0$ and $%
n^{-1}\sum_{i=1}^{N}X_{i,n}^{\dag 2}\overset{p}{\rightarrow }\eta ^{2}$.
Since $\left\vert X_{i,n}^{\dag }\right\vert \leq \left\vert
X_{i,n}\right\vert $ condition (\ref{3.20}) implies that $E\left[
n^{-1}\max_{i}X_{i,n}^{\dag 2}\right] $ is bounded in $n$.

Now show that $S_{N,N}^{\dag }\overset{d}{\mathbf{\rightarrow }}Z$ ($%
\mathcal{C}$-stably). Let $U_{i,n}^{2}=\sum_{j=1}^{i}X_{j,n}^{2}$ and $%
T_{n}^{\dag }\left( t\right) =\tprod\nolimits_{j=1}^{N}\left(
1+itX_{j,n}^{\dag }\right) $ with 
\begin{equation*}
J_{n}=\left\{ 
\begin{array}{cc}
\min \left\{ i\leq N|U_{i,n}^{2}>2C\right\} & \text{if }U_{N,N}^{2}>2C \\ 
N & \text{otherwise}%
\end{array}%
\right. .
\end{equation*}%
Observing that $X_{j,n}^{\dag }=0$ for $j>$ $J_{n}$, and that for any real
number $a$ we have $\left\vert 1+ia\right\vert ^{2}=(1+a^{2})$ and $\exp
(a^{2})\geq 1+a^{2}$, it follows that 
\begin{eqnarray}
E\left[ \left\vert T_{n}^{\dag }\left( t\right) \right\vert ^{2}\right] &=&E%
\left[ \tprod\nolimits_{j=1}^{N}\left( 1+t^{2}X_{j,n}^{\dag 2}\right) \right]
\label{T_dag_bound} \\
&\leq &E\left[ \left\{ \exp \left( t^{2}\sum_{j=1}^{J_{n}-1}X_{j,n}^{\dag
2}\right) \left( 1+t^{2}X_{J_{n},n}^{\dag 2}\right) \right\} \right]  \notag
\\
&\leq &\left\{ \exp (2Ct^{2})\right\} \left( 1+t^{2}E\left[
X_{J_{n},n}^{\dag 2}\right] \right) .  \notag
\end{eqnarray}%
Since $E\left[ X_{J_{n},n}^{\dag 2}\right] \leq E\left[ X_{J_{n},n}^{2}%
\right] $ is uniformly bounded it follows from the above inequality that $E%
\left[ \left\vert T_{n}^{\dag }\left( t\right) \right\vert ^{2}\right] $ is
uniformly bounded in $n$.

Now define $I_{n}=\exp \left( itS_{N,N}^{\dag }\right) $ and $W_{n}=\exp
\left( -\frac{1}{2}t^{2}\sum_{i=1}^{N}X_{i,n}^{\dag 2}+\sum_{i=1}^{N}r\left(
tX_{i,n}^{\dag }\right) \right) $ where $r\left( .\right) $ is implicitly
defined by $e^{ix}=\left( 1+ix\right) \exp \left( -\frac{1}{2}x^{2}+r\left(
x\right) \right) $ as in Hall and Heyde (1980), p. 57. Then 
\begin{equation}
I_{n}=T_{n}^{\dag }(t)\exp \left( -\eta ^{2}t^{2}/2\right) +T_{n}^{\dag
}(t)(W_{n}-\exp \left( -\eta ^{2}t^{2}/2\right) ).  \label{HH1}
\end{equation}%
For $S_{nk_{n}}^{\dag }\overset{d}{\mathbf{\rightarrow }}Z$ ($\mathcal{C}$
stably) it is enough to show that 
\begin{equation}
E\left( I_{n}\varsigma \right) \rightarrow E\left[ \exp \left( -\eta
^{2}t^{2}/2\right) \varsigma \right]  \label{HH2}
\end{equation}%
for any $P$-essentially bounded $\mathcal{C}$-measurable random variable $%
\varsigma $. Because $\mathcal{F}_{n}^{0}\subset \mathcal{F}_{n}^{i}$ it
follows that $\exp \left( -\eta ^{2}t^{2}/2\right) \varsigma $ is $\mathcal{F%
}_{n}^{i}$-measurable for all $n$ and $i\leq N$. Hence,%
\begin{eqnarray*}
E\left[ T_{n}^{\dag }\left( t\right) \exp \left( -\eta ^{2}t^{2}/2\right)
\varsigma \right] &=&E\left[ \exp \left( -\eta ^{2}t^{2}/2\right) \varsigma
\left( \tprod\nolimits_{j}^{N}\left( 1+itX_{j,n}^{\dag }\right) -1\right) %
\right] \\
&=&E\left\{ E\left[ \exp \left( -\eta ^{2}t^{2}/2\right) \varsigma
\tprod\nolimits_{j}^{N}\left( 1+itX_{nj}^{\dag }\right) |\mathcal{F}%
_{n}^{N-1}\right] \right\} \\
&=&E\left\{ \exp \left( -\eta ^{2}t^{2}/2\right) \varsigma
\tprod\nolimits_{j}^{N-1}\left( 1+itX_{nj}^{\dag }\right) E\left[ \left(
1+itX_{N,n}^{\dag }\right) |\mathcal{F}_{n}^{N-1}\right] \right\} \\
&=&E\left\{ \exp \left( -\eta ^{2}t^{2}/2\right) \varsigma
\tprod\nolimits_{j}^{N-1}\left( 1+itX_{nj}^{\dag }\right) \right\} \\
&&+E\left\{ \exp \left( -\eta ^{2}t^{2}/2\right) \varsigma
\tprod\nolimits_{j}^{N-1}\left( 1+itX_{nj}^{\dag }\right) E\left[
itX_{N,n}^{\dag }|\mathcal{F}_{n}^{N-1}\right] \right\}
\end{eqnarray*}%
where 
\begin{eqnarray*}
&&\left\vert E\left\{ \exp \left( -\eta ^{2}t^{2}/2\right) \varsigma
\tprod\nolimits_{j}^{N-1}\left( 1+itX_{nj}^{\dag }\right) E\left[
itX_{N,n}^{\dag }|\mathcal{F}_{n}^{N-1}\right] \right\} \right\vert \\
&\leq &E\left[ \left\vert \varsigma \right\vert ^{2}\left\vert
\tprod\nolimits_{j}^{N-1}\left( 1+itX_{nj}^{\dag }\right) \right\vert ^{2}%
\right] ^{1/2}E\left[ \left\vert E\left[ itX_{N,n}^{\dag }|\mathcal{F}%
_{n}^{N-1}\right] \right\vert ^{2}\right] \\
&=&E\left[ \left\vert \varsigma \right\vert ^{2}\left\vert
\tprod\nolimits_{j}^{N-1}\left( 1+itX_{nj}^{\dag }\right) \right\vert ^{2}%
\right] ^{1/2}E\left[ \left\Vert E\left[ X_{i,n}^{\dag }|\mathcal{F}%
_{n}^{N-1}\right] \right\Vert _{2,\zeta }^{2}\right] \\
&\leq &K^{2}\sup_{i}c_{i}^{2}E\left[ \left\vert J_{k}\left( q_{N}\right)
\right\vert \psi _{h_{N}^{\prime }}\left( \zeta \right) ^{2}\right]
\end{eqnarray*}%
where $P\left( \left\vert \varsigma \right\vert <K\right) =1$ for some $K,$
the fact that $E\left[ \left\vert \tprod\nolimits_{j}^{N-1}\left(
1+itX_{nj}^{\dag }\right) \right\vert ^{2}\right] $ is bounded by (\ref%
{T_dag_bound}) and $E\left[ \left\Vert E\left[ X_{i,n}^{\dag }|\mathcal{F}%
_{n}^{i-1}\right] \right\Vert _{2,\zeta }^{2}\right] <\sup_{i}c_{i}^{2}E%
\left[ \left\vert J_{k}\left( q_{N}\right) \right\vert \psi _{h^{\prime
}}\left( \zeta \right) ^{2}\right] $ by (\ref{X_dag_Mix}). By the same
arguments, 
\begin{eqnarray*}
&&E\left\{ \exp \left( -\eta ^{2}t^{2}/2\right) \varsigma
\tprod\nolimits_{j}^{N-1}\left( 1+itX_{nj}^{\dag }\right) \right\} \\
&=&E\left\{ \exp \left( -\eta ^{2}t^{2}/2\right) \varsigma
\tprod\nolimits_{j}^{N-2}\left( 1+itX_{nj}^{\dag }\right) \right\} \\
&&+E\left\{ \exp \left( -\eta ^{2}t^{2}/2\right) \varsigma
\tprod\nolimits_{j}^{N-2}\left( 1+itX_{nj}^{\dag }\right) E\left[
itX_{N-1,n}^{\dag }|\mathcal{F}_{n}^{N-2}\right] \right\}
\end{eqnarray*}%
where 
\begin{equation*}
\left\vert E\left\{ \exp \left( -\eta ^{2}t^{2}/2\right) \varsigma
\tprod\nolimits_{j}^{N-2}\left( 1+itX_{nj}^{\dag }\right) E\left[
itX_{N-1,n}^{\dag }|\mathcal{F}_{n}^{N-2}\right] \right\} \right\vert \leq
\sup_{i}c_{i}^{2}K^{2}E\left[ \left\vert J_{k}\left( q_{N-1}\right)
\right\vert \psi _{h^{\prime }}\left( \zeta \right) \right] .
\end{equation*}%
Continuing the recursion it follows that for $h^{\prime }=\max_{i\leq
N}h_{i}^{\prime }$ 
\begin{eqnarray*}
\left\vert E\left[ T_{n}^{\dag }\left( t\right) \exp \left( -\eta
^{2}t^{2}/2\right) \varsigma \right] -E\left[ \exp \left( -\eta
^{2}t^{2}/2\right) \varsigma \right] \right\vert &\leq
&K^{2}\sup_{i}c_{i}^{2}E\left[ \sum_{i=1}^{N}\left\vert J_{k}\left(
q_{i}\right) \right\vert \psi _{h_{i}^{\prime }}\left( \zeta \right) ^{2}%
\right] \\
&\leq &K^{2}\sup_{i}c_{i}^{2}E\left[ \psi _{h^{\prime }}\left( \zeta \right)
^{2}\sum_{i=1}^{N}\left\vert J_{k}\left( q_{i}\right) \right\vert \right] \\
&\leq &n\sup_{i}c_{i}K^{2}E\left[ \psi _{h^{\prime }}\left( \zeta \right)
^{2}\right] =O\left( n^{-\delta }\right)
\end{eqnarray*}%
since $\sum_{i=1}^{N}\left\vert J_{k}\left( q_{i}\right) \right\vert \leq n$
and by Condition (v) of the Proposition.

Thus, in light of (\ref{HH1}), for (\ref{HH2}) to hold it suffices to show
that 
\begin{equation}
E\left[ T_{n}^{\dag }(t)\left( W_{n}-\exp \left( -\eta ^{2}t^{2}/2\right)
\right) \varsigma \right] \rightarrow 0.  \label{HH3}
\end{equation}%
Let $K$ be some constant such that $P(\left\vert \varsigma \right\vert \leq
K)=1$, then $E\left[ \left\vert T_{n}^{\dag }\left( t\right) \exp \left(
-\eta ^{2}t^{2}/2\right) \varsigma \right\vert ^{2}\right] \leq K^{2}E\left[
\left\vert T_{n}^{\dag }\left( t\right) \right\vert ^{2}\right] $ is
uniformly bounded in $n$, since $E\left[ \left\vert T_{n}^{\dag }\left(
t\right) \right\vert ^{2}\right] $ is uniformly bounded as shown above.
Observing that $\left\vert I_{n}\right\vert =1$ we also have $E\left[
\left\vert I_{n}\varsigma \right\vert ^{2}\right] \leq K^{2}$. In light of (%
\ref{HH1}) it follows furthermore that%
\begin{equation*}
E\left[ \left\vert T_{n}^{\dag }(W_{n}-\exp \left( -\eta ^{2}t^{2}/2\right)
)\varsigma \right\vert ^{2}\right] \leq 2E\left[ \left\vert I_{n}\varsigma
\right\vert ^{2}\right] +2E\left[ \left\vert T_{n}^{\dag }\left( t\right)
\exp \left( -\eta ^{2}t^{2}/2\right) \varsigma \right\vert ^{2}\right]
\end{equation*}%
is uniformly bounded in $n$, it follows that $T_{n}^{\dag }\left( t\right)
(W_{n}-\exp \left( -\eta ^{2}t^{2}/2\right) )\varsigma $ is uniformly
integrable. Having established uniform integrability, Condition (\ref{HH3})
now follows since as shown by Hall and Heyde (1980, Lemma 3.1), $W_{n}-\exp
\left( -\eta ^{2}t^{2}/2\right) \overset{p}{\rightarrow }0\ $by using
Conditions (\ref{3.18}) and (\ref{3.19}). Thus, it follows that $T_{n}^{\dag
}\left( W_{n}-\exp \left( -\eta ^{2}t^{2}/2\right) \right) \varsigma \overset%
{p}{\rightarrow }0$. This completes the proof that $S_{nk_{n}}^{\dag }%
\overset{d}{\rightarrow }Z$ ($\mathcal{C}$-stably) when $\eta ^{2}$ is a.s.
bounded.

The case where $\eta ^{2}$ is not a.s. bounded can be handled in the same
way as in Hall and Heyde (1980, p.62) after replacing their $I\left(
E\right) $ with $\varsigma .$

Let $\xi $ $\sim N(0,1)$ be some random variable independent of $\mathcal{C}$%
, and hence independent of $\eta $ (possibly after redefining all variables
on an extended probability space), then for any $P$-essentially bounded $%
\mathcal{C}$-measurable random variable $\varsigma $ we have 
\begin{equation*}
E\left[ \varsigma \exp (it\eta \xi )\right] =E\left[ \varsigma \exp (-\frac{1%
}{2}\eta ^{2}t^{2})\right]
\end{equation*}
by iterated expectations, and thus $S_{nk_{n}}\overset{d}{\rightarrow }\eta
\xi $ ($\mathcal{C}$-stably).
\end{proof}

\begin{proof}[Proof of Proposition \protect\ref{Theorem_Spatial_Mixing}]
By Assumptions (iii) and (iv)\ and setting 
\begin{equation*}
N=n/\left( c_{T}\left\lfloor n^{1/4}\right\rfloor +c_{J}\left\lfloor
n^{3/4}\right\rfloor \right)
\end{equation*}
it follows that%
\begin{equation*}
\frac{\left( \left\lfloor n^{3/4}\right\rfloor -1\right) n}{n\left(
\left\lfloor n^{1/4}\right\rfloor +\left\lfloor n^{3/4}\right\rfloor \right) 
}\leq \frac{\left\vert J_{k_{n}^{_{i}}}\left( q_{i}\right) \right\vert }{n}%
N\leq \frac{\left\lfloor n^{3/4}\right\rfloor n}{n\left( \left\lfloor
n^{1/4}\right\rfloor +\left\lfloor n^{3/4}\right\rfloor \right) }
\end{equation*}%
holds eventually as $n\rightarrow \infty .$ This implies that 
\begin{equation*}
\underset{n\rightarrow \infty }{\lim \inf }\inf_{i}\left\vert \frac{%
\left\vert J_{k_{n}^{q_{i}}}\left( q_{i}\right) \right\vert }{n}N\right\vert
=1
\end{equation*}%
and 
\begin{equation*}
\underset{n\rightarrow \infty }{\lim \sup }\sup_{i}\left\vert \frac{%
\left\vert J_{k_{n}^{q_{i}}}\left( q_{i}\right) \right\vert }{n}N\right\vert
=1
\end{equation*}%
By Condition (v) it follows that $n^{-1/4+\epsilon }\left\vert
T_{k_{n}^{i},h_{n}^{i}}\left( q_{i}\right) \right\vert =1$ a.s. With the
sets $J_{k_{n}^{q_{i}}}\left( q_{i}\right) $ and $T_{k_{n}^{i},h_{n}^{i}}%
\left( q_{i}\right) $ form the random variables $X_{i,n}=\sum_{j\in
J_{k_{n}^{_{i}}}\left( q_{i}\right) }\left( v_{j,n}\left( \zeta \right) -\mu
_{j,n}\right) $ for $i=1,...,N$ and $U_{i,n}=\sum_{j\in
T_{k_{n}^{i},h_{n}^{i}}\left( q_{i}\right) }\left( v_{j,n}\left( \zeta
\right) -\mu _{j,n}\right) $ for $i=1,...,N.$ It follows that for $%
S_{n}=\sum_{i=1}^{n}\left( v_{i,n}\left( \zeta \right) -\mu _{i,n}\right) $
one obtains $S_{n}=\sum_{i=1}^{N}\left( X_{i,n}+U_{i,n}\right) .$

The next step in the argument consists in showing that the component $%
n^{-1/2}\sum_{i=1}^{N}U_{i,n}$ in $n^{-1/2}S_{n}$ is asymptotically
negligible. For $\varepsilon >0$ and $\delta >0$ consider 
\begin{eqnarray}
P\left( \left\vert n^{-1/2}\tsum\nolimits_{i=1}^{N}U_{i,n}\right\vert
>\varepsilon \right) &\leq &\frac{1}{\varepsilon ^{2+\delta }}E\left[
\left\vert n^{-1/2}\tsum\nolimits_{i=1}^{N}U_{i,n}\right\vert ^{2+\delta }%
\right]  \label{Sum_Uin_bound} \\
&\leq &\frac{1}{\varepsilon ^{2+\delta }n^{1+\delta /2}}E\left[ \left(
\tsum\nolimits_{i=1}^{N}E\left[ \left\vert U_{i,n}\right\vert ^{2+\delta
}|\zeta \right] ^{1/\left( 2+\delta \right) }\right) ^{2+\delta }\right] 
\notag
\end{eqnarray}%
by an equality of Stout (1974, p. 201). Repeated application of that
inequality gives%
\begin{eqnarray}
E\left[ \left\vert U_{i,n}\right\vert ^{2+\delta }|\zeta \right] &\leq
&\left( \sum_{j\in T_{k_{n}^{i},h_{n}^{i}}\left( q_{i}\right) }E\left[
\left\vert v_{j,n}\left( \zeta \right) -\mu _{j,n}\right\vert ^{2+\delta
}|\zeta \right] ^{1/\left( 2+\delta \right) }\right) ^{2+\delta }
\label{Uin_bound} \\
&\leq &2^{1+\delta /2}\left( \sum_{j\in T_{k_{n}^{i},h_{n}^{i}}\left(
q_{i}\right) }\left( E\left[ \left\vert v_{j,n}\left( \zeta \right)
\right\vert ^{2+\delta }|\zeta \right] +\left\vert \mu _{j,n}\right\vert
^{2+\delta }\right) ^{1/\left( 2+\delta \right) }\right) ^{2+\delta }  \notag
\\
&\leq &2^{1+\delta /2}\left( E\left[ z\left( \zeta \right) ^{2+\delta
}|\zeta \right] +K^{2+\delta }\right) \left\vert
T_{k_{n}^{i},h_{n}^{i}}\left( q_{i}\right) \right\vert ^{2+\delta }  \notag
\end{eqnarray}%
where the second inequality is using the H\"{o}lder inequality. By Condition
(i) of the theorem $E\left[ \left\vert v_{j,n}\left( \zeta \right)
\right\vert ^{2+\delta }|\zeta \right] \leq E\left[ z\left( \zeta \right)
^{2+\delta }|\zeta \right] $ where $E\left[ z\left( \zeta \right) ^{2+\delta
}|\zeta \right] $ has bounded expectation. Then, substituting (\ref%
{Uin_bound}) into (\ref{Sum_Uin_bound}) and using the Cauchy-Schwartz
inequality leads to 
\begin{eqnarray}
&&P\left( \left\vert n^{-1/2}\tsum\nolimits_{i=1}^{N}U_{i,n}\right\vert
>\varepsilon \right)  \label{Negligible} \\
&\leq &\frac{2^{\left( 1+\delta /2\right) }E\left[ \left( E\left[ z\left(
\zeta \right) ^{2+\delta }|\zeta \right] +K^{2+\delta }\right) ^{2}\right]
^{1/2}}{\varepsilon ^{2+\delta }n^{1+\delta /2}}E\left[ \left(
\tsum\nolimits_{i=1}^{N}\left\vert T_{k_{n}^{i},h_{n}^{i}}\left(
q_{i}\right) \right\vert \right) ^{4+2\delta }\right] ^{1/2}  \notag \\
&=&O\left( n^{-\left( 1+\delta /2\right) +\left( 1/2-\epsilon \right) \left(
2+\delta \right) }\right) =o\left( 1\right)  \notag
\end{eqnarray}%
because $n^{-1/2+\epsilon }\tsum\nolimits_{i=1}^{N}\left\vert
T_{k_{n}^{i},h_{n}^{i}}\left( q_{i}\right) \right\vert =1$ a.s.

By Assumption (ii) of the Theorem $n^{-1}\sum_{i=1}^{N}X_{i,n}^{2}%
\rightarrow _{p}\eta ^{2}$ and by Assumption (1) of the Theorem it follows
from Lemma \ref{Lemma_Suff_1} that $\max_{i}\left\vert
n^{-1/2}X_{i,n}\right\vert \rightarrow _{p}0$ and $E\left[
n^{-1}\max_{i}\left\vert X_{i,n}^{2}\right\vert \right] $ is bounded in $n.$
This shows that Assumptions (i)-(iii) of Proposition \ref{Theorem_CLT_Blocks}
hold. Assumptions (iii) and (v) of the Theorem imply that Assumptions (iv)
and (v) of Proposition \ref{Theorem_CLT_Blocks} hold. Then, by Theorem \ref%
{Theorem_CLT_Blocks} it follows that 
\begin{equation*}
n^{-1/2}\sum_{i=1}^{N}X_{i,n}\rightarrow _{d}N\left( 0,\eta ^{2}\right) 
\text{ }C\text{-stably.}
\end{equation*}%
Since (\ref{Negligible}) holds, it follows by the continous mapping theorem
for stable convergence, see for example Kuersteiner and Prucha (2013), that $%
n^{-1/2}S_{n}=n^{-1/2}\sum_{i=1}^{N}X_{i,n}+o_{p}\left( 1\right) $ and 
\begin{equation*}
n^{-1/2}S_{n}\rightarrow _{d}N\left( 0,\eta ^{2}\right) \text{ }C\text{%
-stably.}
\end{equation*}
\end{proof}

\begin{proof}[Proof of Theorem \protect\ref{Theorem_Main}]
The result follows by construction from the algorithm given in (\ref{Algo_1}%
)-(\ref{Algo_12}) because the sequences $k_{n}^{i}$ and $h_{n}^{i}$ can
always be chose such that Conditions (iv) and Condition (v) of Proposition %
\ref{Theorem_Spatial_Mixing} are satisfied. The remaining conditions of
Proposition \ref{Theorem_Spatial_Mixing} are maintained in this Theorem. The
result thus follows immediately from Proposition \ref{Theorem_Spatial_Mixing}%
.
\end{proof}

\subsection{Proofs for Section\label{Proofs_Network} \protect\ref%
{Section_Network_Model}}

\begin{proof}[Proof of Proposition \protect\ref{Proposition_Example_Network} 
]
For (i) first consider the case of $v_{i,n}.$ First note that (\ref%
{Definition_dij_Neighborhood}) implies that $E\left[ v_{i,n}\right] =\mu
_{i,n}=\sum_{j=\min (1,i-\kappa _{u})}^{\max \left( n,i+\kappa _{u}\right) }E%
\left[ d_{ij}\right] .$ Choose $g_{ij}\left( \zeta \right) =H(\alpha
_{0}+\alpha _{\zeta }\left\vert \zeta _{ij}\right\vert )=H\left( \zeta
_{ij},\alpha \right) $ and consider 
\begin{equation*}
w_{j,i,n}^{k}\left( \zeta \right) =v_{j,n}\left( \zeta \right) 1\left\{
g_{ij}\left( \zeta \right) \leq H\left( \zeta _{ij},\alpha \right) \right\} .
\end{equation*}%
Let 
\begin{equation*}
\mathcal{B}_{i,n}^{\kappa _{u}}=\sigma \left( w_{1,n}^{k}\left( \zeta
\right) ,...,w_{n,n}^{k}\left( \zeta \right) \right)
\end{equation*}%
where $\mathcal{B}_{i,n}^{k}$ is the $\sigma $-field generated by $%
v_{l,n}\left( \zeta \right) $ for $l<i-\kappa _{u}-1$ and $l>i+\kappa
_{u}+1. $ The mean $\mu _{i,n}$ is given as 
\begin{equation*}
\mu _{i,n}=\sum_{j=1}^{n}E\left[ d_{ij}\right] =\sum_{j=\min (1,i-\kappa
_{u})}^{\max \left( n,i+\kappa _{u}\right) }E\left[ H\left( -\left\vert
\zeta _{ij}\right\vert \right) 1\left\{ \left\vert \zeta _{ij}\right\vert
<\kappa _{u}\right\} \right]
\end{equation*}%
because by the properties of the distribution of $\zeta $ it follows that $%
1\left\{ \left\vert \zeta _{ij}\right\vert <\kappa _{u}\right\} =0$ a.s. for 
$j<i-\kappa _{u}-1$ or $j>i+\kappa _{u}+1.$ Similarly, 
\begin{eqnarray*}
E\left[ v_{i,n}\left( \zeta \right) |\mathcal{B}_{i,n}^{k}\right]
&=&\sum_{j=1}^{n}E\left[ d_{ij}|\mathcal{B}_{i,n}^{k}\right] \\
&=&\sum_{j=\min (1,i-\kappa _{u}-1)}^{\max \left( n,i+\kappa _{u}+1\right)
}E \left[ H\left( -\left\vert \zeta _{ij}\right\vert \right) 1\left\{
\left\vert \zeta _{ij}\right\vert <\kappa _{u}\right\} |\mathcal{B}_{i,n}^{k}%
\right]
\end{eqnarray*}%
where the last equality follows from the fact that $v_{l,n}\left( \zeta
\right) =\sum_{k=1}^{n}d_{lk}$ does not depend on $\epsilon _{ij}$ for $%
l\neq i.$

Now only consider the case where $k$ is integer and $k>\kappa _{u}.$ Note
that $\mathcal{B}_{i,n}^{k}$ is the $\sigma $-field generated by $%
v_{l,n}\left( \zeta \right) $ for $l<i-k$ or $l>i+k.$ Then, noting that for $%
j\in \left\{ i-k,...,i+k\right\} ,$ $v_{l,n}$ does not depend on $\zeta
_{ij} $ and $g_{il}$ also does not depend on $\zeta _{ij}$, while $g_{ij}=0$
for $j\in \left\{ i-k,...,i+k\right\} .$ Thus, $\mathcal{B}_{i,n}^{k}$ does
not depend on $\zeta _{ij}$ as long as $k>\kappa _{u}.$ It follows that 
\begin{equation*}
E\left[ d_{ij}|\mathcal{B}_{i,n}^{k}\right] =E\left[ d_{ij}\right]
\end{equation*}%
and consequently that 
\begin{equation*}
\mu _{i,n}-E\left[ v_{i,n}\left( \zeta \right) |\mathcal{B}_{i,n}^{k}\right]
=0.
\end{equation*}%
This implies that $v_{i,k}\left( \zeta \right) =0$ for $k>\kappa _{u}.$
Similarly, for $v_{i}$ the same argument above applies except that now 
\begin{equation*}
E\left[ v_{i}\right] =\mu _{i}=\mu =\sum_{j=i-\kappa _{u}}^{i+\kappa _{u}}E%
\left[ d_{ij}\right]
\end{equation*}%
because for each $i,$ $d_{ij}$ has a distribution that only depends on $%
\left\vert i-j\right\vert $ but not on $i$. Similarly, 
\begin{equation*}
E\left[ v_{i}\left( \zeta \right) |\mathcal{B}_{i}^{k}\right]
=\sum_{j=i-\kappa _{u}}^{i+\kappa _{u}}E\left[ H\left( -\left\vert \zeta
_{ij}\right\vert \right) 1\left\{ \left\vert \zeta _{ij}\right\vert <\kappa
_{u}\right\} |\mathcal{B}_{i}^{k}\right]
\end{equation*}%
which, by the same arguments as before is constant for $k>\kappa _{u}$ and
equal to $\mu .$

For (ii) choose 
\begin{equation*}
A_{k_{m}}\left( i,j\right) =\left\{ \omega |\Lambda \left( k_{m-1}\right)
<g_{ij}\left( \zeta \right) \leq \Lambda \left( k_{m}\right) \right\}
\end{equation*}%
with $\Lambda \left( k\right) =H\left( k,\alpha \right) ,$ $k_{m}=m.$ It
follows that $v_{i,n}\leq 2\kappa _{u}$ a.s. because of the bounded support
assumption and the fact that network connections are limited to close
neighbors. Then Assumption \ref{Assume_Moments_Mixing} holds and $\left\Vert
\mu _{i,n}-E\left[ v_{i,n}\left( \zeta \right) |\mathcal{B}_{i,n}^{k}\right]
\right\Vert _{2,\zeta }\leq 4\kappa _{u}$ for all $k\leq \kappa _{u}$ wheras
by (i) the upper bound is zero for $k\geq \kappa _{u}.$ The same holds for $%
\left\Vert \mu _{i}-E\left[ v_{i}\left( \zeta \right) |\mathcal{B}_{i}^{k}%
\right] \right\Vert _{2,\zeta }.$ One then obtains, together with (i), that $%
\sum_{m=1}^{\infty }E\left[ \psi _{i,k_{m}}\left( \zeta \right)
|A_{k,n}\left( i,j\right) \right] \leq 4\kappa _{u}^{2}$ uniformly in $i$
and $j.$ It also follows that $\sum_{j=1}^{n}\Pr \left( A_{k_{m}}\left(
i,j\right) \right) \leq 2$ for all $i.$ Thus, 
\begin{equation*}
\sum_{j=i}^{n}\sum_{m=1}^{\infty }E\left[ \psi _{i,k_{m}}\left( \zeta
\right) |A_{k_{m}}\left( i,j\right) \right] \Pr \left( A_{k_{m}}\left(
i,j\right) \right) \leq 8\kappa _{u}^{2}
\end{equation*}%
and Assumption \ref{Assume_Probability_Sum} holds by Remark \ref%
{Remark_Probability_Summability}. The assumptions of the example then imply
that $E\left[ v_{i}\right] =\mu $ and by Theorem \ref{Theorem_Stout_3.7.1}
it follows 
\begin{equation}
n^{-1}\sum_{i=1}^{n}v_{i}\rightarrow _{a.s.}\mu .  \label{vi_as}
\end{equation}%
For $v_{i,n}$ note that when $\kappa _{u}+1<i\leq n-\kappa _{u}-1$ it
follows that $v_{i,n}=v_{i}$. Therefore write 
\begin{eqnarray}
n^{-1}\sum_{i=1}^{n}v_{i,n}\left( \zeta \right) &=&\frac{n-2\left( \kappa
_{u}+1\right) }{n}\frac{1}{n-2\left( \kappa _{u}+1\right) }\sum_{i=\kappa
_{u}+2}^{n-\kappa _{u}-1}v_{i}\left( \zeta \right)  \label{vn_Decomposition}
\\
&&+n^{-1}\sum_{i=1}^{\kappa _{u}+1}v_{i,n}\left( \zeta \right)
+n^{-1}\sum_{i=n-\kappa _{u}}^{n}v_{i,n}\left( \zeta \right) .  \notag
\end{eqnarray}%
By absolute convergence and the fact that $\left\vert v_{i,n}\left( \zeta
\right) \right\vert \leq 2\kappa _{u}$ it follows that 
\begin{equation}
\left\vert n^{-1}\sum_{i=n-\kappa _{u}}^{n}v_{i,n}\left( \zeta \right)
\right\vert \leq 2\kappa _{u}\frac{\kappa _{u}}{n}\rightarrow 0
\label{vin_1}
\end{equation}%
and 
\begin{equation}
\left\vert n^{-1}\sum_{i=1}^{\kappa _{u}+1}v_{i,n}\left( \zeta \right)
\right\vert \leq \kappa _{u}\frac{\kappa _{u}+1}{n}\rightarrow 0
\label{vin_2}
\end{equation}%
which means that the last two terms in (\ref{vn_Decomposition}) converge to
zero almost surely. For the first term in (\ref{vn_Decomposition}) note that 
$\frac{1}{n-2\left( \kappa _{u}+1\right) }\sum_{i=\kappa _{u}+2}^{n-\kappa
_{u}-1}v_{i}\left( \zeta \right) \rightarrow _{a.s.}\mu $ by the same
argument as in (\ref{vi_as}). Finally, the factor $\frac{n-2\left( \kappa
_{u}+1\right) }{n}\rightarrow 1$ as $n\rightarrow \infty .$ The result then
follows from the continuous mapping theorem.

For (iii) the assumptions of the central limit theorem need to be checked.
First consider $v_{i}$ which is the easier case. Note that $\left\vert
v_{i}\right\vert \leq 2\kappa _{u}$ which implies that Condition (i) of
Proposition \ref{Theorem_Spatial_Mixing}.

Using the properties of the joint distribution $P_{\zeta }$ choose sets $%
J_{k}\left( q_{1}\right) =\left\{ 1,...,\left\lfloor n^{3/4}\right\rfloor
\right\} $ where $q_{1}$ can be located at the center of $\left\{
1,...,\left\lfloor n^{3/4}\right\rfloor \right\} .$ Then choose 
\begin{equation*}
T_{h,k}\left( q_{1}\right) =\left\{ \left\lfloor n^{3/4}\right\rfloor
+1,...,\left\lfloor n^{3/4}\right\rfloor +\left\lfloor n^{1/4-\epsilon
}\right\rfloor \right\}
\end{equation*}
and continuing in this fashion. Let $N=n/\left( \left\lfloor
n^{3/4}\right\rfloor +\left\lfloor n^{1/4-\epsilon }\right\rfloor \right) .$
It then follows that once that Conditions (iv) and (v) of Proposition \ref%
{Theorem_Spatial_Mixing} hold. By the stationary nature of the process it
follows that $n^{-3/4}E\left[ X_{i,n}^{2}\right] =\sigma ^{2}+O\left(
n^{-3/4}\right) $ which implies that for $N=n/\left( \left\lfloor
n^{3/4}\right\rfloor +\left\lfloor n^{1/4}\right\rfloor \right) =O\left(
n^{1/4}\right) $ it follows that 
\begin{equation*}
n^{-1}\sum_{i=1}^{N}E\left[ X_{i,n}^{2}\right] =\frac{1}{N}\frac{N}{n^{1/4}}%
\sum_{i=1}^{N}n^{-3/4}E\left[ X_{i,n}^{2}\right] =\sigma ^{2}+o\left(
1\right) .
\end{equation*}%
Since in this model, $X_{i,n}$ is eventually (as $n$ increases) independent
of $X_{j,n}$ it follows immediately by a strong law of large numbers, or by
applying the theory developed in Section \ref{Section_LLN} that 
\begin{equation*}
n^{-1}\sum_{i=1}^{N}\left( X_{i,n}^{2}-E\left[ X_{i,n}^{2}\right] \right)
\rightarrow _{a.s.}0.
\end{equation*}%
This establishes Condition (ii) in Proposition \ref{Theorem_Spatial_Mixing}
holds. Condition (iii) of the Proposition was established before in (i). If $%
g_{ij}=1/\left\vert i-j\right\vert $ and $\Lambda \left( k\right) =1/k$ then 
$k_{n}^{1}=\left\lfloor n^{3/4}\right\rfloor ,$ $h_{n}^{1}=\left\lfloor
n^{3/4}\right\rfloor +\left\lfloor n^{1/4-\epsilon }\right\rfloor $ and $%
k_{n}^{i}=h_{n}^{i-1}+\left\lfloor n^{3/4}\right\rfloor $ and $%
h_{n}^{i}=k_{n}^{i}+\left\lfloor n^{1/4-\epsilon }\right\rfloor .$ It
follows that $h_{n}^{\prime }=\left\lfloor n^{1/4-\epsilon }\right\rfloor $
and by the result in (i) it follows that Condition (vi) of Proposition \ref%
{Theorem_Spatial_Mixing} holds. By Proposition \ref{Theorem_Spatial_Mixing}
it therefore follows that $n^{-1/2}\sum_{i=1}^{n}\left( v_{i}-\mu \right)
\rightarrow _{d}N\left( 0,\sigma ^{2}\right) $ $\mathcal{C}$-stably. Also
note that the regularity condition that $\Pr \left( g_{ij}\left( \zeta
\right) =g_{ik}\left( \zeta \right) \right) =0$ for all $i$ and all $j\neq k$
in Theorem \ref{Theorem_Main} trivially holds in this example. Finally, the
result for $v_{i,n}$ follows by the same argument as in (\ref{vin_1}) and (%
\ref{vin_2}) to show that the difference between $v_{i,n}$ and $v_{i}$ is
asymptotically negligible.
\end{proof}

\subsection{Proofs for Examples\label{ExampleProofs}}

\begin{proof}[Proof of Example \protect\ref{Example_Mixing}]
Let $H$ be the logistic CDF. Choose $g_{ij}\left( \zeta \right)
=p_{ij}\left( \zeta \right) =E\left[ d_{ij}|\zeta \right] $ and $\Lambda
\left( k\right) =H\left( -k\right) .$ This implies that $1\left\{
g_{ij}\left( \zeta \right) \leq \Lambda \left( k\right) \right\} =1\left\{
\left\vert \zeta _{i}-\zeta _{j}\right\vert >k\right\} .$ The mean $\mu
_{i,n}$ is given as 
\begin{equation*}
\mu _{i,n}=\sum_{j=1}^{n}E\left[ d_{ij}\right] =\sum_{j=1}^{n}E\left[
H\left( -\left\vert \zeta _{i}-\zeta _{j}\right\vert \right) \right]
\end{equation*}%
It is useful to decompose $E\left[ d_{ij}\right] $ as follows 
\begin{eqnarray*}
E\left[ d_{ij}\right] &=&E\left[ H\left( -\left\vert \zeta _{i}-\zeta
_{j}\right\vert \right) \right] =\int H\left( -\left\vert \zeta _{i}-\zeta
_{j}\right\vert \right) dP_{\zeta } \\
&=&\int_{\left\vert \zeta _{i}-\zeta _{j}\right\vert \leq k}H\left(
-\left\vert \zeta _{i}-\zeta _{j}\right\vert \right) dP_{\zeta
}+\int_{\left\vert \zeta _{i}-\zeta _{j}\right\vert >k}H\left( -\left\vert
\zeta _{i}-\zeta _{j}\right\vert \right) dP_{\zeta } \\
&=&E\left[ d_{ij}|\left\vert \zeta _{i}-\zeta _{j}\right\vert \leq k\right]
P\left( \left\vert \zeta _{i}-\zeta _{j}\right\vert \leq k\right) \\
&&+E\left[ d_{ij}|\left\vert \zeta _{i}-\zeta _{j}\right\vert >k\right]
P\left( \left\vert \zeta _{i}-\zeta _{j}\right\vert >k\right) \\
&\leq &E\left[ d_{ij}|\left\vert \zeta _{i}-\zeta _{j}\right\vert \leq k%
\right] P\left( \left\vert \zeta _{i}-\zeta _{j}\right\vert \leq k\right)
+\exp \left( -k\right) .
\end{eqnarray*}
Similarly, 
\begin{eqnarray*}
E\left[ v_{i,n}\left( \zeta \right) |\mathcal{B}_{i,n}^{k}\right]
&=&\sum_{j=1}^{n}E\left[ d_{ij}|\mathcal{B}_{i,n}^{k}\right] \\
&=&\sum_{j=1}^{n}E\left[ H\left( -\left\vert \zeta _{i}-\zeta
_{j}\right\vert \right) |\mathcal{B}_{i,n}^{k}\right]
\end{eqnarray*}%
where the last equality follows from the fact that $v_{l,n}\left( \zeta
\right) =\sum_{k=1}^{n}d_{lk}$ does not depend on $\epsilon _{ij}$ for $%
l\neq i.$

For $E\left[ H\left( -\left\vert \zeta _{i}-\zeta _{j}\right\vert \right) |%
\mathcal{B}_{i,n}^{k}\right] $ distinguish two cases: if $%
w_{j,i,n}^{k}\left( \zeta \right) >0$ then $\left\vert \zeta _{i}-\zeta
_{j}\right\vert >k$ and if $w_{j,i,n}^{k}\left( \zeta \right) =0$ then
either $\left\vert \zeta _{i}-\zeta _{j}\right\vert \leq k$ or $\left\vert
\zeta _{i}-\zeta _{j}\right\vert >k$ and $\epsilon _{jk}>\left\vert \zeta
_{j}-\zeta _{k}\right\vert $ for all $k=1,...,n.$ Consider the first case
where $\left\vert \zeta _{i}-\zeta _{j}\right\vert >k$. Since%
\begin{equation*}
H\left( -\left\vert \zeta _{i}-\zeta _{j}\right\vert \right) <\exp \left(
-\left\vert \zeta _{i}-\zeta _{j}\right\vert \right) \leq \exp \left(
-k\right)
\end{equation*}%
it follows that 
\begin{equation*}
E\left[ H\left( -\left\vert \zeta _{i}-\zeta _{j}\right\vert \right) |%
\mathcal{B}_{i,n}^{k}\right] <\exp \left( -k\right) .
\end{equation*}%
Similarly, for the third case where $\left\vert \zeta _{i}-\zeta
_{j}\right\vert >k$ and $\epsilon _{jk}>\left\vert \zeta _{j}-\zeta
_{k}\right\vert $ and noting that $H\left( -\left\vert \zeta _{i}-\zeta
_{j}\right\vert \right) $ does not depend on $\epsilon _{jk}$ one obtains
the bound%
\begin{equation*}
E\left[ H\left( -\left\vert \zeta _{i}-\zeta _{j}\right\vert \right) |%
\mathcal{B}_{i,n}^{k}\right] <\exp \left( -k\right) .
\end{equation*}%
Finally, when $\left\vert \zeta _{i}-\zeta _{j}\right\vert \leq k$ note that
since $H\left( -\left\vert \zeta _{i}-\zeta _{j}\right\vert \right) $ is
only a function of $\left\vert \zeta _{i}-\zeta _{j}\right\vert $ it follows
that $E\left[ H\left( -\left\vert \zeta _{i}-\zeta _{j}\right\vert \right) |%
\mathcal{B}_{i,n}^{k}\right] =E\left[ H\left( -\left\vert \zeta _{i}-\zeta
_{j}\right\vert \right) |\left\vert \zeta _{i}-\zeta _{j}\right\vert \leq k%
\right] .$ Then consider 
\begin{eqnarray*}
\left\vert E\left[ d_{ij}\right] -E\left[ d_{ij}|\mathcal{B}_{i,n}^{k}\right]
\right\vert &=&\left\vert E\left[ d_{ij}|\left\vert \zeta _{i}-\zeta
_{j}\right\vert \leq k\right] P\left( \left\vert \zeta _{i}-\zeta
_{j}\right\vert \leq k\right) -E\left[ d_{ij}|\left\vert \zeta _{i}-\zeta
_{j}\right\vert \leq k\right] \right\vert \\
&&+E\left[ d_{ij}|\left\vert \zeta _{i}-\zeta _{j}\right\vert >k\right]
P\left( \left\vert \zeta _{i}-\zeta _{j}\right\vert >k\right) \\
&\leq &\left\vert P\left( \left\vert \zeta _{i}-\zeta _{j}\right\vert \leq
k\right) -1\right\vert +\exp \left( -k\right) .
\end{eqnarray*}%
To summarize it follows from the above calculations that 
\begin{equation*}
\left\vert E\left[ d_{ij}\right] -E\left[ d_{ij}|\mathcal{B}_{i,n}^{k}\right]
\right\vert \leq \left\{ 
\begin{array}{cc}
P\left( \left\vert \zeta _{i}-\zeta _{j}\right\vert \leq k\right) +\exp
\left( -k\right) \left( 1+P\left( \left\vert \zeta _{i}-\zeta
_{j}\right\vert >k\right) \right) & \text{if }\left\vert \zeta _{i}-\zeta
_{j}\right\vert >k \\ 
\left\vert P\left( \left\vert \zeta _{i}-\zeta _{j}\right\vert \leq k\right)
-1\right\vert +\exp \left( -k\right) & \text{if }\left\vert \zeta _{i}-\zeta
_{j}\right\vert \leq k%
\end{array}%
\right.
\end{equation*}%
such that 
\begin{eqnarray*}
E\left[ \left\vert E\left[ d_{ij}\right] -E\left[ d_{ij}|\mathcal{B}%
_{i,n}^{k}\right] \right\vert \right] &\leq &\left( P\left( \left\vert \zeta
_{i}-\zeta _{j}\right\vert \leq k\right) +2\exp \left( -k\right) \right)
P\left( \left\vert \zeta _{i}-\zeta _{j}\right\vert >k\right) \\
&&+\left\vert P\left( \left\vert \zeta _{i}-\zeta _{j}\right\vert \leq
k\right) -1\right\vert P\left( \left\vert \zeta _{i}-\zeta _{j}\right\vert
\leq k\right) \\
&&+\exp \left( -k\right) P\left( \left\vert \zeta _{i}-\zeta _{j}\right\vert
\leq k\right)
\end{eqnarray*}%
To show that 
\begin{eqnarray*}
\sum_{j=1}^{n}E\left[ \left\vert E\left[ d_{ij}\right] -E\left[ d_{ij}|%
\mathcal{B}_{i,n}^{k}\right] \right\vert \right] &\leq &E\left[ \psi
_{i,k}\left( \zeta \right) \right] \\
&\leq &\sum_{j=1}^{n}\left( P\left( \left\vert \zeta _{i}-\zeta
_{j}\right\vert \leq k\right) +2\exp \left( -k\right) \right) P\left(
\left\vert \zeta _{i}-\zeta _{j}\right\vert >k\right) \\
&&+\sum_{j=1}^{n}\left\vert P\left( \left\vert \zeta _{i}-\zeta
_{j}\right\vert \leq k\right) -1\right\vert P\left( \left\vert \zeta
_{i}-\zeta _{j}\right\vert \leq k\right) \\
&&+\exp \left( -k\right) \sum_{j=1}^{n}P\left( \left\vert \zeta _{i}-\zeta
_{j}\right\vert \leq k\right)
\end{eqnarray*}%
goes to zero as $k\rightarrow \infty $ additional restrictions on the joint
distribution $P_{\zeta }$ of $\zeta $ are required. Fist note that 
\begin{equation*}
\sum_{j=1}^{n}P\left( \left\vert \zeta _{i}-\zeta _{j}\right\vert \leq
k\right) P\left( \left\vert \zeta _{i}-\zeta _{j}\right\vert >k\right)
=\sum_{j=1}^{n}P\left( \left\vert \zeta _{i}-\zeta _{j}\right\vert \leq
k\right) \left( 1-P\left( \left\vert \zeta _{i}-\zeta _{j}\right\vert \leq
k\right) \right)
\end{equation*}%
such that an overall bound on $E\left[ \psi _{i,k}\left( \zeta \right) %
\right] $ is given by 
\begin{eqnarray*}
E\left[ \psi _{i,k}\left( \zeta \right) \right] &\leq
&2\sum_{j=1}^{n}\left\vert P\left( \left\vert \zeta _{i}-\zeta
_{j}\right\vert \leq k\right) -1\right\vert P\left( \left\vert \zeta
_{i}-\zeta _{j}\right\vert \leq k\right) \\
&&+3\exp \left( -k\right) \sum_{j=1}^{n}P\left( \left\vert \zeta _{i}-\zeta
_{j}\right\vert \leq k\right) .
\end{eqnarray*}%
Now use the fact that $\sup_{i}\sum_{j=1}^{\infty }P\left( \left\vert \zeta
_{i}-\zeta _{j}\right\vert \leq k\right) <\infty $ for any $0\leq k<\infty .$
This implies that for any $\varepsilon >0$ there exists a $k_{1}<\infty $
such that 
\begin{equation}
\exp \left( -k_{1}\right) \sum_{j=1}^{n}P\left( \left\vert \zeta _{i}-\zeta
_{j}\right\vert \leq k_{1}\right) \leq \frac{\varepsilon }{2}.
\label{EX_Bound1}
\end{equation}%
More specifically, since $\sum_{j=1}^{n}P\left( \left\vert \zeta _{i}-\zeta
_{j}\right\vert \leq k\right) \leq K$ choose $k_{1}$ such that $k_{1}\geq
\log \left( 2K/\varepsilon \right) .$ For the same $\varepsilon $ there
exists an $n_{2}<\infty $ such that for all $n^{\prime }>n_{2}$ and for any $%
k_{2}<\infty $ fixed it holds that 
\begin{equation}
\sum_{j=n_{2}+1}^{\infty }P\left( \left\vert \zeta _{i}-\zeta
_{j}\right\vert \leq k_{2}\right) \leq \frac{\varepsilon }{8}.
\label{EX_TailBound}
\end{equation}%
Finally, for any $n_{2}$ given in (\ref{EX_TailBound})\ there is a $%
k_{3}<\infty $ such that 
\begin{equation}
\inf_{j\leq n_{2}}P\left( \left\vert \zeta _{i}-\zeta _{j}\right\vert \leq
k_{3}\right) \geq 1-\frac{\varepsilon }{4n_{2}}  \label{EX_CoverageBound}
\end{equation}%
and It then follows that for $k_{4}=\max \left( k_{2},k_{3}\right) $ 
\begin{eqnarray}
&&\sum_{j=1}^{n}\left\vert P\left( \left\vert \zeta _{i}-\zeta
_{j}\right\vert \leq k_{4}\right) -1\right\vert P\left( \left\vert \zeta
_{i}-\zeta _{j}\right\vert \leq k_{4}\right)  \label{EX_Bound2} \\
&\leq &\sum_{j=1}^{n_{2}}\left\vert P\left( \left\vert \zeta _{i}-\zeta
_{j}\right\vert \leq k_{4}\right) -1\right\vert +\sum_{j=n_{2}+1}^{\infty
}2P\left( \left\vert \zeta _{i}-\zeta _{j}\right\vert \leq k_{4}\right) 
\notag \\
&\leq &\frac{\varepsilon }{4n_{2}}n_{2}+2\frac{\varepsilon }{8}=\frac{%
\varepsilon }{2}  \notag
\end{eqnarray}%
where the last inequality used (\ref{EX_TailBound}) and (\ref%
{EX_CoverageBound}). Finally, set $k=\max \left( k_{1},k_{4}\right) $and
combine (\ref{EX_Bound1}) and (\ref{EX_Bound2}) to show that $E\left[ \psi
_{i,k}\left( \zeta \right) \right] \leq \varepsilon .$
\end{proof}

\begin{proof}[Proof of Example \protect\ref{Example_NeighborhoodNetwork}]
The proof follows a similar strategy as the proof for Example \ref%
{Example_Mixing}. The mean $\mu _{i,n}$ is given as 
\begin{equation*}
\mu _{i,n}=\sum_{j=1}^{n}E\left[ d_{ij}\right] =\sum_{j=i-1}^{i+1}E\left[
H\left( -\left\vert \zeta _{ij}\right\vert \right) 1\left\{ \left\vert \zeta
_{ij}\right\vert <1\right\} \right]
\end{equation*}%
because by the properties of the distribution of $\zeta $ it follows that $%
1\left\{ \left\vert \zeta _{ij}\right\vert <1\right\} =0$ for $j<i-1$ or $%
j>i+1.$ Similarly, 
\begin{eqnarray*}
E\left[ v_{i,n}\left( \zeta \right) |\mathcal{B}_{i,n}^{k}\right]
&=&\sum_{j=1}^{n}E\left[ d_{ij}|\mathcal{B}_{i,n}^{k}\right] \\
&=&\sum_{j=i-1}^{i+1}E\left[ H\left( -\left\vert \zeta _{i}{}_{j}\right\vert
\right) 1\left\{ \left\vert \zeta _{i}{}_{j}\right\vert <1\right\} |\mathcal{%
B}_{i,n}^{k}\right]
\end{eqnarray*}%
where the last equality follows from the fact that $v_{l,n}\left( \zeta
\right) =\sum_{k=1}^{n}d_{lk}$ does not depend on $\epsilon _{ij}$ for $%
l\neq i.$

Now only consider the case where $k$ is integer and $k>1.$ Note that $%
\mathcal{B}_{i,n}^{k}$ is the $\sigma $-field generated by $v_{l,n}\left(
\zeta \right) $ for $l<i-k$ or $l>i+k.$ Then, noting that $j\in \left\{
i-k,...,i+k\right\} ,$ $v_{l,n}$ does not depend on $\zeta _{ij}$ and $%
g_{il} $ also does not depend on $\zeta _{ij}$, while $g_{ij}=0$ for $j\in
\left\{ i-k,...,i+k\right\} .$ Thus, $\mathcal{B}_{i,n}^{k}$ does not depend
on $\zeta _{ij}$ as long as $k>1.$ It follows that 
\begin{equation*}
E\left[ d_{ij}|\mathcal{B}_{i,n}^{k}\right] =E\left[ d_{ij}\right]
\end{equation*}%
and consequently that 
\begin{equation*}
\mu _{i,n}-E\left[ v_{i,n}\left( \zeta \right) |\mathcal{B}_{i,n}^{k}\right]
=0.
\end{equation*}%
This implies that $v_{i,k}\left( \zeta \right) =0$ for $k>1.$
\end{proof}

\end{document}